\documentclass[a4paper,11pt]{article}

\usepackage[subsec]{modstyle}
\usepackage{microtype}

\usepackage[backend=biber,style=ext-alphabetic,articlein=false,sorting=nyt,maxbibnames=4]{biblatex}
\usepackage{hyperref}
\usepackage{alphanum}
\usepackage{appendix}

\let\oldabstract\abstract
\let\oldendabstract\endabstract
\makeatletter
\renewenvironment{abstract}
{%
               {\list{}{\addtolength{\leftmargin}{-1em}
                        \listparindent 1.5em%
                        \itemindent    \listparindent%
                        \rightmargin   \leftmargin%
                        \parsep        \z@ \@plus\p@}%
                \item\relax}%
               {\endlist}%
\oldabstract}
{\oldendabstract}
\makeatother

\title{Synthetic Equivariant Spectra for Finite Abelian Groups and Motivic Homotopy Theory}
\author{Keita Allen, Lucas Piessevaux}

\date{October 22, 2025}

\begin{document}
\maketitle
\begin{abstract}
 We prove a topological reconstruction result for the category of cellular $A$-equivariant motivic spectra over the complex numbers where $A$ is a finite abelian group: after completion at an arbitrary prime, this is equivalent to the completion of a category of synthetic $A$-equivariant spectra. The latter is a deformation of equivariant spectra which categorifies the equivariant perfect even filtration and is closely related to the equivariant Adams--Novikov spectral sequence. Our main computational input is a description of the bigraded homotopy groups of equivariant algebraic cobordism in terms of equivariant formal group laws.
\end{abstract}

\begingroup\tosfstyle
\let\oldcontentsline=\contentsline
\def\contentsline#1#2#3#4{\oldcontentsline{#1}{#2}{\texttosf{#3}}{#4}}
\tableofcontents
\endgroup
\section{Introduction}
%\subsection{Deformations and motivic homotopy theory}
Motivic homotopy theory was developed in order to apply techniques from algebraic topology to problems of algebro-geometric origin. However, over the years, there were insights which suggested that over algebraically closed fields of characteristic zero, a large portion of motivic homotopy theory is controlled by purely topological phenomena. Let $\SH(\bbC)^{\cell} \subset \SH(\bbC)$ the smallest subcategory of the $\bbC$-motivic stable homotopy category containing the motivic spheres $(\bbP^1)^{\otimes n}$ and closed under tensor products, desuspensions and colimits.
\begin{theorem*}[{\cite[Theorem 6.12]{gheorghe2022c}} at $p = 2$, {\cite[Theorem 1.4]{pstragowski_synthetic_2022}} at all primes]\label{synmotequiv}
    Let $p$ be an arbitrary prime, then there is an equivalence $$\SH(\bbC)^{\cell}_p \simeq \Syn_p,$$ where $\Syn$ is the category of synthetic spectra. 
\end{theorem*}
One can think of the category of synthetic spectra as a certain \emph{one-parameter deformation} of the category of spectra $\Sp$, encoding the reconstruction of a spectrum from its Adams-Novikov spectral sequence. In $\Syn$, we have bigraded spheres $S^{n, m}$ and a special map $$\tau: S^{0, -1} \to S^{0, 0}.$$ This map can be thought of as a deformation parameter; when we invert $\tau$, we recover the category of spectra, and when we kill $\tau$, we obtain an even variant of Hovey's derived category of stable $\MU_*\MU$ comodules, algebraic data which form the input to the Adams-Novikov spectral sequence. In this way, one can think of $\Syn$ as a ``categorification" of the Adams-Novikov spectral sequence, as illustrated in Diagram \cref{eq:categorif}.
\begin{equation}\label{eq:categorif}
    \begin{tikzcd}
	{\mathrm{Stable}^\even_{\MU_{\ast}\MU}} & \arrow[l, "\tau=0"']\Syn\arrow[r, "\tau^{-1}"] & \Sp\\
\end{tikzcd}
\end{equation}

From this perspective, the surprising fact that $\SH^{\cell}(\bbC)$ is equivalent to $\Syn$ after $p$-completion serves as a conceptual explanation for results like that of \cite{levine2015adams}, which tells us that the slice tower of the motivic sphere recovers the Adams-Novikov spectral sequence, or those from \cite{dugger2010motivic} and \cite{gheorghe2021special}, which tell us that inverting or killing $\tau$ in $\bbC$-motivic homotopy theory recover familiar objects from algebraic topology. This in turn has powered innovations such as those of \cite{isaksenstable23}, using $\bbC$-motivic homotopy theory to drastically advance the computation of the stable homotopy groups of spheres.

The main goal of this paper is to provide an extension of this theory to the equivariant setting, where the group of equivariance is a finite abelian group. In fact, our ambition is twofold: on one hand we obtain a certain reconstruction result for cellular equivariant motivic spectra over $\C$, while on the other hand we introduce a suitable category of synthetic equivariant spectra. The definition of the latter is in a sense tailor-made for the reconstruction result, but in \cref{sec:featuresofsynthetic} and \cref{sec: filtered reconstruction} we will prove several key properties about this synthetic category that indicate that it is the correct extension of synthetic spectra to the equivariant setting.

Let $A$ be a finite abelian group viewed as a constant group scheme over $\C$ and define the subcategory of cellular $A$-equivariant motivic spectra over $\C$ as the full subcategory
\[\SH^A(\C)^\cell\subset\SH^A(\C)\]
generated under colimits by objects of the form $\Th(V)\otimes (A/K)_+$ where $\Th(V)$ ranges over the motivic Thom spectra of virtual complex $A$-representations, and $A/K_+$ ranges over the orbits corresponding to subgroups $K$ of $A$. As in the nonequivariant setting, this really is a proper subcategory since schemes generically can not be equipped with (equivariant) cell structures, but it contains several key players of interest.
  \begin{introthm}[\ref{thm: synthetic reconstruction}, \ref{thm: generic fiber}, \ref{cor: modules over cofiber of tau}]\label{introthm:main}
    There is a category of \emph{synthetic $A$-spectra} $\Syn^A$, described purely topologically, such that 
    \begin{enumerate}
        \item \label{introthm:mainpart1} There is a functor
    \[\SH^A(\C)^\cell\to \Syn^A\]
    which induces an equivalence on subcategories of $p$-complete objects for an arbitrary prime $p$.
        \item \label{introthm:mainpart2} $\Syn^A$ categorifies the equivariant Adams-Novikov spectral sequence, in the sense that it fits into a diagram analogous to Diagram \cref{eq:categorif}.
    \end{enumerate}
\end{introthm}

In the remainder of the introduction, we will give an overview of some of our methods and results, and discuss some consequences.

\subsection{Global methods and a motivic refinement}
Nonequivariantly, the synthetic description of $\bbC$-motivic homotopy theory hinges largely on understanding the \emph{algebraic cobordism} spectrum $\MGL$; in many ways, this spectrum behaves similarly to $\MU$, and this close connection powers the connection between $\bbC$-motivic homotopy theory and topology. In the equivariant setting, we will exploit a similar close relationship between the spectrum of equivariant algebraic cobordism $\MGL_A$ introduced in \cite{khan_generalized_2022} and tom Dieck's spectrum of equivariant homotopical bordism $\MU_A$ from \cite{dieck1970bordism}. Contrary to the nonequivariant setting, the coefficients of $\MU_A$, and by extension $\MGL_A$, are extremely complicated and we do not have a general presentation for these rings. At abelian compact Lie groups, work of Comezaña in \cite[\S XXVIII]{comezana1996calculations} at least tells us that the coefficients of $\MU_A$ are concentrated in even degrees. In \cite{hausmann2022global}, the coefficients of abelian-equivariant homotopical bordism shown to possess a universal property related to the notion of equivariant formal groups laws introduced in \cite{cole2000equivariant}, thus extending the celebrated result of Quillen in \cite{quillen2007formal}.
\begin{theorem*}[{\cite[Theorem A]{hausmann2022global}}]
    For all compact abelian Lie groups $A$, $$\pi_*^A \MU_A \simeq L_A,$$ where $L_A$ is the \emph{$A$-equivariant Lazard ring}, the ring carrying the universal \emph{$A$-equivariant formal group law}.  
\end{theorem*}
Hausmann's work takes advantage of the fact that the $\MU_G$ for different compact Lie groups $G$ assemble into a \emph{global spectrum}, as in \cite{schwede2018global}. Global spectra are given an $\infty$-categorical description in \cite{linskens2022global}; as discussed in this work, a global spectrum $X$ in particular consists of the following data.
\begin{itemize}
    \item For every compact Lie group $G$, a $G$-spectrum $X_G$. 
    \item For every group homomorphism $f: G \to H$, a comparison map $$a_f: f^{*}X_H \to X_G,$$ where $f^*$ is the inflation-restriction functor corresponding to $f$, such that $a_f$ is an equivalence whenever $f$ is injective. 
\end{itemize}

Being one of the spectra $X_G$ in a global spectrum is quite restrictive; for example, the underlying spectrum with $G$-action of $X_G$ must have trivial action. On the other hand, this additional coherence can be very helpful and is essential in proving the universal property enjoyed by the rings $\pi^A_* \MU_A$. Hausmann (in \cite{hausmann2022global}) introduces the notion of \emph{global group laws}, objects which in particular gives rise to an equivariant formal group law at every abelian compact Lie group, but which are \emph{decompleted} in a certain sense, making them more amenable to algebraic operations and universal properties. The homotopy groups $\pi^A_* \MU_A$ for each $A$ assemble into a global group law, and this is initial in the category of global group laws.

We use a similar approach in order to understand the equivariant algebraic cobordism spectra\footnote{We use the notational convention that for a nice abelian group $A$ over $\C$, $\MGL_A$ refers to $\MGL_{\B A}$} $\MGL_A$ of \cite{khan_generalized_2022}. Slightly expanding upon work of Khan and Ravi in \cite{khan_generalized_2022}, we exhibit the collection of motivic spectra $\MGL_\cX$ over regular base stacks $\cX$ with a refinement to an absolute motivic spectrum over $\C$, which may be roughly described as follows.
\begin{construction*}[\cref{def: absolute motivic spectra}]
    An absolute motivic spectrum $E$ over $\C$ roughly consists of the following.
    \begin{itemize}
        \item For every regular stack $\cX$ over $\C$, an object $E_\cX\in \SH(\cX)$
        \item For every map $f\colon \cX\to \cY$ of regular stacks over $\C$, a comparison map
        \[\alpha_f\colon f^\ast E_\cY\to E_\cX\]
        such that $\alpha_f$ is an equivalence whenever $f$ is representable.
    \end{itemize}
\end{construction*}
Upon restriction to classifying stacks, an absolute motivic spectrum gives rise to a collection of equivariant motivic spectra which are compatible under restriction. We refer to \cref{warn: global caveats} for a comparison with absolute motivic spectra in the sense of \cite{khan_generalized_2022} and global spectra in the topological setting. 
\begin{introprop}[\ref{prop: MGL is absolute}]
    The algebraic cobordism spectra $\MGL_\cX$ lift to an absolute motivic spectrum. 
\end{introprop}
Using this simple observation, we are able to describe the coefficients of equivariant algebraic cobordism in terms of the equivariant Lazard rings, as summarised in the following omnibus.
\begin{introthm}[\ref{prop: ggl on MGL}, \ref{cor: homotopy groups of gfp of mgl}, \ref{cor : universal map of ggl's}, \ref{prop: vanishing in MGL}, \ref{prop: identification of the Chow line}, \ref{prop: mod p^i equivariant HKO}]\label{introthm:vanishing}
    The assignment sending an abelian compact Lie group $A$ to the bigraded ring $\pi^A_{\ast,\ast}\MGL$ defines a bigraded global group law satisfying the following.
    \begin{enumerate}
        \item For all $t-2w<0$ the group $\pi_{t,w}^A\MGL$ vanishes.
        \item It is regular: all Euler classes of surjective characters are regular elements.
        \item Its geometric fixed points can be described explicitly in terms of nonequivariant $\MGL$ and its formal group law.
        \item When $t - 2w = 0$, the coefficients recover the equivariant Lazard ring: there is an isomorphism of graded rings $$\pi_{2*}^A \MU \cong \pi_{2*, *}^A \MGL.$$
        \item After completion at an arbitrary prime, the bigraded coefficients are generated by the equivariant Lazard ring and a variable $\tau$ in bidegree $(0,-1)$: the isomorphism above extends to an isomorphism of bigraded rings $$(\pi_{2*}^A \MU)_{p}^{\wedge}[\tau] \cong (\pi^A_{*, *} \MGL)^\wedge_p.$$
    \end{enumerate}
\end{introthm}
\subsection{An equivariant Chow t-structure}
Having established the vanishing results in Theorem \ref{introthm:vanishing}, we can use it to obtain a number of structural results for $\SH^A(\bbC)^{\cell}$ quite easily. Inspired by \cite{bachmann2022chow} and \cite{haine2023spectral}, let the subcategory of \emph{perfect pure} $A$-equivariant motivic spectra $$\Pure_\C(A) \subset \SH^{A}(\bbC)$$ be the minimal subcategory closed under extensions and retracts and containing all equivariant motivic spectra of the form $\Th(V) \otimes A/K_+$ for all complex $A$-representations $V$ and subgroups $H$. Note that the localising subcategory generated by $\Pure_\C(A)$ is all of $\SH^{A}(\bbC)^{\cell}$. Using the technology of weight structures due to Bondarko (see \cite[\S 2.2]{elmanto2022nilpotent} for an introduction), we can easily prove the following.
\begin{introthm}[\ref{prop: Chow weight structure}]
    Let $\Pure(A;\MGL)$ denote the subcategory of $\MGL_{A}$-modules in $\SH^{A}(\C)^\cell$ consisting of free $\MGL_{A}$-modules on objects of $\Pure_\C(A)$. The inclusion induces an equivalence
    \[\cP_\Sigma(\Pure(A;\MGL);\Sp)\simeq \Mod(\SH^{A}(\bbC)^{\cell}; \MGL_{A}).\]
\end{introthm}

Using a close analysis of the cell structure on equivariant algebraic cobordism, we can bootstrap this $\MGL$-linear result up to the sphere spectrum.
\begin{introthm}[\ref{thm: Chow heart structure}]
    The inclusion of $\Pure_\C(A)$ induces an equivalence
    \[\Shv_\Sigma(\Pure(A);\Sp)\simeq \SH^{A}(\C)^\cell,\]
    where the left hand side consists of additive presheaves that send fibre sequences of perfect pure motivic spectra to fibre sequences.
\end{introthm}
Note the similarity with the nonequivariant results of \cite[\S 3]{haine2023spectral}. From this perspective, the desired algebraic special fibre arises as the derived category of the heart of the natural t-structure on this category of sheaves. This t-structure ought to be seen as an equivariant cellular version of Chow t-structure in \cite{bachmann2022chow}. 
\begin{introthm}[ \ref{cons: Chow t-structure}, \ref{cor: identification of the heart}, \ref{cor: identification of SF}]
    Let $\tau_{c\geq*}$ denote the covers in the natural t-stucture on $\SH^A(\C)^\cell$. Then we may identify the heart as
    \[\SH^{A}(\C)^\cell_{c=0}\simeq \coMod(\underline{\pi}^{A}_\star\MU^{\otimes 2}),\]
    i.e. as the abelian category of $\RU(A)$-graded comodules in $A$-Mackey functors over the Green Hopf algebroid coming from $\MU_{A}$ and its co-operations. Letting $\tau_{c=0}\mathbb{1}_{A}$ denote the Chow coconnective cover of the unit, we may further identify
    \[\Mod(\tau_{c=0}\mathbb{1}_{A})\simeq \Ind\cD^b_\Pure(\coMod(\underline{\pi}^{A}_\star\MU^{\otimes 2})).\]
\end{introthm}
We remark that it is not known whether the category $\Ind\cD^b_\Pure$ appearing in the statement is equivalent to the stable derived category of a certain moduli stack of $\RU(A)$-graded $A$-equivariant formal groups. We are now ready to compare to a homotopically defined synthetic category, which we define in terms of an appropriate notion of evenness given by complex representations.
\begin{construction*}[\ref{def: topological ppure}, \ref{def: equivariant synthetif}]
    The subcategory of \emph{perfect pure $A$-spectra} $$\Pure(A) \subset \Sp^A$$ is the smallest subcategory containing $S^V\otimes(A/K)_+$ for all virtual complex $A$-representations $V$ and subgroups $K$ and which is closed under extensions and retracts. 
    The category of \emph{synthetic $A$-spectra} $$\Syn^A = \Shv_{\Sigma}(\Pure(A); \Sp)$$ is the category of additive presheaves on $\Pure(A)$ which take fibre sequences in $\Pure(A)$ to fibre sequences of spectra.  
\end{construction*}
Since both categories of interest in \cref{introthm:main} appear as sheaf categories on more tractable subcategories, the problem is reduced to proving the desired equivalence for perfect pure objects on either side.
\begin{introthm}[\ref{cor: clp}, \ref{lem: common envelope}, \ref{thm: pure reconstruction}]
    The restricted equivariant Betti realisation functor 
    \[\Be^{A}\colon \Pure_\C(A)\to \Pure(A)\] satisfies the following.
    \begin{enumerate}
        \item It has the covering lifting property and admits a common envelope.
        \item It induces an equivalence on mod $p$ mapping spectra.
    \end{enumerate}
\end{introthm}
This result can be thought of as an equivariant lift of \cite[Theorem 7.30]{pstragowski_synthetic_2022}, itself a generalisation of \cite[Theorem 4.5]{gheorghe2017structure}. In fact, we use isotropy separation to reduce to the nonequivariant result. The motivic-synthetic comparison result, Theorem \ref{introthm:main} \ref{introthm:mainpart1}, now follows immediately.

\subsection{Relation to the perfect even filtration}
There are two categories, introduced at a similar time, which serve similar roles as categorifications of the Adams-Novikov spectral sequence providing topological reconstructions of $\bbC$-motivic stable homotopy theory: 
\begin{itemize}
    \item The category of even $\MU$-synthetic spectra, constructed in \cite{pstragowski_synthetic_2022}:
    $$\Syn^{\ev}_{\MU} = \Shv_{\Sigma}\left(\Sp_{\MU}^{\mathrm{fpe}}; \Sp\right),$$ where $\Sp_{\MU}^{\mathrm{fpe}}$ is the site of finite spectra with even, projective $\MU$-homology of \cite[Definition 5.8]{pstragowski_synthetic_2022}.
    \item The category of $\Gamma_* \mathbb{1}$-modules, constructed in \cite{gheorghe2022c}:
    \begin{align*}
        &\Mod(\Fil(\Sp);\Gamma_* \mathbb{1}), & \Gamma_* \mathbb{1}\simeq \Tot(\tau_{\geq 2*}\MU^{\otimes\bullet+1})\in \CAlg(\Fil(\Sp)),
    \end{align*}
    where $\Fil(\Sp) := \Fun(\bbZ^{\op}, \Sp)$ is the category of filtered spectra. 
\end{itemize}
Knowing that these two categories both recover the $p$-complete cellular motivic category tells us that they are equivalent after $p$-completion; in fact, one can show that they are equivalent integrally. There are different ways to approach this, but our work takes inspiration from the perspective of Pstr{\k a}gowski in \cite{pstrkagowski2023perfect}, where it is explained that both $\Syn^{\ev}_{\MU}$ and $\Gamma_* \mathbb{1}$ can be described without ever mentioning $\MU$ but only the notion of \emph{evenness}. 
\begin{itemize}
    \item Let $\mathrm{Perf}(S^0)_\even$ denote the site of \emph{perfect even $S^0$-modules} from \cite[Definition 2.2]{pstrkagowski2023perfect} (which coincides with $\Pure(\{e\})$ in our notation discussed above). Then \cite[Proposition 3.17]{pstrkagowski2023perfect} supplies an equivalence of sites
    \[\mathrm{Perf}(S^0)_\even\simeq\Sp^{\mathrm{fpe}}_\MU.\]
    \item Following \cite[Definition 2.21]{pstrkagowski2023perfect}, the site $\mathrm{Perf}(S^0)_\even$ equips any spectrum with a descending filtration, giving rise to a lax symmetric monoidal functor
    \[\fil_*^\even\colon\Sp\to \Fil(\Sp)\]
    which we call the \emph{perfect even filtration}.
    \item There is a symmetric monoidal equivalence
    \[\Syn\simeq\Shv_\Sigma(\mathrm{Perf}(S^0)_\even;\Sp)\simeq\Mod(\Fil(\Sp);\fil^\even_*\mathbb{1})\]
    \item By \cite[Theorem 7.5]{pstrkagowski2023perfect} and \cite[Corollary 2.2.21]{hahn2022motivic} there is an equivalence of commutative algebras in filtered spectra
    \[\fil^\even_*\mathbb{1}\simeq \Gamma_*\mathbb{1}\simeq\Tot(\tau_{\geq 2*}\MU^{\otimes\bullet+1}).\]
\end{itemize}
We are able to show that much of this story generalizes quite nicely to our equivariant setting. The connection between $\Syn^A$, the equivariant perfect even filtration, and equivariant Adams--Novikov descent may be summarised as follows.

\begin{introthm}[\ref{prop: filtered model}, \ref{thm: identification of even filtration on unit}]\label{introthm:filteredmodel}
    Let $\RU(A)$ denote the poset of isomorphism classes of virtual complex representations of $A$, ordered under inclusion of subrepresentations. Then there is a lax symmetric monoidal \emph{equivariant perfect even filtration} $$\fil_{\star}^{\ev}: \Sp^{A} \to \Fil_{\RU(A)} \Sp^A := \Fun(\RU(A)^\op,\Sp^A)$$ such that
    \begin{itemize}
        \item there is a monadic equivalence $$\Syn^{A} \simeq \Mod(\Fil_{\RU(A)} \Sp^A; \fil_{\star}^{\ev}(\mathbb{1}_{A}) ),$$ 
        \item and an explicit description $$\fil_\star^\ev(\mathbb{1}_{A}) \simeq \Tot \tau_{\geq 2\star} \MU_{A}^{\otimes \bullet + 1},$$ where $\tau_{\geq V} X := \Sigma^{V} \tau_{\geq 0}\Sigma^{-V} X.$
    \end{itemize}
\end{introthm}

We note that in our approach, we bypass any equivariant analogues of the site of finite spectra with even projective $\MU$-homology as well as any $\E_\infty$ even filtration\footnote{We therefore warn the reader that our notation $\fil^\even_\star$ refers to the \emph{perfect} even filtration.} as in \cite{hahn2022motivic}. 

The equivalence of $\Syn^A$ with modules over the $\RU(A)$-filtered even filtration follows from quite formal considerations, but has interesting consequences when connected to motivic homotopy theory. Informally, one can think of the map $\tau$ in the nonequivariant setting as corresponding to the one filtration direction in the filtered model for synthetic spectra. Analogously, the $\RU(A)$-filtered model for $\Syn^A$ encodes equivariant lifts of $\tau$ $$\tau_{V}: S^V \to \Th(V)$$ for any complex representation $V$ of $A$ (see \cref{ssec: eq lifts of tau} for notation). These maps are constructed formally in $\Syn^A$, but under the equivalence of \cref{introthm:main} \ref{introthm:mainpart1}, they give maps between spheres in $\SH^A(\bbC)^\cell_p$, many of which we know no clear geometric construction of. In particular, in contrast with the nonequivariant case, one did not have to know their existence prior to proving the equivalence. 

\subsection*{Overview}
\label{subsec:organization}
The body of this paper is split roughly into four blocks, which we will describe in more detail below:
\begin{enumerate}
    \item \label{block:1}In \cref{sec: equivariant SH,sec: absolute MGL}, we recall necessary aspects of equivariant motivic and absolute motivic homotopy theory, towards giving a description of equivariant algebraic cobordism $\MGL_A$ and structure on it. 
    \item \label{block:2}In \cref{sec: MGL and ggl} and \cref{sec: vanishing in MGL}, we move to a closer analysis of the homotopy groups of $\MGL_A$, taking advantage of the structure discussed previously. 
    \item \label{block:3}In \cref{sec: perfect pure}, \cref{sec: special fibre} and \cref{sec: synthetic reconstruction}, we levy our understanding of $\MGL_A$ to prove structural results about the cellular $A$-equivariant $\bbC$-motivic stable homotopy category $\SH^A(\bbC)^{\cell}$, culminating in the $p$-complete equivalence with our category of synthetic $A$-spectra $\Syn^A$.   
    \item \label{block:4}Finally, in \cref{sec:featuresofsynthetic} and \cref{sec: filtered reconstruction}, we move our attention from $\SH^A(\bbC)^{\cell}$ to our category $\Syn^A$, proving results internal to this topological category that connect it to the perfect even filtration and the Adams-Novikov spectral sequence.  
\end{enumerate}
We now give a linear overview the contents. In \cref{sec: equivariant SH}, we recall the setup of equivariant motivic homotopy theory and define absolute motivic spectra. \cref{sec: absolute MGL} is where we set up some properties of the equivariant algebraic cobordism spectra.
In \cref{sec: MGL and ggl}, we equip the former with a global group law and prove it is regular. This is the key input to \cref{sec: vanishing in MGL}, where we obtain our deescription of its homotopy groups.
In \cref{sec: perfect pure}, we introduce perfect pure objects and obtain structural descriptions of the cellular equivariant motivic categories, and \cref{sec: special fibre} is dedicated to describing the algebraic fibre .
In \cref{sec: synthetic reconstruction}, we briefly introduce the category of synthetic $A$-spectra and prove the main reconstruction theorem. \cref{sec:featuresofsynthetic} is dedicated to further study of the category of synthetic $A$-spectra.
Finally in \cref{sec: filtered reconstruction}, we exhibit the relation between synthetic $A$-spectra and the equivariant Adams--Novikov spectral sequence.
There are three technical appendices. In \cref{appendix: Betti realisation}, we construct an equivariant lift of the Betti realisation functor and an equivariant unit functor. In \cref{sec: ANSS convergence}, we show convergence of the equivariant motivic Adams-Novikov spectral sequence. In \cref{appendix: gfp2}, we provide some definitions and details on geometric fixed points at general nice abelian groups over $\C$ which form the technical input into the regularity result of \cref{sec: MGL and ggl}. 

\subsection*{Acknowledgements}
The authors wish to thank Robert Burklund, Sanath Devalapurkar, Jeremy Hahn, Jeremiah Heller, and Marc Hoyois for encouragement, suggestions and helpful discussions that have shaped the contents of this article. We are especially grateful to Markus Hausmann and William Balderrama for explaining the topological versions of the regularity resp. convergence arguments that play a key role in our proof as well as several enlightening conversations on the topic of equivariant homotopical bordism. We have further benefited from helpful conversations with Tom Bachmann, Emma Brink, Bastiaan Cnossen, Tobias Lenz, Klaus Mattis, Arjun Nigam, Levi Poon, Piotr Pstr\k{a}gowski, Natalie Stewart, and Noah Wisdom. We thank Florian Riedel and Markus Hausmann for proofreading and advice. The authors would like to thank the Isaac Newton Institute for Mathematical Sciences, Cambridge, for support and hospitality during the programme Equivariant homotopy theory in context, where work on this paper was undertaken. This work was supported by EPSRC grant EP/Z000580/1. The first author was supported by the NDSEG Fellowship. The second author was supported by the Hausdorff Center for Mathematics at the University of Bonn (DFG GZ 2047/1, project ID 390685813).
\newpage
\subsection*{Conventions}
\label{subsec:conventions}
\begin{itemize}
    \item Categories, unless stated otherwise, are $(\infty,1)$-categories and we follow the foundations set up in \cite{HTT, HA}.
    \item We use the term \emph{anima} (plural \emph{animæ}) for what is otherwise called a space, $\infty$-groupoid, homotopy type, etc. These form a category $\Ani$.
    \item Given an $\E_n$-algebra $R$ in a symmetric monoidal category $\cC$, the notation $\Mod(\cC;R)$ refers to the category of left $R$-modules in $\cC$. The unit of a symmetric monoidal category is denoted by some variation on the symbol $\mathbb{1}$. In particular, in the case of spectra this recovers $\mathbb{1}=\mathbb{S}=S^0$.
    \item The notations $\cP(\cC;\cD)=\Fun(\cC^\op,\cD)$ and $\Shv(\cC;\cD)$ denote the category of presheaves resp. sheaves on a category resp. site $\cC$ with values in $\cD$. Categories of additive (pre-)sheaves are decorated with a subscript $\Sigma$. If the coefficient category is not indicated, it is understood to be $\cD=\Ani$.
    \item Our generic notation for a cosimplicial object is $X^\bullet$. A $\bbZ$-graded/filtered object is generically denoted $X_\ast$, while a $\RU(A)$-graded/filtered object is denoted $X_\star$.
    \item Our main stacks of interest arise as classifying stacks of \emph{nice abelian groups}. Following \cite[\S 2.1]{khan_generalized_2022}, a \emph{nice} group is an fppf affine group scheme which arises as an extension of a finite étale group scheme of invertible order by a group of multiplicative type. In particular, we may view every nice abelian group over $\C$ as a produt of (split) tori and finite cyclic group schemes, the latter being constant. 
    \item Unless specified otherwise, the word \emph{stack} in the motivic context will be shorthand for \emph{nicely scalloped algebraic stack} cf. \cite[Definition 2.9]{khan_generalized_2022} over an implicit base.
    \item Given an abelian compact Lie group $A$, we denote the corresponding affine group scheme over $\C$ by $A_\C=\underline{\hom}(\underline{A}^\vee,\G_{m,\C})$ where $A^\vee$ denotes the Pontryagin dual, and when there is no potential for confusion we simply denote it by $A$ as well. In the case of a finite cyclic group, we see that $C_{n,\C}=\mu_n$. The notation $\T$ stands for the abelian compact Lie group which is a torus of rank \emph{one}, i.e. $S^1$ or $\mathrm{U}(1)$ while $\T_\C=\G_{m,\C}$ is the corresponding (split) torus of rank one over $\C$.
    \item When dealing with t-structures we use homological conventions.
    \item Given a compact Lie group $G$, the terms $G$-spectra and $G$-animæ always refer to the \emph{genuine} variants.
\end{itemize}
\newpage
\section{Equivariant motivic homotopy theory}\label{sec: equivariant SH}
In this section, we recall some key points from the theory of equivariant motivic homotopy theory as in \cite{hoyois_six_2017, khan_generalized_2022, gepner2023tom}. Throughout, we only have to explicitly deal with stacks of the form $[X/G]$ with $X$ separated.
\subsection{Recollections}
\begin{construction}
    Let $\cX$ be a stack and consider the category $\Smooth_\cX$ of stacks with a smooth map to $\cX$. This can be equipped with the Nisnevich topology as in \cite[\S 2.2]{khan_generalized_2022}. We then define the category of motivic spaces over $\cX$, denoted $\spcs({\cX})$, as the category of presheaves that satisfy Nisnevich descent and $\A^1$-homotopy invariance in the sense of \cite[Definition 3.2, Lemma 3.3]{khan_generalized_2022}.
\end{construction}
\begin{remark}\label{rem: motivic localisation functor}
    Following \cite[Lemma 3.4]{khan_generalized_2022}, the category $\spcs(\cX)$ of motivic spaces over $\cX$ is in fact an accessible left Bousfield localisation of $\cP(\Smooth_\cX)$ with localisation functor denoted
\[
    L_\mot\colon\cP(\Smooth_\cX)\to \spcs(\cX)
\]
and referred to as motivic localisation.
\end{remark}
We will frequently identify objects of $\Smooth_\cX$ with their motivically localised Yoneda image in $\spcs(\cX)$ using the same notation.
\begin{remark}
    When the base stack $\cX$ is of the form $\cX=[X/G]$ for a scheme $X$ with $G$-action, we will frequently denote $\spcs(\cX)$ by $\spcs^G(X)$ in accordance with the notation of \cite{hoyois_six_2017}. In this case, we will furthermore use the equivalence in \cite[\S A.3.4]{khan_generalized_2022} and work with the site of $X$-schemes with $G$-action appearing in \cite{hoyois_six_2017}.
\end{remark}
\begin{construction}\label{cons: Thom space}
    Given a finite locally free sheaf $\cE$ on $\cX$, we define its Thom space $\Th_\cX(\cE)\in \spcs(\cX)_\ast$ as sitting in the cofibre sequence
    \[(\mathbb{V}_\cX(\cE)\setminus \cX)_+\to \mathbb{V}_\cX(\cE)_+\to \Th_\cX(\cE)\]
    in $\spcs(\cX)_\ast$, where $\mathbb{V}_\cX(\cE)$ is the total space of $\cE$ viewed as an object of $\Smooth_\cX$ and we view $\cX$ as sitting inside of it by the zero section.
\end{construction}
We will frequently omit the base stack $\cX$ from the notation $\Th_\cX(\cE)$ when there is no risk for confusion. Further, note that this construction takes extensions of finite locally free sheaves to tensor products following \cite[Remark 4.4]{khan_generalized_2022} and is in particular symmetric monoidal.
\begin{defn}\label{def: Euler class}
    Since $\mathbb{V}_\cX(\cE)$ is $\A^1$-homotopy equivalent to $\cX$, we see that the inclusion of the zero section gives us a map
    \[\mathbb{1}_\cX\to \Th_\cX(\cE)\]
    from the unit in $\spcs(\cX)_\ast$. We will refer to this map as the pre-Euler class of $\cE$ and denote it by $a_\cE$.
\end{defn}
\begin{remark}\label{rem: purity cof seq for tori}
    Consider a nice abelian group $A$ with a surjective character $\alpha\colon A\to \T_\C$ giving rise to a similarly named finite locally free sheaf on $\B A$. The cofibre sequence defining $\Th(\alpha)$ takes the form
    \[A/K_+\to \mathbb{1}_{\B A}\xrightarrow{a_\alpha}\Th(\alpha)\]
    where $K\subset A$ is the kernel of $\alpha$. Indeed, since $\alpha$ is one-dimensional we see that $\mathbb{V}_{\B A}(\alpha)\setminus \B A$ can be identified with $\T_\C=\G_m$ where the action of $A$ is provided by the map $\alpha$, since $\alpha$ was surjective we may identify this with $A/K$.
\end{remark}
\begin{construction}
    Let $\cX$ be a stack, then we define the category of motivic spectra $\SH(\cX)$ over $\cX$ by tensor-inverting all Thom spaces $\Th_\cX(\cE)$ in $\spcs(\cX)_\ast$. 
\end{construction}
\begin{remark}
    Suppose $\cX$ lives over $\B G$ for a nice group $G$, then by \cite[Corollary 6.7]{hoyois_six_2017} it suffices to invert the pullbacks of Thom objects of $G$-representations: there is a symmetric monoidal equivalence $\SH(\cX)\simeq \spcs(\cX)_\ast\otimes_{\spcs(\B G)}\SH(\B G)$. Furthermore, by \cite[Lemma 6.3]{hoyois_six_2017} such objects are $3$-symmetric so that this inversion process can be modeled as a filtered colimit in $\PrL$.
\end{remark}
The assignment $\cX\mapsto \SH(\cX)$ comes with excellent functoriality properties, of which we now recall the salient points. Throughout, we will not explicitly make use of the shriek functors, so they will not be mentioned.
\begin{thm}[{\cite[Theorems 4.5, 4.10]{khan_generalized_2022}}]\label{thm: 6FF omnibus}
    The motivic six functor formalism $\SH$ on qcqs algebraic spaces extends to stacks, with the following functoriality.
    \begin{enumerate}
        \item\label{omni: presentable} Every $\SH(\cX)$ is a presentable stable category admitting a colimit-preserving symmetric monoidal structure.
        \item\label{omni: pullbacks} For every map $f\colon \cX\to \cY$ there exists an adjoint pair
        \[f^\ast\colon \SH(\cY)\rightleftarrows \SH(\cX)\colon f_\ast\]
        where the left adjoint is symmetric monoidal.
        \item The pullback functoriality makes $\SH$ into a Nisnevich sheaf with values in $\CAlg(\PrL_\mathrm{st})$
        \item For every smooth map $f\colon \cX\to \cY$, the pullback $f^\ast$ admits a further left adjoint $f_\sharp$ which satisfies the following properties\footnote{The abbreviations SBC and SPF stand for \emph{Smooth base change} and \emph{Smooth projection formula} respectively}:
        \begin{itemize}
    \item[(SBC)]\label{omni: sbc} it commutes with arbitrary $\ast$-pullbacks, and
    \item[(SPF)]\label{omni: spf} it is $f^\ast$-linear.
\end{itemize}
    \item For every finite locally free sheaf $\cE$ on $\cX$, the suspension spectrum of the Thom space $\Th(\cE)$ in $\SH(\cX)$ is $\otimes$-invertible\footnote{We will therefore denote the endofunctor $\Th(\cE)\otimes -$ by $\Sigma^\cE$ and adopt the notation $\Th(-\cE)$ for the tensor-inverse of $\Th(\cE)$.}. The construction of Thom spectra is furthermore compatible with arbitrary pullbacks.
    \item For every smooth map $f\colon \cX\to \cY$, let $\Omega_f$ denote its cotangent complex viewed as a finite locally free sheaf on $\cX$, then there is a natural transformation
    \[f_\sharp\Sigma^{-\Omega_f}\to f_\ast\]
    which is an equivalence if $f$ was additionally proper.
    \end{enumerate}    
\end{thm}
\begin{remark}
    Since Thom spaces are now invertible, we will frequently abuse notation by identifying the pre-Euler class of a finite locally free sheaf $\cE$ on a stack $\cX$ with its desuspension, i.e. we view it as a map $a_\cE\colon \Th_\cX(-\cE)\to \mathbb{1}_\cX$.
\end{remark}
\begin{remark}
    Following standard notation in motivic homotopy theory, every category $\SH(\cX)$ obtains a family of bigraded spheres $S^{t,w}=\Sigma^{t-2w}(\P^1_{\cX})^{\otimes w}$, where $\P^1=\Th_{\cX}(\cO_{\cX})$ can be viewed as the Thom space of the trivial vector bundle of rank one. In general, these bigraded spheres are therefore in the image of the pullback functor along the structure map $\cX\to B$ to an implicit base scheme.
\end{remark}
The functoriality described above allows us to perform many familiar operations from equivariant homotopy theory in the motivic setting. We will fix a base $\Spec(\C)$ below and let $A$ vary over nice abelian groups. 
\begin{construction}
   Let $i\colon K\subset A$ be the inclusion of a subgroup, so that the corresponding map $i\colon \B K\to \B A$ is smooth. We adopt the following nomenclature:
    \begin{itemize}
        \item $i^\ast$ is the restriction functor,
        \item $i_\sharp$ is the induction functor,  
        \item $i_\ast$ is the coinduction functor,
        \item $\Omega_i$ is the adjoint representation,
        \item $i_\sharp\Sigma^{-\Omega_i}\to i_\ast$ is the Wirthmüller transformation.
    \end{itemize}
If $i\colon K\to A$ is such that $A/K$ is furthermore proper (which in our case is only true when it is an inclusion of codimension zero), then the Wirthmüller transformation becomes the Wirthmüller isomorphism.
In the non-representable setting, consider the case where $p\colon \B A\to \Spec(\C)$ is the terminal morphism, then we adopt the following nomenclature.
\begin{itemize}
    \item $p^\ast$ is the inflation functor,
    \item $p_\ast$ is the fixed point functor, also denoted $(-)^A$,
    \item given $E\in \SH^A(\C)$, its $A$-equivariant homotopy groups are defined by
    \[\pi^A_{t,w}E=\pi_0\map(\Sigma^{t,w}\mathbb{1}_{\B A},E)\simeq \pi_0\map(\Sigma^{t,w}\mathbb{1}_\C,p_\ast E).\]
\end{itemize} 
\end{construction}
\begin{notation}
    Given a nice abelian group $A$, the unit in $\SH(\B A)$ will be denoted $\mathbb{1}_A$ when there is no potential for confusion.
\end{notation}
\subsection{Isotropy separation}
An essential tool in equivariant motivic homotopy theory is that of isotropy separation. This is discussed in the stacky language in \cite[\S 4]{bachmann2022motivic} and in the case of finite group actions in \cite[\S 3]{gepner2023tom}. We follow the latter and discuss some of the key technical features that we will need; we will be mainly concerned with the case of a finite abelian group $A$ acting trivially on the base scheme $\C$. If $\cF$ is a family of subgroups of $A$, we adopt the notation $\Smooth_\C^A[\cF]$ for the subcategory of $\Smooth_\C^A$ with isotropy contained in $\cF$. If $\cF$ is the maximal family of all subgroups this recovers $\Smooth_\C^A$.
\begin{defn}\label{def : universal motivic spaces}
    Let $\cF$ be a family of subgroups of $A$. Then the universal motivic space associated to $\cF$ is denoted $\E\cF$ and defined as a presheaf on $\Smooth_\C^A$ by
    \[\E\cF(X)=\begin{cases}
        \ast & X\in \Smooth_\C^A[\cF],\\
        \varnothing & \text{else}.
    \end{cases}\]
\end{defn}
Following \cite[Proposition 3.3]{gepner2023tom} we see that this defines an object of $\spcs(A)$. When $\cF$ is the trivial family we denote the corresponding motivic space by $\E A$, so that $\E A$ is the colimit of all free $A$-schemes. One can also define a relative version: given any inclusion of families $\cF\to \cF'$ define a motivic $A$-space $\E(\cF',\cF)$ as the cofibre
\[\E\cF_+\to \E\cF'_+\to \E(\cF',\cF)\]
of the natural comparison map. When $\cF'$ is the family of all subgroups we will also denote the cofibre by $\widetilde{\E}\cF$ in keeping with equivariant conventions. 
An essential observation, going back to \cite{morel19991}, is that equivariant classifying spaces in motivic homotopy theory admit geometric descriptions in nice cases. 
\begin{example}\label{ex: model for total space}
    Given a family $\cF$ of subgroups of $A$, we let $\E^n\cF$ denote the inductive system in smooth $\A$-schemes given by $\mathbb{V}(n\rho)\setminus \bigcup_{K\in \mathrm{co}(\cF)}\mathbb{V}(n\rho)^K$ where $\rho$ is the regular representation of $A$ and $K$ ranges over the complement of $\cF$ in the maximal family. Then there is an equivalence
    \[\varinjlim_n\E^n\cF\simeq \E\cF.\]
\end{example}
\begin{example}\label{ex: model for classifying space}
    In the case of the trivial family, since $\E^nA$ is an $A$-scheme with free action, the quotient $\E^nA/A$ exists as a smooth scheme and we denote it by $\B^n_\mot A$ such that we obtain an equivalence
    \[\varinjlim_n \B^n_\mot A\simeq \E A/A=\B_\mot A\]
    where the right hand side the geometric classifying space of \cite{morel19991} which is equivalent as a motivic space to the étale classifying space of $A$.
\end{example}
\begin{construction}\label{cons: horbits}
    Given any $X\in \SH^A(\C)$, the tensor product $\E A_+\otimes X$ is now a free $A$-equivariant motivic spectrum, so that one can apply the quotient functor of \cite[\S 5]{gepner2023tom} and obtain the homotopy orbits
    \[X_{\mathrm{h} A}\coloneq (\E A_+\otimes X)/A\]
    of \cite[Definition 7.3]{gepner2023tom}.
\end{construction}
\begin{construction}\label{cons: gfp}
Given a subgroup $A$ of $K$, consider the family $\cF[K]$ of subgroups of $A$ that do not contain $K$. Following \cite[\S 4.3]{gepner2023tom} we define a $K$-geometric fixed points functor
\[\Phi^K\colon \SH^A(\C)\to \SH^{A/K}(\C)\]
given by tensoring with $\widetilde{\E}\cF[K]$ and using the equivalence in \cite[Proposition 4.8]{gepner2023tom}. In the special case where $K=A$, $\cF[A]$ is the family of proper subgroups and we recover the usual notion of geometric fixed points. In \cref{appendix: gfp2} this construction is extended to allow for general nice abelian groups and indexing universes.
\end{construction}
\begin{defn}\label{def: adjacent families}
    Let $\cF_i\subset \cF_{i+1}$ be an inclusion of families of subgroups of $A$, we say that the families are adjacent at some subgroup $K$ if the complement consists of the singleton $\{K\}$. We adopt the notation $\E[K]=\E(\cF_{i+1},\cF_i)$.
\end{defn}
\begin{remark}
    Since $A$ was assumed to be a finite group, there is always a finite filtration of the form
    \[\varnothing=\cF_{-1}\subset \cF_0\subset\cdots\subset \cF_{n-1}\subset \cF_n=\mathrm{Sub}(A)\]
    terminating at the family of all subgroups such that every pair of the form $\cF_i\subset \cF_{i+1}$ is adjacent. Indeed, one can just take a maximal chain of families of subgroups. 
\end{remark}
\begin{remark}\label{rem: adjacent isotropy filtration}
    As a consequence, any object $X$ of $\SH^A(\C)$ admits a finite filtration of the form $\{(\E\cF_i)_+\otimes X\}_i$ with associated graded $\{\E[K]\otimes X\}_K$. By \cite[Proposition 4.12, Theorem 6.33]{gepner2023tom} we may identify the resulting filtration on the fixed points $X^{A}$. Its associated graded is given by $$\{(\E[K]\otimes X)^A\}_K=\{(\Phi^K X)_{\mathrm{h}A/K}\}_{K}$$ as $K$ ranges over all subgroups of $A$.
 \end{remark}   
\begin{ex}\label{ex: horbits of Thom spectrum}
    Let $\cE$ be a vector bundle over $\B A$, i.e. a finite dimensional complex $A$-representation. In order to idenfity the associated graded of the filtration on $\Th(\cE)$, we make the following observations.
    \begin{enumerate}
        \item Let $\{\E^i A\}_i$ be an Ind-scheme model for $\E A$ as in \cref{ex: model for total space}, with structure maps $p_i\colon \E^i A\to \B A$. Then the smooth projection formula allows us to identify
        \[\Th(\cE)\otimes \E A_+\simeq \varinjlim_i(p_i)_\sharp\Th_{\E^i A}(p_i^\ast \cE).\]
        We now wish to compute its quotient (in the sense of \cite[\S 5]{gepner2023tom}). Per construction, note that every $\E^i A$ is a free $A$-scheme so that it has a quotient scheme $q_i\colon\B^i_\mot A\to \Spec(\C)$ which presents an Ind-scheme model for the étale classifying space of $A$ by \cref{ex: model for classifying space}. Furthermore, the pullbacks $p_i^\ast \cE$ over the $\E^iA$ descend to quotient bundles $p^i_\ast \cE/A$ on every $\B^i_\mot A$. The homotopy orbits of $\Th(\cE)$ can then be identified more explicitly as
        \[\Th(\cE)_{\mathrm{h}A}=(\Th(\cE)\otimes \E A_+)/A\simeq \varinjlim_i (q_i)_\sharp\Th_{\B^i_\mot A}(p_i^\ast\cE/A).\]
        In more convenient notation, we may write this as
        \[\Th(\cE)_{\mathrm{h}A}\simeq \Th_{\B_\mot A}(\widetilde{\cE})\]
        where $\cE$ is the \emph{associated bundle} over $\B_\mot \mu_n$.
        \item $\Phi^K\Th(\cE)$ can be identified with the Thom spectrum of $\cE^K$ viewed as an $A/K$-representation, cf. \cite[Proposition 4.10]{gepner2023tom}.
    \end{enumerate}
    We conclude that the filtration on $\Th(\cE)^A$ has associated graded given by the collection $\{\Th_{\B_\mot A/K}(\widetilde{\cE^K})\}_K$ of Thom spectra over étale classifying spaces as $K$ ranges over subgroups of $A$.
\end{ex}
\subsection{Absolute motivic spectra}
Let us now construct an analogue of global spectra in the equivariant motivic setting by slightly expanding on \cite[Definition 9.3]{khan_generalized_2022}. We warn the reader that these objects are not truly \emph{global} and have more structure than \emph{absolute}, since they are arguably closed to the latter in definition, we have opted to use the terminology \emph{absolute}. See \cref{warn: global caveats} below for details.
\begin{defn}\label{def: absolute motivic spectra}
    Let $\Stk_\cB$ denote the category of quasi-fundamental stacks over some fixed base $\cB$, and let $\Stk^\dagger_{\cB}$ denote its marking at representable morphisms, then the category of absolute motivic spectra over $\cX$ is the partially lax limit
    \[\SH^{\mathrm{abs}}(\cB)=\plaxlim{\Stk_\cB^{\op,\dagger}}{\SH(\bullet)}\]
\end{defn}
\begin{remark}
    It is clear that the functor $\SH\colon \Stk_B\to \Cat$ considered above lifts to $\CAlg(\Prlst)$ and that $\SH^{\mathrm{abs}}(\cB)$ therefore lifts to a stable presentably symmetric monoidal category as well, cf. the general discussion in \cite[\S 5]{linskens2022global}.
\end{remark}
\begin{warning}\label{warn: regularity hypothesis nomenclature}
    One can modify the test category $\Stk_\cB$ as well, e.g. by considering only \emph{regular} stacks over $\cB$, marked at representable morphisms. Since our main example --algebraic cobordism-- only lifts to the partially lax limit over this regular test category, we will also refer to these objects as \emph{absolute} with the regular test category being implicit. 
\end{warning}
More explicitly, we want to think of an object $E$ of $\SH^{\abs}(\cB)$ as a compatible collection of objects $E_\cX\in \SH(\cX)$ as $\cX$ ranges over quasi-fundamental stacks over $\cB$, together with a coherent collection of comparison maps
\[\alpha_f\colon f^\ast E_\cY\to E_\cX\]
for all maps $f\colon \cX\to \cY$ in $\Stk_\cB$ such that $\alpha_f$ is an equivalence whenever $f$ is representable. We are particularly interested in the case where $\cB$ is a fixed base scheme such as $\Spec(\C)$ and the stacks in $\Stk_\C$ are of the form $\B G=[\Spec(\C)/G]$ for nice groups $G$. Note that any map of group schemes $H\to G$ gives rise to a map in $\Stk_{\C}$ of the form $\B H\to \B G$ which is representable if the map $H\to G$ is a monomorphism.
\begin{remark}\label{rem: absolute homotopy groups functorial}
    Let $E\in \SH^{\abs}(\C)$ be an absolute motivic spectrum over $\C$, then the equivariant homotopy groups $\pi^G_{\ast,\ast}E$ of $E$ are fully contravariantly functorial in arbitrary group homomorphisms. Indeed, given an arbitrary map of nice groups $H\to G$ over $\C$ we obtain a commutative diagram
    \[\begin{tikzcd}
    \B H\arrow[r, "f"]\arrow[d, "q"]&\B G\arrow[d, "p"]\\
    \Spec(\C)\arrow[r, "\id"]&\Spec(\C)
    \end{tikzcd}\]
    giving rise to a composite map
    \[p_\ast E_{G}\to q_\ast f^\ast E_G\xrightarrow{q_\ast\alpha_f}q_\ast E_H\]
    where the first map is the Beck--Chevalley transformation. Note that if $H\to G$ was a monomorphism so that $\alpha_f$ is an equivalence, then this is nothing but the usual restriction functoriality in injective group homomorphisms.
\end{remark}
In global homotopy theory, one can show that for any global spectrum $E$ and any group $G$ its underlying spectrum with $G$-action (i.e. the Borel completion of $E_G$) must have a trivial action. This essentially follows from the observation that the $G$-action comes from the global functoriality and that this can already be trivialised in the indexing category (see \cref{warn: global caveats}). We will now prove an analogue of this in the setting of absolute motivic spectra.
\begin{lemma}\label{lem: global action is trivial}
    Let $A$ be a finite group acting trivially on a base scheme $S$ with structure map $p\colon [S/A]\to S$, and let $E$ be an absolute motivic spectrum over $S$. Then the comparison map
    \[\alpha_p\colon p^\ast E\to E_A\] is an $\E A_+$-equivalence.
\end{lemma}
\begin{proof}
    Recall from \cref{def : universal motivic spaces} and \cite[Proposition 3.7]{gepner2023tom} that $\E A_+$ can be obtained as the colimit of a filtered diagram of smooth free $A$-schemes. It therefore suffices to prove that for any smooth free $A$-scheme $X$ with structure map $q\colon [X/A] \to [S/A]$ the map
    \[\alpha_p\otimes [X/A]_+\colon p^\ast E\otimes [X/A]_+\to E_A\otimes[X/A]_+\]
    is an equivalence. By the smooth projection formula and coherence of the absolute structure maps this can be identified with $q_\sharp q^\ast\alpha_p\simeq q_\sharp \alpha_{pq}$ but by the freeness assumption $pq$ is representable (see \cite[Tag 07S7]{stacks}) so that $\alpha_{pq}$ is an equivalence.
\end{proof}
\begin{warning}\label{warn: global caveats}
    While the definition of absolute motivic spectra above is enough for our purposes, we point out to the reader that it contains some deficits and subtleties.
    \begin{enumerate}
        \item The pre-existing notion of \emph{absolute} motivic spectrum in \cite[Definition 9.3]{khan_generalized_2022} is slightly weaker than ours: Khan--Ravi only ask for equivalences $\alpha_f$ along representable maps $f$. In practice, constructing the maps $\alpha_f$ in full generality is not hard and the true content is in verifying that these are equivalences when $f$ is representable. We have therefore opted to re-use their terminology.
        \item In the homotopical setting, there is a notion of global spectra, which by \cite{linskens2022global} can be realised as a partially lax limit of the functor that sends a compact Lie group $G$ to the category of $G$-spectra. However, this limit is not taken over the $1$-category of compact Lie groups but a higher categorical enhancement denoted $\Glo$ in \cite{linskens2022global} whose homotopy category recovers the $1$-category of compact Lie groups. 
        \end{enumerate}
In the case of finite groups, the corresponding global indexing category $\Glo^\mathrm{fin}$ can be realised as a full subcategory of topological stacks (\cite[\S 2]{gepner2023equivariant}) or differentiable stacks (\cite[Corollary 2.17]{clough2024global}) and this will be exploited in \cref{cons: glo to stk comparison}. On the other hand, $\Glo$ itself is not truncated as a category, and by \cite[Propositions 2.29, 2.30]{clough2024global} it instead arises as a full subcategory of \emph{homotopy invariant sheaves} on differentiable stacks. The main difference between global spectra in the homotopical setting and absolute spectra in the motivic setting is therefore that the indexing category for the partially lax limit is highly non-truncated in the homotopical setting while it is a mere $(2,1)$-category in our motivic setting. Furthermore, in the homotopical setting one is only interested in functoriality as the group changes, i.e. over classifying stacks of compact Lie groups. On the other hand, an absolute motivic spectrum is functorial over a large supply of schematic morphisms as well. The latter allows us to prove \cref{lem: global action is trivial} while in the homotopical setting the higher homotopies trivialising this action are already found in the indexing category, see \cite{schwede_localizations} for a proof of this folklore result. 
\end{warning}
The remarks above lead us to the following natural question inspired by \cite{clough2024global} and \cite[Proposition 6.13]{linskens2022global}.
\begin{question}
Can one define a stacky global indexing category over $\C$ satisfying the following?
\begin{itemize}
    \item It admits a description as a full subcategory of $\A^1$-invariant Nisnevich sheaves on quasi-fundamental stacks.
    \item It contains all (quasi--fundamental) Deligne--Mumford stacks as a full subcategory.
    \item The functor $\SH$ on quasi-fundamental stacks admits a unique extension to this stacky global indexing category.
    \item It admits a factorisation system such that functorality of $\SH$ in the right class is automatic.
    \item It contains sufficiently many higher homotopies to prove an analogue of \cref{lem: global action is trivial} in full generality for any object in the partially lax limit of $\SH$ over this indexing category.
\end{itemize}
\end{question}

\newpage
\section{Equivariant algebraic cobordism}\label{sec: absolute MGL}
We now recall the construction of the motivic spectrum $\MGL_\cX\in \SH(\cX)$ and its absolute functoriality, as well as its cell structure in terms of Gra{\ss}mannians.
\subsection{Motivic Thom spectra}
Since $\MGL_\cX\in \SH(\cX)$ will be constructed as the Thom spectrum of an equivariant motivic analogue of the $J$-homomorphism, we will first recall its construction following \cite[\S 16]{bachmann2021norms}.
\begin{defn}\label{def: motivic Thom spectrum}
    Let $\cX$ be a quasi-fundamental stack and consider the functor
    \[\SH\colon\Smooth_\cX^\op\to \Cat\]
    given by $\SH$ with its pullback functoriality, whose unstraightening will be denoted $(\Smooth_\cX)_{\sslash \SH}$. Let
    \[\mathrm{M}_\cX\colon \cP((\Smooth_\cX)_{\sslash\SH})\to \SH(\cX)\]
    be motivic Thom spectrum functor over $\cX$, defined as the left Kan extension of the functor that sends a pair $(f\colon \cY\to \cX, E_\cY)$ in $(\Smooth_\cX)_{\sslash\SH}$ to $f_\sharp E_\cY$.
\end{defn}
\begin{remark}
    A more coherent and precise definition of this functor is given in \cite[\S 16.3]{bachmann2021norms} and we simply note that the construction goes through in our setting as well.
\end{remark}
Note that the source $\cP((\Smooth_\cX)_{\sslash \SH})$ contains a full subcategory equivalent to $\cP(\Smooth_\cX)_{\sslash \SH}$ and this further contains a wide subcategory $\cP(\Smooth_\cX)_{/\SH}$ given by the usual slice category. In the remainder of this section, we will only make use of the latter subcategory and denote the restriction of $\mathrm{M}_\cX$ using the same notation.
\begin{remark}[{\cite[Remark 16.5]{bachmann2021norms}}]\label{rem: colimit formula}
Given an object $\phi\colon \cF\to \SH$ in $\cP(\Smooth_\cX)_{/\SH}$ we can explicitly describe $\mathrm{M}_\cX(\phi)$ by the formula
\[\mathrm{M}_{\cX}(\phi)\simeq \varinjlim_{(f,\alpha)}f_\sharp\phi_\cY(\alpha),\]
where the colimit is indexed over object $(f\colon \cY\to \cX,\alpha)$ of $(\Smooth_{\cX})_{\sslash \cF}.$
\end{remark}
In the same vein, one  can shows that the motivic Thom spectrum functor is compatible with pullbacks (cf. \cite[Lemma 16.7]{bachmann2021norms}) as a consequence of the smooth projection formula recalled in \cref{thm: 6FF omnibus}.
\begin{defn}[{\cite[Remark 4.12]{khan_generalized_2022}}]\label{def: stacky j homomorphism}
Let $\cK_\cX\in \cP(\Smooth_\cX)$ denote the presheaf sending a smooth stack over $\cX$ to its $\mathrm{K}$-theory anima, and let 
\[j_\cX\colon \cK_\cX\to \SH\]
denote the natural transformation sending a virtual vector bundle to its Thom spectrum. Letting $\cK_\cX^\circ\subset \cK_\cX$ denote the subanima on virtual bundles of rank zero, we will denote the corresponding restriction by $j_\cX^\circ$.
\end{defn}
\begin{remark}
    As in loc. cit. the natural transformation is constructed from the usual Thom spectrum functor $\Th_\cX\colon \mathrm{Vect}(\cX)\to \Pic(\SH(\cX))$ which is a map of $\E_\infty$-monoids landing in the Picard anima per construction. We then observe that the target is a grouplike Zariski sheaf to extend it to $\cK(\cX)$ and note that this construction is compatible with pullbacks to extend it to a natural transformation of presheaves over $\Smooth_{\cX}$.
\end{remark}
\begin{defn}[{\cite[Construction 10.9]{khan_generalized_2022}}]\label{def: MGL}
    Given a quasi-fundamental stack $\cX$, we define the motivic spectrum of algebraic cobordism over $\cX$ by
    \[\MGL_\cX= \mathrm{M}_\cX(j_\cX^\circ)\in \SH(\cX),\]
    while its periodic variant is defined by
    \[\MGLP_\cX= \mathrm{M}_\cX(j_\cX)\in \SH(\cX).\]
\end{defn}
\begin{remark}\label{rem: colimit formula for MGL}
    Applying the formula in \cref{rem: colimit formula}, we immediately obtain an equivalence
\[\MGL_\cX\simeq \varinjlim_{(f,\xi)}f_\sharp \Th_\cY(\xi)\]
where the colimit ranges over pairs $(f\colon \cY\to \cX,\xi\in \cK^\circ(\cY))$.
\end{remark}
\begin{remark}\label{rem: motivic Thom iso}
Again following the general arguments provided in \cite[Proposition 16.28, Example 16.30]{bachmann2021norms} and the notation from \cref{rem: colimit formula for MGL} we see that $\MGL_\cX$ is oriented in the sense that given any $g\colon \cZ\to \cX$ smooth over $\cX$ and $\zeta\in \cK^\circ(\cZ)$, there is an equivalence
\[\MGL_\cX\otimes \Th_\cZ(\zeta)\simeq \varinjlim_{(f,\xi)}(f\times g)_\sharp\Th_{\cY\times \cZ}(\xi\boxtimes\zeta)\simeq \MGL_{\cX}\otimes \cZ_+\]
just by rearranging the colimit.
\end{remark}
\subsection{An absolute refinement}
We now establish a slight refinement of \cite[Proposition 10.10]{khan_generalized_2022}, which essentially states that the spectra $\MGL_\cX$ lift to an absolute motivic spectrum when working over regular global quotient stacks (see \cref{warn: regularity hypothesis nomenclature}, \cref{rem: regularity assumption}). As in loc. cit., the proof passes through a similar result for presheaves $\cK_\cX^\circ$ and pullback functoriality of the motivic Thom spectrum functors. We begin with a simple observation.
\begin{lemma}
    Let $\Stk_\cB$ denote the category from \cref{def: absolute motivic spectra} and $\cK_\cX^\circ$ the presheaf from \cref{def: stacky j homomorphism} for $\cX\in \Stk_\cB$. Then this lifts to an object
    \[\cK^\circ\in \laxlim{\cX\in\Stk_\cB}\cP(\Smooth_{\cX})\]
\end{lemma}
\begin{proof}
    For any $n\geq 0$ the group $\GL_n$ in $\cP(\Smooth_\cX)$ is pulled back from $\cP(\Smooth_\cB)$. Therefore, there are tautological comparison maps
    \[\alpha_f\colon f^\ast\B_{\mathrm{fppf}}\GL_{n,\cY}\to \B_{\mathrm{fppf}}\GL_{n,\cX}\]
    for any morphisms $f\colon \cX\to \cY$ in $\Stk_\cB$. Assembling the moduli of rank $n$ bundles for all $n$ and group completing levelwise, we obtain the desired result for $\cK^\circ$ as the maps above are precisely enough to define an object in the lax limit.
\end{proof}
After motivic localisation we can say something stronger, namely that this lands in the partially lax limit. Note that the partially lax limit is a full subcategory of the lax limit so this can be tested as a \emph{property}. 
\begin{lemma}\label{lem: K-theory is absolute}
    Consider the motivic localisation functor $$L_{\mot}\colon \laxlim{\cX\in\Stk_\cB}\cP(\Smooth_\cX)_\ast\to \laxlim{\cX\in\Stk_\cB}\spcs(\cX)_\ast.$$ obtained by applying 
    $L_\mot$ levelwise. Then $L_\mot \cK^\circ$ lands in the full subcategory $$\plaxlim{\cX\in \Stk_\cB^\dagger}\spcs(\cX)_\ast\subset \laxlim{\cX\in\Stk_\cB}\spcs(\cX)_\ast$$
\end{lemma}
\begin{proof}
    This follows from \cite[Corollary 2.9, \S 5]{hoyois_cdh_2020}.
\end{proof}
We would now like to apply the motivic Thom spectrum functor from \cref{def: motivic Thom spectrum} and obtain an absolute refinement of algebraic cobordism. However, the absolute refinement of the rank zero K-theory anima only holds after motivic localisation, which the motivic Thom spectrum functor need not respect in general. 
\begin{lemma}\label{lem: motivic invariance of Thom}
    Consider the motivic Thom spectrum functor
    \[\mathrm{M}_\cX\colon \cP((\Smooth_\cX)_{\sslash \SH})\to \SH(\cX).\] If $A\in \cP((\Smooth_\cX)_{\sslash \SH})$ lies in the subcategory given by $\spcs(\cX)_{\slash \SH}$, then the restriction of $\mathrm{M}_\cX$ to $\cP(\Smooth_{\cX})_{/A}$ inverts motivic equivalences.
\end{lemma}
\begin{proof}
    The argument in \cite[Proposition 16.9]{bachmann2021norms} goes through once we know that $\SH$ itself satisfies Nisnevich descent and homotopy invariance in the form of \cite[Proposition 4.5 (iii), Example 5.12 (iii)]{khan_generalized_2022}.
\end{proof}
\begin{propn}\label{prop: MGL is absolute}
    Let $\Stk_{\cB}^{\mathrm{reg},\dagger}$ denote the category of regular quasi-fundamental stacks over some base $\cB$, marked at representable morphisms. Then the collection of algebraic cobordism spectra over these stacks assembles to an absolute object.
\end{propn}
\begin{proof}
    Since the motivic Thom spectrum functor is compatible with pullbacks, we may apply it levelwise to \cref{lem: K-theory is absolute}, where the passage to the motivic localisation is justified by \cref{lem: motivic invariance of Thom} and the observation that $\cK^\circ_\cX$ lives in $\spcs(\cX)$ whenever $\cX$ is regular by combining \cite[Theorem 4.9]{krishna2018algebraic} and \cite[Theorems 2.7, 4.1, 5.7]{thomason87groupscheme}.
\end{proof}
\begin{remark}
    In fact, since the pullback functors are strong symmetric monoidal, we similarly obtain a commutative algebra object in absolute motivic spectra.
\end{remark}
\begin{remark}\label{rem: regularity assumption}
    Note that $\MGL$ only lifts to an object of the partially lax limit over regular stacks since we need \cref{lem: motivic invariance of Thom}. Following \cref{warn: regularity hypothesis nomenclature} we will still refer to this as an \emph{absolute} lift. The reader may rest assured that we never need to consider the functoriality of algebraic cobordism over non-regular bases. 
    In the non-stacky context, following \cite[Remark 16.11]{bachmann2021norms} this regularity assumption may be circumvented by the observation that the functor $\SH$ on schemes satisfies cdh descent so that the map $\cK^\circ\to \Pic(\SH)$ factors through the cdh sheafification of the source, which is homotopy invariant (combining \cite[Theorem 6.3]{kerz2018algebraic} and a Noetherian approximation result) so that \cref{lem: motivic invariance of Thom} may be applied. Over stacks, it is known that $\SH$ and homotopy K-theory satisfy cdh descent by \cite{hoyois_cdh_2020}, but the converse --whether the cdh sheafification of algebraic K-theory of stacks is homotopy invariant-- is not known to the authors.
\end{remark}
\subsection{The equivariant cell structure}
A key feature of the spectrum $\MGL$ is that it admits a nice presentation in terms of Thom spectra on Gra{\ss}mannians, i.e. a cell structure. We show that this can be done in the equivariant context as well. For simplicitly, we will work over the base $\B A=[\Spec(\C)/A]$ throughout, where $A$ is a nice abelian group.
\begin{construction}
    Let $V_i$ denote a finite dimensional $A$-representation and let $\Gr_d(V_i)$ denote the Gra{\ss}mannian of $d$-planes in $V_i$, viewed a smooth projective $\C$-scheme with $A$ acting by its action on $V_i$. This admits a tautological rank $d$ bundle denoted $Q_i^d$. Let 
    \[\overline{Q}_i^d= Q_i^d-\cO^{\oplus d}\]
    denote its reduced variant.
\end{construction}
Next, we need an appropriate motivic notion of a complete $A$-universe, which is provided by a complete and saturated system of vector bundles in the sense of \cite[Definition 2.2]{hoyois_cdh_2020}.
\begin{lemma}\label{lem: Grass model for K}
    Let $\{V_i\}_{i\in I}$ be a complete and saturated system of vector bundles on $\B A$, then the map
    \[\varinjlim_{d, i}\Gr_d(V_i)\to \cK^\circ_A\]
    induced by the classifying maps of all $\overline{Q}^d_i$ is a motivic equivalence. 
\end{lemma}
\begin{proof}
    This follows immediately from \cite[Corollary 2.10]{hoyois_cdh_2020}.
\end{proof}
\begin{cor}\label{cor: Grass model for MGL}
    Letting $\{V_i\}_{i\in I}$ denote a complete and saturated system of vector bundles on $\B A$ as above, there is an equivalence
    \[\MGL_A\simeq \varinjlim_{d, i}\Sigma^{-2d,-d}\Th_{\Gr_d(V_i)}(Q_i^d).\]
\end{cor}
\begin{proof}
    Since we are working over a regular base, we may use \cref{lem: motivic invariance of Thom} and \cref{rem: colimit formula} to conclude. 
\end{proof}
\begin{remark}
    A similar statement holds for $\MGLP_A$, where we now obtain the formula
    \[\MGLP_A\simeq \varinjlim_{d, i}\Th_{\Gr_d(V_i)}(Q_i^d).\]
\end{remark}
This geometric presentation in particular allows us to determine the Betti realisation (cf. Appendix \ref{appendix: Betti realisation}) of $\MGL_A$.
\begin{cor}\label{cor: Betti image of MGL}
    There is an equivalence
    \[\Be^A_\C\MGL_A\simeq \MU_{A^\an}\]
\end{cor}
\begin{proof}
    By \cite[\S 5.2]{juliusthesis} we have an analogous colimit formula for $\MU_{A^\an}$. 
    It is clear that Betti realisation preserves Gra{\ss}mannians and Thom spectra by \cref{lem: examples of Betti realisations} and preserves filtered colimits per construction.
\end{proof}
An essential property of the Gra{\ss}mannian model for $\MGL_A$ is that it tells us that $\MGL_A$ is built essentially only out of Thom spectra; i.e. it admits an equivariant cell structure.
\begin{lemma}\label{lem: cell structure on Grass at T}
    Let $W$ be a finite-dimensional $\T_\C$-representation, then for any $d$ the motivic spectrum $\Gr_d(W)$ can be written as an iterated extension of motivic spectra of the form $\Th(\cE)$ for $\cE$ a finite-dimensional $\T_\C$-representation.
\end{lemma}
\begin{proof}
    We use the Bia{\l}ynicki--Birula decomposition of a smooth projective variety with $\T_\C$-action (see \cite[Theorem 3.1]{wendt2010more}). Note that the fixed point locus can be described as the coproduct
    \[\Gr_d(W)^{\T_\C}\simeq \bigsqcup_{\alpha\subset W}\Gr_{d_\beta}(\beta)\]
    of lower-dimensional nonequivariant Gra{\ss}mannians indexed by the characters of $\T_\C$ appearing in $W$. The Bia{\l}ynicki--Birula decomposition then gives us a finite filtration of $\Gr_d(W)$ by $\T_\C$-invariant subspaces\footnote{The usual statement of the Bia{\l}ynicki--Birula stratification does not make it clear that this is a ${\T_\C}$-equivariant stratification, but a glance at the construction (see \cite{jelisiejew2019bialynicki} for a general and modern account) shows that this is true.} such that the complements are affine bundles over some component of the fixed point locus. A slight modification of the argument in \cite[Proposition 3.2]{wendt2010more} then tells us that $\Gr_d(W)$ can be built as an iterated extension of Thom spectra over nonequivariant Gra{\ss}mannians. In order to further break up the nonequivariant Gra{\ss}mannians we use Wendt's argument again, in the form of \cite[Theorem A.3]{bachmann2022chow}.
\end{proof}
\begin{cor}\label{cor: cell structure on Thom over Grass at T}
    Letting $Q_W^d$ denote the tautological bundle over $\Gr_d(W)$ as above, then $\Th_{\Gr_d(W)}(Q_W^d)$ can also be built as an iterated extension of Thom spectra.
\end{cor}
\begin{proof}
    It suffices to note that one can \emph{Thomify} the purity cofibre sequences appearing in the proof of \cref{lem: cell structure on Grass at T} as in \cite[Lemma A.2]{bachmann2022chow} cf. the proof of \cite[Corollary 3.25]{hoyois_six_2017}.
\end{proof}
\begin{remark}
    The arguments of the Lemmata above can in fact be applied inductively to prove the same result at a torus of arbitrary rank, by applying the Bia{\l}ynicki--Birula stratification for the restricted action of a single one-dimensional torus to reduce to a torus of lower rank.
\end{remark}
\begin{cor}\label{cor: cellularity for MGL}
    Let $A$ be a nice abelian group, then $\MGL_A$ can be written as a filtered colimit where each term in the diagram arises as an iterated extension of Thom spectra of $A$-representaions.
\end{cor}
\begin{proof}
    With the Gra{\ss}mannian formula from \cref{cor: Grass model for MGL} in hand, this now follows immediately from \cref{cor: cell structure on Thom over Grass at T} in the case where our group is a torus of arbitrary rank. In order to obtain the result for all nice abelian groups, note that any such $A$ can be embedded into a sufficiently large torus $\T_\C^n$ by assumption. Since $\MGL$ is absolute we see that $\MGL_A$ is the restriction of $\MGL_{\T^n}$ along the inclusion. It then suffices to note that restriction preserves all colimits and is compatible with the formation of Thom spectra.
\end{proof}
\begin{remark}
    Upon applying equivariant Betti realisation, we see that this is a motivic lift of the observation that $\MU$ as well as its equivariant versions at abelian compact Lie groups are built out of complex representation cells. In particular, $\MU$ is only built out of cells of even dimension which is a key feature that powers the (equivariant) even filtration.
\end{remark}
Not only can $\MGL_A$ be written as a filtered colimit of strongly cellular objects, the filtered colimit itself has excellent cellularity properties which we will make critical use of.
\begin{lemma}\label{lem: inclusion of Grassmannians 1}
    Let $V\subset V\oplus\alpha=W$ be an inclusion of subrepresentations of $A$ and let $d$ be a positive integer no greater than the rank of $W$, then the closed immmersions
    \[\Gr_{d}(V)\xhookrightarrow{i} \Gr_d(W)\xhookleftarrow{j} \Gr_{d-1}(V)\]
     corresponding respectively to the inclusion $V\subset W$ and the map sending $V'\subset V$ to $V'\oplus\alpha\subset W$ have disjoint images and are complementary up to $\A^1$-homotopy. 
\end{lemma}
\begin{proof}
    Let $\Gr_d(W)\setminus \Gr_d(V)$ denote the complement of the image of $i$, then it is clear that $j$ factors through this complement as $$j'\colon \Gr_{d-1}(V)\to\Gr_d(W)\setminus \Gr_d(V).$$ Note that $j'$ is the zero section of a vector bundle whose fibre over some $V''\subset V$ is the collection of lines in a complement of $V''\subset W$ that are not entirely contained in $V$. As such, it is an $\A^1$-homotopy equivalence. 
\end{proof}

As a final application of the Gra{\ss}mannian model for $\MGL_A$, we see that its $A$-geometric fixed points can be described explicitly.
\begin{propn}\label{prop: identification of gfp}
    There is an equivalence of commutative ring spectra in $\SH(\C)$ of the form
    \[\Phi^A\MGLP\simeq\MGLP\otimes\bigotimes_{\alpha\in A^\vee\setminus\epsilon}\BGLP_+\]
\end{propn}
\begin{proof}
    Using the results of \cref{appendix: gfp2}, more specifically \cref{cor: gfp of Thom}, we see that one can reduce the computation to determining the nonequivariant Thom spectra
    \[\Th_{\Gr_d(V_i)^A}((Q^d_i)^A).\]
    Since $A$ is acting on $\Gr_d(V_i)$ through its action on $V_i$, we see that a fixed point of $\Gr_d(V_i)$ is precisely an invariant subspace of $V_i$. We may there decompose $V_i$ into its isotypical components $V_i^\alpha$ and obtain an equivalence
    \[\Gr_d(V_i)^A\simeq \coprod_{|W|=d}\prod_{\alpha\in A^\vee}\Gr_{|W^\alpha|}(V_i^\alpha)\]
    where the coproduct ranges over $d$-dimensional $A$-representations $W$ with isotopyical components $\alpha$. Taking the colimit over $d$ gives an equivalence
    \[(\BGLP_A)^A\simeq {\prod_{\alpha\in A^\vee}}'\BGLP.\]
    The notation above refers to a weak product; i.e. the union of all products over finite subsets of $A^\vee$. It then suffices to Thomify the expression above. In the case of a finite Gra{\ss}mannian, note that the tautological bundle $Q_i^d$ is such that its fibre over a $d$-dimensional subspace $W$ in $V_i$ is precisely $W$. If $W$ corresponded to an invariant subspace with only nontrivial characters in its isotpyical decompositions, its fixed points are of rank zero. If it instead only contains trivial isotypical summands then the fixed points are the whole subspace. We conclude that the resulting fixed points bundle on $(\BGLP_A)^A$ is trivial on every factor except the one corresponding to the trivial character $\epsilon\in A^\vee$. As such, taking Thom spectra recovers either $\MGLP$ in the case where $\alpha=\epsilon$ and $\BGLP_+$ otherwise, and we may conclude.
\end{proof}
\begin{remark}
    In the case of $\MGL_A$, inspection of the argument above tells us that there is a similar equivalence
    \[\Phi^A\MGL_A\simeq \MGL\otimes\bigotimes_{\alpha\in A^\vee\setminus \epsilon}\BGLP.\]
\end{remark}
\begin{cor}\label{cor: homotopy groups of gfp of mgl}
    The geometric fixed points homotopy groups of $\MGLP_A$ are given by
    \[\Phi^A_{\ast,\ast}\MGLP\cong \MGLP_{\ast,\ast}[(b_0^\alpha)^\pm,b_i^\alpha\mid \alpha\in A^\vee\setminus\epsilon,i\geq 1]\]
    where the classes $b_i^\alpha$ are in bidegree $(2i,i)$ for $i\geq 1$ and the classes $b_0^\alpha$ are in bidegree $(2,1)$.
\end{cor}
\begin{proof}
    Using \cref{prop: identification of gfp} we see that this is now an entirely nonequivariant question. We apply the results of \cite[Proposition 6.2, (ii,b)]{naumann2009motivic} where the class $b_i^\alpha$ comes from the factor of $\BGLP$ corresponding to the character $\alpha$ and the invertible classes $b_0^\alpha$ encode the fact we are working with the periodic version of $\BGL$.
\end{proof}
\begin{remark}
    As usual, we also obtain a non-periodic identification
    \[\Phi^A_{\ast,\ast}\MGL\cong \MGL_{\ast,\ast}[(b_0^\alpha)^\pm,b_i^\alpha\mid \alpha\in A^\vee\setminus\epsilon, i\geq 1].\]
\end{remark}
\begin{remark}
    Since we computed that geometric fixed points are a smashing localisation that precisely inverts the pre-Euler classes of all nontrivial characters in \cref{cor: euler model for gfp}, we see that this is also a computation of the localisation
    \[\Phi^A_{\ast,\ast}\MGL_A\cong \pi_{\ast,\ast}^A\MGL_A[e_\alpha^{-1}\mid \alpha\in A^\vee\setminus\epsilon].\]
\end{remark}
\begin{remark}\label{rem: gfp chern classes are also in image of unit map}
Not only does this computation supply us with an explicit presentation for the geometric fixed points, it also tells us that the generators for the geometric fixed point homotopy groups as an $\MGL_{\ast,\ast}$-algebra are entirely determined by the formal group law on $\MGL$: following \cite[Lemma 6.4 (i)]{naumann2009motivic} they arise as the images of the usual generators in the cooperation algebra for $\MU$. 
\end{remark}
\subsection{Incomplete algebraic cobordism spectra}
In this brief section, we construct and analyse variants of the (periodic) equivariant algebraic cobordism spectra $\MGLP_A$ indexed over general $A$-universes $\cU$ in the sense of \cref{appendix: gfp2}. These only arise in the proof of the regularity result below so our construction is rather ad-hoc. In the following, $\cU$ denotes an $A$-universe, and for any finite-dimensional summand $V$ of $\cU$ and integer $d$, $\Gr_d(V)$ denotes the corresponding Gra{\ss}mannian with its tautological bundle $Q^d_V$. We further adopt notation from \cref{appendix: gfp2} so that for an $A$-scheme $X$, $\SH^A(X)_\cU$ denotes the category of equivariant motivic spectra over $X$ indexed over $\cU$.
\begin{lemma}
    The suspension spectrum functor
    \[\SH^A(\Gr_d(V))_{\epsilon^\infty}\to \SH^A(\Gr_d(V))_\cU\]
    sends $\Th_{\Gr_d(V)}(Q^d_V)$ to a tensor-invertible object.
\end{lemma}
\begin{proof}
    The assumption that $V$ is a summand of $\cU$ guarantees that the fibres of $Q^d_V$ appear as summands of $\cU$. One can then repeat the argument in the proof of \cite[Proposition 6.5]{hoyois_six_2017} to conclude.
\end{proof}
\begin{defn}\label{def: incomplete MGLP}
    Let $\{V_i\}_{i\in I}$ run through a complete and saturated flag of $\cU$, then we define the incomplete periodic algebraic cobordism spectrum indexed on $\cU$ as the colimit
    \[\MGLP_{A,\cU}=\varinjlim_{d,i}\Th_{\Gr_d(V_i)}(Q^d_{V_i})\]
    in $\SH^A(\C)_\cU$.
\end{defn}
\begin{remark}
    The colimit $\varinjlim_{d,i}\Gr_d(V_i)$ appearing above will be denoted $\BGLP_{A,\cU}$.
\end{remark}
\begin{remark}\label{rem: comparison map}
    It is clear that the class of $A$-equivariant vector bundles arising as a summand of $\cU$ is closed under direct sum, we may therefore equip $\MGLP_{A,\cU}$ with a commutative algebra structure in $\SH^A(\C)_\cU$. Furthermore, we obtain a comparison map of commutative algebra objects in $\SH^A(\C)$ of the form
    \[\Sigma^{\cU_A-\cU}\MGLP_{A,\cU}\to \MGLP_{A}\]
    which is an equivalence if $\cU$ is complete.
\end{remark}
\begin{remark}
    Let $\alpha$ be a character which appears as a summand in $\cU$, then $\Th(\alpha)$ is an invertible object in $\SH^A(\C)_\cU$ and the colimit formula in \cref{def: incomplete MGLP} produces an equivalence
    \[\MGLP_{A,\cU}\otimes \Th(\alpha)\simeq \MGLP_{A,\cU}\]
    in analogy with \cref{rem: motivic Thom iso}. In particular, we obtain an Euler class $e_\alpha$ in bidegree $(-2,-1)$.
\end{remark}

\newpage
\section{Algebraic cobordism and global group laws}\label{sec: MGL and ggl}
Global group laws were introduced in \cite{hausmann2022global} as a way to decomplete the theory of equivariant formal group laws. In this section, we recall elements from this theory. We show that the bigraded homotopy groups of $\MGL_A$ assemble into a global group law, and thus that they admit a map from the universal object $\pi_{2*} \MU_A$. We prove an analogous result for higher tensor powers of $\MGL_A$. In \cref{ssec:regularity}, we show that this global group law has regular Euler classes. This will allow us to reduce problems to geometric fixed points, and is the main technical input into our results of \cref{sec: vanishing in MGL}.
\subsection{Global group laws}
First, we recall some basic definitions..
\begin{defn}
    An $\Ab$-algebra is a functor
    \[X\colon\Ab^\op\to \mathrm{CRing}\]
    from the (opposite of the) $1$-category of abelian compact Lie groups to commutative rings. A coordinate on an $\Ab$-algebra $X$ is a class $e\in X(\T)$ such that for any split surjective character $\alpha$ of a torus $A$ with kernel $i\colon K\to A$ there is a short exact sequence
    \[0\to X(A)\xrightarrow{(\alpha^\ast e)\cdot-}X(A)\xrightarrow{i^\ast}X(K)\to 0.\]
\end{defn}
\begin{remark}
    One can also define graded and bigraded $\Ab$-algebras in the obvious way. We will then require the coordinate to live in degree $-2$ resp. bidegree $(-2,-1)$.
\end{remark}
Related to this is the definition of a global group law. We note that the definition of a coordinate on an $\Ab$-algebra only involves its values at tori, and a global group law is therefore a priori only determined at tori.
\begin{defn}
    A global group law is a functor
    \[X\colon \mathrm{Lat}\cong\{\mathrm{tori}\}^\op \to \mathrm{CRing}\]
    together with a class $e\in X(\T)$, called the coordinate, such that for any split surjective character $\alpha$ of a torus $A$ with kernel $i\colon K\hookrightarrow A$, there is a short exact sequence 
    \[0\to X(A)\xrightarrow{(\alpha^\ast e)\cdot-}X(A)\xrightarrow{i^\ast}X(K)\to 0.\]
    A morphism of global group laws is a natural transformation which preserves the coordinate.
\end{defn}
\begin{remark}
    Clearly, any $\Ab$-algebra with coordinate gives rise to a global group law by restricting it to tori. The converse question; whether an $\Ab$-algebra with coordinate can be entirely recovered from its underlying global group law, is closely related to a phenomenon called regularity which will play an important role later.
\end{remark}
The primary examples of $\Ab$-algebras with coordinates come from commutative rings in global spectra which are equivariantly complex oriented, cf. \cite[Proposition 5.19]{hausmann2022global}. In fact, the graded $\Ab$-algebra with coordinate associated to the global spectrum $\MU$ recovers the universal example (\cite[Theorem C]{hausmann2022global}).
\subsection{The global group law on algebraic cobordism}
As established in \cref{prop: MGL is absolute} and \cref{rem: motivic Thom iso}, $\MGL\in \SH^\abs(\C)$ behaves like an equivariantly complex oriented global ring spectrum. We will use this to construct a bigraded $\Ab$-algebra with a coordinate coming from $\MGL$.
\begin{construction}\label{cons: ggl on MGL}
    Let $\underline{\pi}_{\ast,\ast}\MGL$ be the bigraded $\Ab$-algebra that sends a compact Lie group $A$ to the $A_\C$-equivariant homotopy groups of $\MGL$, i.e. $\pi_{\ast,\ast}^{A_\C}\MGL$. The coordinate is given by the Euler class $e_{\tau}\in \pi_{-2,-1}^{\T_\C}\MGL$ arising as the composite
    \[\MGL_\T\xrightarrow{a_\tau}\MGL_\T\otimes\Th(\tau)\simeq \MGL_\T\otimes\P^1.\]
\end{construction}
\begin{remark}
    Note that the $\Ab$-algebra above is well defined, i.e. fully functorial in all morphisms of compact Lie groups by \cref{rem: absolute homotopy groups functorial}.
\end{remark}
\begin{propn}\label{prop: ggl on MGL}
    The pair $(\underline{\pi}_{\ast,\ast}\MGL,e_\tau)$ forms a bigraded $\Ab$-algebra with coordinate.
\end{propn}
\begin{proof}
    Let $\alpha\colon A\to \T$ be a split surjective character of a torus, and let $i\colon K\hookrightarrow A$ denote the (split) inclusion of the kernel. Recall the cofibre sequence from \cref{rem: purity cof seq for tori} of the form
    \[A/K_+\to \mathbb{1}_A\xrightarrow{a_\alpha}\Th(\alpha).\]
    Let us now contemplate the corresponding long exact sequence in $\MGL_A$-cohomology. Every third term is of the form
    \begin{align*}
        [A/K_+,\Sigma^{t,w}\MGL_A]&\cong [i_\sharp\mathbb{1}_K,\Sigma^{t,w}\MGL_A]\\&\cong [\mathbb{1}_K,\Sigma^{t,w}i^\ast\MGL_A]\\&\cong[\mathbb{1}_K,\Sigma^{t,w}\MGL_K]
    \end{align*}
by the fact that $i$ is smooth and $\MGL$ is absolute. We can further identify $[\Th(\alpha),\Sigma^{t,w}\MGL_A]$ with $[\P^1,\Sigma^{t,w}\MGL_A]$ by the orientation isomorphism which furthermore takes $a_\alpha$ to $e_\alpha$. It then suffices to prove that this short exact sequence splits into the desired short exact sequences; since $\MGL$ is absolute and $i$ was assumed to be a split inclusion, this follows from \cref{rem: absolute homotopy groups functorial}. 
\end{proof}
\begin{remark}
    While $\underline{\pi}_{\ast,\ast}\MGL$ forms a bigraded $\Ab$-algebra with coordinate, it is also clear that one can restrict to the subring $\underline{\pi}_{2\ast,\ast}\MGL$ on classes in bidegrees of the form $(2w,w)$ and obtain singly graded $\Ab$-algebra with coordinate.  One can similarly extract an $\Ab$-algebra with coordinate from $\underline{\pi}_{\ast,\ast}\MGLP$ using the same arguments.
\end{remark}
\begin{cor}\label{cor : universal map of ggl's}
    There is a map of $\Ab$-algebras with coordinate
    \[\underline{\pi}_{2\ast}\MU\to \underline{\pi}_{2\ast,\ast}\MGL.\]
\end{cor}
\begin{proof}
    This follows from the universality of the former, which is \cite[Theorem C]{hausmann2022global}, where we have used that equivariant $\MU$ is concentrated in even degrees at abelian compact Lie groups to restrict to the even subring (\cite[Theorem 5.3]{comezana1996calculations}).
\end{proof}
\begin{remark}
    Upon evaluation at the trivial group, we see that this map is none other than the map from the Lazard ring classifying the formal group law on $\MGL$ that arises from the orientation corresponding to $e_\tau$; see \cite[\S 5.5]{hausmann2022global}. In particular, note that the map above sends the coordinate on $\underline{\pi}_{\ast}\MU$ to the coordinate on $\underline{\pi}_{\ast,\ast}\MGL$, i.e. the Euler class $e_{\tau^\an}$ on the topological side is sent to the motivic Euler class $e_\tau$. In fact, the Euler class of any character of any nice abelian group is pulled back from the universal Euler class at the torus along said character, so for any abelian compact Lie group $A$ with character $\alpha^\an$ we see that $e_{\alpha^\an}$ gets sent to the corresponding motivic Euler class.
\end{remark}
\begin{remark}\label{rem: Chern classes are hit too}
    We may also contemplate the effect of the universal map above after inverting all Euler classes of nontrivial characters; this gives rise to a ring map
    \[\Phi^A_{\ast}\MU\to \Phi^A_{2\ast,\ast}\MGL.\]
    The source has an essentially identical formula as an $\MU_\ast$-algebra in terms of Chern classes (see \cite[Proposition 2.25]{hausmann2023invariant}) arising from the universal formal group law on $\MU$. Using \cref{cor: homotopy groups of gfp of mgl} and \cref{rem: gfp chern classes are also in image of unit map} we see that these are compatible: the classes $b_i^\alpha$ in the target are hit by the classes $b_i^{\alpha^\an}$ in the source, cf. \cite[Lemma 6.4 (ii)]{naumann2009motivic}, \cite[Proposition 7.3, Lemma 7.9]{pstragowski_synthetic_2022}.
\end{remark}
\subsection{The co-operation algebra}
Let us remark that the discussion of $\Ab$-algebras with coordinates extends further to $\Ab$-algebras with several coordinates, and once again $\MU$ provides a universal example.
\begin{defn}[{\cite[\S 5.8]{hausmann2022global}}]
    Let $X$ be an $\Ab$-algebra with two coordinates $e^{(1)}, e^{(2)}$. Then the pair $(e^{(1)}, e^{(2)})$ is said to be strict if there exists a unit $\lambda\in X(\T)$ restricting to $1$ at $X(\{e\})$ such that $e^{(1)}=\lambda e^{(2)}$. More generally, an $n$-tuple of coordinates $(e^{(1)},\ldots,e^{(n)})$ is said to be strict if every pair $(e^{(i)},e^{(j)})$ is strict.
\end{defn}
\begin{remark}
    Note that this really is a condition on the unit: since $e^{(1)}$ and $e^{(2)}$ cut out the same ideal in $X(\T)$ by the axioms of a coordinate they must be related by a unit in $X(\T)$ which need not restrict to $1$ a priori.
\end{remark}
\begin{thm}[{\cite[Theorem E]{hausmann2022global}}]
    The $\Ab$-algebra $\underline{\pi}_\ast\MU^{\otimes n}$ admits a strict $n$-tuple of coordinates coming from the $n$ possible unit maps $\MU_A\to \MU_A^{\otimes n}$ and it is universal among $\Ab$-algebras with a strict $n$-tuple of coordinates.
\end{thm}
We see that it is rather straightforward to generalise \cref{prop: ggl on MGL} to this setting. 
\begin{construction}
    Consider the absolute commutative algebra $\MGL^{\otimes n}$. At the group $\T$, let $e_\tau^{(i)}$ denote the image of $e_\tau$ under the $i$-th unit map $\MGL_\T\to \MGL_\T^{\otimes n}$ for $i=1,\ldots,n$.
\end{construction}
\begin{lemma}
    The classes $e_\tau^{(i)}$ are coordinates on the bigraded $\Ab$-algebra $\underline{\pi}_{\ast,\ast}\MGL^{\otimes n}$.
\end{lemma}
\begin{proof}
    The argument of \cref{prop: ggl on MGL} goes through unchanged, except that one has two choices for the orientation isomorphism taking $a_\alpha$ to $e_\alpha$. Using the orientation isomorphism from the $i$-th tensor factor exhibits $e_{\tau}^{(i)}$ as a coordiante for all $i=1,\ldots,n$
\end{proof}
\begin{lemma}
    The $n$-tuple $(e^{(1)},\ldots,e^{(n)})$ of coordinates on $\underline{\pi}_{\ast,\ast}\MGL^{\otimes n}$ is strict.
\end{lemma}
\begin{proof}
    Fix $1\leq i\neq j\leq n$, we will prove that the pair $(e^{(i)},e^{(j)})$ is strict.
    Since $e^{(i)}$ and $e^{(j)}$ cut out the same ideal in $\pi^\T_{\ast,\ast}\MGL^{\otimes n}$ we know that there must be some unit $\lambda$ such that $\lambda e_\tau^{(i)}=e^{(j)}_\tau$. In fact, we may choose $\lambda$ to be the quotient of the $i$-th resp. $j$-th Thom class for $\tau$. Upon restriction to the trivial group we see that $\tau$ restricts to the trivial rank one bundle over a point, whence the equivariant Thom classes for $\tau$ must restrict to $1$ and the same is true for $\lambda$.
\end{proof}
\begin{remark}
    As usual, the arguments above go through for the singly graded $\Ab$-algebra $\underline{\pi}_{2\ast,\ast}\MGL^{\otimes n}$ as well as for $\MGLP^{\otimes n}$.
\end{remark}
\begin{cor}\label{cor: unit map for n coords}
    There is a map of $\Ab$-algebras with a strict $n$-tuple of coordinates
    \[\underline{\pi}_{2\ast}\MU^{\otimes n}\to \underline{\pi}_{2\ast,\ast}\MGL^{\otimes 
    n}.\]
\end{cor}
\subsection{Regularity}\label{ssec:regularity}
An important notion in the study of global group laws is that of \emph{regularity} as in \cite[Definition 5.9]{hausmann2022global}: A (potentially (bi)graded) global group law $X$ is said to be regular if for every torus $A$ and linearly independent $n$-tuple of characters $\alpha_1,\ldots,\alpha_n$ of $A$ the sequence $(e_{\alpha_1},\ldots,e_{\alpha_n})$ forms is regular in the ring $X(A)$. Per definition, the Euler class of a split surjective character is always a regular element since it participates in the short exact sequences defining a global group law. Asserting regularity of these further Euler classes has two distinct advantages that will be critical in the rest of the work. First, it says that inverting Euler classes is a harmless process since the value at a torus will inject into the inversion. This allows us to reduce statements about a global group law to its geometric fixed points. Second, when a global group law arises from a global spectrum or similar object, restriction to finite abelian groups from a torus is encoded by killing a (pre-)Euler class on the spectral level. If this class is regular then it simply induces a quotient on the ring level. An essential observation in \cite[Corollary 5.25]{hausmann2022global} is that the universal global group law is regular. We will prove that the global group law on $\MGLP$ (and hence also $\MGL$) is regular as well. The argument presented below is originally due to Markus Hausmann in the topological setting (i.e. for $\MUP$) and set to appear in \cite{regularityarg}. We thank him for explaining his proof.
\begin{construction}
    In the following, we let $A$ be a nice abelian group. We let $\chi_{A,n}$ be the subgroup of $(A\times \T)^\vee$ on characters of the form $\alpha\otimes \tau^i$ with $0<|i|<n$ a power of the tautological character of $\tau$ and $\alpha$ an arbitrary character of $A$. We construct a corresponding $A\times \T$-universe $\cU_n$ which contains countably infinitely many copies of each character in $\chi_{A,n}$ as well as a copy of the inflation (along the quotient map) of a complete $A$-universe. Further, $\cF_n$ will denote the family of subgroups of $A\times \T$ which is the union of the kernels of all characters in $\chi_{A,n}$. We will consistently view $\T$ as a subgroup of $A\times \T$ using the inclusion $0\times\id$
\end{construction}
\begin{remark}
   For any $n$, the fixed-point universe $\cU_n^\T$ forms a complete $A$-universe. Note that the family $\cF_n$ consists of all subgroups of the form $K\times \mu_d$ for $1\leq d<n$. We therefore remark the following.
   \begin{itemize}
       \item As $n$ grows, the families $\cF_n$ grow larger with colimit $\cF_\infty=\cF[\T]$ the family of subgroups not containing $T$.
       \item $\cU_\infty$ is a complete $A\times \T$-universe.
       \item For any $n$, $\cU_n$ restricts isomorphically to the nontrivial isotypical part of a complete $A\times\mu_n$-universe.
   \end{itemize} 
\end{remark}
Let $\E\cF_n$ denote the construction of \cref{cons: univ space of family} applied to the family $\cF_n$, for which we now construct a geometric model. We refer to \cref{appendix: gfp2} for details on the notions used below.
\begin{lemma}\label{lem : cell model for Fn}
    Let $\{V^n_i\}_{i\in I}$ denote a saturated flag of the nontrivial isotypical part of $\cU_n-\epsilon^\infty$, then the map
    \[\varinjlim_{i\in I}\mathbb{V}(V_i)\setminus 0\to \E\cF_n\]
    is a motivic equivalence.
\end{lemma}
\begin{proof}
    This follows immediately from a similar argument as in \cref{lem: cellular model for gfp}.
\end{proof}
In particular, the corresponding localisation $\widetilde{\E}\cF_n$ is obtained by inverting all Euler classes of characters in $\chi_{A,n}$. We now proceed with a key technical lemma that allows us to compute geometric fixed points internal to an incomplete universe by inverting a small set of Euler classes.
\begin{lemma}\label{lem: gfp from less euler classes}
    Let $\widetilde{\E}\cF_n$ be as above, then the functor
    \[\SH^{A\times \T}(\C)_{\cU_n}\to \Mod(\SH^{A\times \T}(\C)_{\cU_n};\widetilde{\E}\cF_n)\xrightarrow{(-)^\T_{\cU_n}}\SH^A(\C)\]
    agrees with $\Phi^\T_{\cU_{\infty}}$.
\end{lemma}
\begin{proof}
    Let $W$ be a finite-dimensional summand of the nontrivial part of $\cU_\infty$, then the inclusion of fixed points induces an equivalence of the form
    \[\Th(W^\T)\otimes \widetilde{\E}\cF_K\simeq \Th(W)\otimes\widetilde{\E}\cF_K.\]
    Indeed, the complementary representation $W-W^\T$ has isotropy contained in $\cF_n$. Therefore, any $\Th(W)$ is equivalent to a tensor-invertible object in $\Mod(\SH^{A\times \T}(\C)_{\cU_n};\widetilde{\E}\cF_n)$ as $W^\T$ is a summand of $\cU_n$. Recall that $(-)^\T_{\cU_n}$ is defined as the right adjoint to the inflation functor from $\SH^A(\C)$ while $\Phi^\T_{\cU_\infty}$ is defined as the unique extension along $\Sigma^{\cU_\infty-\cU_n}$ of the fixed points functor and is insensitive to first tensoring with $\widetilde{\E}\cF_n$ by dint of the identification $\widetilde{\E}\cF_\infty\simeq \E\cF[\T]$. Since the stabilisation functor $\Sigma^{\cU_\infty-\cU_n}$ is an equivalence after tensoring with $\widetilde{\E}\cF_n$, we see that these functors must be equivalent.
\end{proof}
\begin{remark}
    We warn the reader that this equivalence of fixed point functors does not necessarily arise from an identification of the $\widetilde{\E}\cF_n$-local category as was the case with the usual geometric fixed points functor.
\end{remark}
\begin{lemma}\label{lem: incomplete provides section}
    Let $\MGLP_{A\times \T,\cU_n}$ denote the incomplete algebraic cobordism spectrum from \cref{def: incomplete MGLP} and let $K$ denote the kernel of a surjective $A\times\T$-character of the form $\alpha\otimes\tau^n$ so that $\mu_n=K\cap\T$ is a subgroup. Then the composite
    \[(\MGLP_{A\times \T,\cU_n}\otimes \widetilde{\E}\cF_n)^\T_{\cU_n}\to (\MGLP_{A\times \T}\otimes \widetilde{\E}\cF_n)^\T\to (\MGLP_{K}\otimes \widetilde{\E}\cF_n)^{\mu_n}\]
    of the comparison map from \cref{rem: comparison map} and the restriction map is an equivalence of $A$-spectra.
\end{lemma}
\begin{proof}
    In the case where $\alpha=\epsilon$ is a trivial character of $A$, $K$ is of the split form $A\times\mu_n$. In general, it may be a non-split extension of $A$ by $\mu_n$, but we may restrict our proof to the maximal case $K=A\times\mu_n$ since it only uses that $\widetilde{\E}\cF_n$ restricts to $\widetilde{\E}\cF[\mu_n]$ at $K$, or in fact at any subgroup of $A\times\mu_n$, and that $K/\mu_n$ identifies with $A$. By \cref{lem: gfp from less euler classes} and the aforementioned observation that $\widetilde{\E}\cF_n$ restricts to $\widetilde{\E}\cF[\mu_n]$ at $A\times\mu_n$, we see that composite map is of the form
    \[\Phi^\T_{\cU_n}\MGLP_{A\times \T,\cU_n}\to \Phi^{\mu_n}\MGLP_{A\times\mu_n}.\]
    In fact, this map arises as a Thomification, by \cref{cor: gfp of Thom} it suffices to prove that the restriction map
    \[\BGLP_{A\times \T,\cU_n}^\T\to \BGLP_{A\times\mu_n}^{\mu_n}\]
    is an $A$-equivariant equivalence. By a similar argument as in the proof \cref{prop: identification of gfp} the fixed points of these Gra{\ss}mannians may be computed as
    \begin{align*}
        \BGLP_{A\times \T,\cU_n}^\T&\simeq\prod_{\beta\subset \cU_n}\BGLP_A,\\
        \BGLP_{A\times \T}^{\mu_n}&\simeq\prod_{(A\times\mu_n)^\vee}\BGLP_A,
    \end{align*}
    where the index $\beta$ ranges over characters appearing in $\cU_n$. Per construction, such characters correspond precisely to $(A\times\mu_n)^\vee$ so we are done.
\end{proof}
\begin{thm}\label{thm: regularity}
    The bigraded global group law $\underline{\pi}_{\ast,\ast}\MGLP$ is regular.
\end{thm} 
\begin{proof}
    Following the second characterisation provided in \cite[Lemma 2.9]{hausmann2022global}, we must show that for all nice abelian groups $B$ and surjective characters $\beta$ of $B$, the class $e_\beta$ in $\pi_{\ast,\ast}^B\MGLP$ is regular. We proceed by induction on the rank of the kernel of $\beta$. The base case of rank zero is covered by the axioms of a global group law since we are then dealing with a split surjective character. Now for the induction step, note that $\beta$ is surjective so that there must be a decomposition $B=A\times \T$ so that $\beta$ is of the form $\alpha\otimes\tau^n$. Now for all $1\leq |i|<n$ the Euler classes of the characters in $\chi_{A,n}$ are regular by the induction hypothesis so we may invert these. Associated to $\beta$ and its kernel $K\subset B$ is the purity cofibre sequence
    \[B/K\to \mathbb{1}_B\xrightarrow{a_\beta}\Th(\beta)\]
    which induces a long exact sequence in $\MGLP_B$-cohomology. It therefore suffices to show that this splits into short exact sequence after inverting the Euler classes of characters in $\chi_{A,n}$, e.g. by producing a section of the localised restriction maps
    \[\pi^B_{\ast,\ast}\MGLP[e_{\chi_{A,n}}^{-1}]\to \pi^K_{\ast,\ast}\MGLP[e_{\chi_{A,n}}^{-1}].\]
    Now by \cref{lem : cell model for Fn} this is precisely the effect on $A$-equivariant homotopy groups of the restriction map
    \[(\MGLP_B\otimes\widetilde{\E}\cF_n)^\T\to (\MGLP_K\otimes\widetilde{\E}\cF_n)^{\mu_n}\]
    and by \cref{lem: incomplete provides section} this admits a section induced by the comparison map from the incomplete algebraic cobordism spectrum indexed on $\cU_n$.
\end{proof}

\begin{remark}\label{rem: regularity for variants}
    The argument for regularity provided above is sufficiently sturdy that it can be copied verbatim for sufficiently nice variants of $\MGLP$. Indeed, all operations involved in the proof of regularity are exact and commute with filtered colimits so that one can replace $\MGLP$ by various localisations or quotients since these will still give rise to absolute objects which are oriented at every group whence their homotopy groups give rise to global group laws (as long as one can equip their bigraded equivariant homotopy groups with a commutative ring structure). For example, one can take the absolute motivic spectrum $\MGL/p^i$ for some $i$ large enough such that its homotopy groups admits a commutative ring structure, e.g. $i\geq 5$ by \cite[Theorem 1.2]{burklund2022multiplicative}.
\end{remark}
We conclude our discussion of regularity with two immediate consequences. These are stated for $\MGL$ but are equally valid (and in fact derived from) the analogous statements for $\MGLP$.
\begin{cor}\label{cor: regularity consequence 1}
    Let $A$ be an arbitrary nice abelian group and pick a minimal presentation of $A$ as the kernel of a character $\alpha$ of a torus $\T^{r}$. Then there is an equivalence of $\MGL_{\ast,\ast}$-algebras
    \[\pi^{A}_{\ast,\ast}\MGL\cong \pi^{\T^r}_{\ast,\ast}\MGL/e_\alpha.\]
\end{cor}
\begin{proof}
    This now follows immediately from the splitting of the long exact sequences associated to the character $\alpha$.
\end{proof}
\begin{cor}\label{cor: regularity consequence 2}
    Let $\T^r$ be a torus of arbitrary rank $r\geq 1$, then the map
    \[\pi^{\T^r}_{\ast,\ast}\MGL\to \Phi^{\T^r}_{\ast,\ast}\MGL\]
    is injective
\end{cor}
\begin{proof}
    By \cref{cor: euler model for gfp} we see that the geometric fixed points are obtained by inverting a set of regular elements.
\end{proof}
\begin{remark}\label{rem: regularity of co-ops}
    The regularity statement obtained here is equally valid for the $\Ab$-algebras $\underline{\pi}_{\ast,\ast}\MGLP^{\otimes d}$ for $d\geq 2$. Indeed, the choices of Euler classes differ by a unit so it suffices to pick the ones coming from the first tensor factor. Then inductively apply \cref{cor: cellularity for MGL} to see that at every nice abelian group this bigraded ring is a localisation of a free module over the coefficients of $\MGL$ so that this is an instance of \cref{rem: regularity for variants}.
\end{remark}
\newpage
\section{Coefficients of algebraic cobordism}\label{sec: vanishing in MGL}
In this section, we use the regularity of the global group law on $\MGL$ established in the previous section to obtain comparison and vanishing results for the bigraded homotopy groups $\pi_{t,w}^A \MGL$, showing that it behaves very similarly to the nonequivariant case. The computation is well-organized by the \emph{Chow degree} $t - 2w$; we show that when the Chow degree is negative, these groups vanish, and that when the Chow degree is zero, the map from $\pi_{2*} \MU_A$ constructed in the previous section is an isomorphism. We then show that after $p$-completion, this extends to a complete description of the bigraded homotopy of $\MGL_A$; we have an isomorphism $\pi^A_{*, *}\MGL^\wedge_p \simeq \pi_{2*}^A\MU^\wedge_p[\tau],$ where $\tau$ is in bidegree $(0, -1)$. 
\subsection{Below the Chow line}
\begin{thm}[{see \cite[Theorem B.1]{bachmann2022chow}}]
    For all $k>0,w\in \Z$ the group $\MGL_{2w-k,w}$ vanishes. 
\end{thm}
\begin{remark}
    The result above is true in much larger generality than stated; in fact the cohomological version is true for the algebraic cobordism of any essentially smooth scheme over a semi-local PID after inverting the residue characteristics by reduction to similar vanishing results in motivic cohomology using the slice filtration. We will only work over $\C$ and prove results for the algebraic cobordism of a point.
\end{remark}
Given a bidegree $(t,w)$ we will refer to $t-2w$ as the associated Chow degree; this result can then be interepreted as saying that the homotopy of $\MGL$ vanishes in negative Chow degree. In fact, the equivariant version is immediate from the results established in the previous section.
\begin{propn}\label{prop: vanishing in MGL}
    Let $A$ be a nice abelian group, then for all $k>0,w\in \Z$, the group $\pi^A_{2w-k,w}\MGL$ vanishes.
\end{propn}
\begin{proof}
We begin with the case of a torus $\T^r$, \cref{cor: regularity consequence 2} tells us this follows from the same vanishing range for $\Phi^{\T^r}_{\ast,\ast}\MGLP$, which by \cref{cor: homotopy groups of gfp of mgl} can be expressed as 
\[\Phi^{\T^r}_{\ast,\ast}\MGLP\cong \MGLP_{\ast,\ast}[(b_0^\alpha)^\pm,b_i^\alpha\mid i\geq 1,\alpha\in (\T^r)^\vee\setminus\epsilon]\]
where all the generators sit in Chow degree zero. The case of a general nice abelian group now follows from \cref{cor: regularity consequence 1}.
\end{proof}
\begin{remark}\label{rem : Chow vanishing for perfect pure}
    For general nice abelian groups, we may use the equivariant Thom isomorphism to conclude that the same vanishing range holds for the equivariant algebraic cobordism of Thom spectra of virtual representations. In fact, the cell structure on algebraic cobordism provided by \cref{cor: cellularity for MGL} shows that this vanishing range also holds for the cooperation rings $\pi^A_{\ast,\ast}\MGL^{\otimes d}$.
    When the group $A$ is \emph{finite}, we deduce the same vanishing range for the $\MGL_A$-homology of orbits $A/K_+$ since these are self-dual so this reduces to the $K$-equivariant algebraic cobordism of a point.
\end{remark}
We warn the reader that orbits of positive-dimensional nice abelian groups need not satisfy this vanishing in homology; their duals involve shifts which are \emph{not} in Chow degree zero.
\subsection{The Chow line}
Another critical result in nonequivariant motivic homotopy theory over $\C$ is the identification of the Chow line $\MGL_{2\ast,\ast}$ with the Lazard ring. Once again, the structure of a global group law allows us to prove this by reduction to the nonequivariant setting.
\begin{thm}[{\cite[Proposition 8.2]{hoyois2015algebraic}}]\label{thm: HMH}
    The map $L_\ast\to \MGL_{2\ast,\ast}$ classifying the formal group on $\MGL$ is an equivalence, where $L_\ast$ is the (graded) Lazard ring.
\end{thm}
The role of the Lazard ring is now replaced by the graded global group law associated to $\MU$.
\begin{propn}\label{prop: identification of the Chow line}
    The map from \cref{cor : universal map of ggl's} induces an isomorphism of graded rings
    \[\pi_{2\ast}^{A}\MU\cong \pi_{2\ast,\ast}^{A_\C}\MGL\]
    at all abelian compact Lie groups.
\end{propn}
\begin{proof}
    Once again, we may pass to periodic variants and restrict to tori since the values at all other abelian compact Lie groups are determined by the values at tori by \cref{cor: regularity consequence 1}. Note that the aforementioned map
    \[\pi_{2\ast}^{\T^r}\MUP\to \pi_{2\ast,\ast}^{\T^r}\MGLP\]
    sends Euler classes to their motivic variants since it is a map of global group laws.
    In the other direction, we see that Betti realisation provides us with a map of global group laws which at $\T^r$ is of the form
    \[\pi_{2\ast,\ast}^{\T^r}\MGLP\to \pi_{2\ast}^{\T^r}\MUP.\]
    We therefore obtain a commutative diagram of graded rings
    \[\begin{tikzcd}
    \pi_{2\ast}^{\T^r}\MUP\arrow[d]\arrow[r]&\pi_{2\ast,\ast}^{\T^r}\MGLP\arrow[d]\arrow[r]&\pi_{2\ast}^{\T^r}\MUP\arrow[d]\\
    \Phi^{\T^r}_{2\ast}\MUP\arrow[r]&\Phi^{\T^r}_{2\ast,\ast}\MGLP\arrow[r]&\Phi^{\T^r}_{2\ast}\MUP,
    \end{tikzcd}\]
    where the vertical arrows are injective by \cref{cor: regularity consequence 2}.
    The bottom right horizontal arrow is an isomorphism since both source and target are polynomial rings over the Lazard ring by \cref{cor: homotopy groups of gfp of mgl} and \cref{thm: HMH} for the source and \cite[Proposition 2.25]{hausmann2023invariant} for the target, and Betti realisation carries the motivic generators to the homotopical ones. We conclude that the top right horizontal arrow is injective. Now the top horizontal composite comes from a map of graded global group laws hence must be the identity since the source is initial. We conclude that the top horizontal arrows are inverses. 
\end{proof}
\begin{remark}\label{rem: cooperations on the Chow line}
    In a similar fashion, we see that the map in \cref{cor: unit map for n coords} is also an isomorphism at the circle and its subgroups. Indeed, since geometric fixed points are symmetric monoidal the same argument goes through and the necessary regularity of Euler classes is documented in \cref{rem: regularity of co-ops}.
\end{remark}
\subsection{Above the Chow line}
Another key input in nonequivariant motivic homotopy theory over $\C$ is a complete computation due to \cite{hu2011remarks} ($p=2$) and \cite{stahn2016motivic} (all primes) of the homotopy groups of $p$-complete $\MGL$. We let $p$ be an arbitrary fixed prime in this section.
\begin{thm}[{\cite[Theorem 7]{hu2011remarks}, \cite[Proposition 3.8]{stahn2016motivic}}]\label{thm: HKO computation}
    There is an isomorphism
    \[\pi_{\ast,\ast}\MGL^\wedge_p\cong  (\pi_{2\ast}\MU^\wedge_p)[\tau]\]
    where $\tau$ is in bidegree $(0,-1)$. In particular, we may rewrite the right hand side as $(L_{2\ast})^\wedge_p[\tau]$.
\end{thm}
\begin{remark}\label{rem: tau classes}
    The class $\tau$ appearing here is the eponymous class in the $C\tau$-philosophy, and in fact already appears in the $p$-complete homotopy groups of the motivic sphere spectrum over $\C$. This spherical lift is described at $p=2$ in \cite[Lemma 23]{hu2011remarks} (and see the remark below it for an elementary proof), but in \cite[\S 2.1]{burklund_galois_2022} it is observed that one can already construct (compatible) classes $\tau_i\colon S^1\to \G_m/p^i$ for every $i$ and assemble them to the limit. We will implicitly view these classes as living in the mod $p^i$ equivariant motivic stable stems by the inflation map.
\end{remark}
Our equivariant analogue proceeds by first working modulo a power of $p$ and then assembling the result.
\begin{propn}\label{prop: mod p^i equivariant HKO}
    Let $i\geq 5$ and let $A$ be a nice abelian group, then there is an equivalence
    \[\pi_{\ast,\ast}^A\MGL/p^i\cong \pi_{2\ast}^A\MU/p^i[\tau]\]
    for a class $\tau$ in bidegree $(0,-1)$.
\end{propn}
This may be interpreted as saying that the equivariant homotopy groups of $\MGL/p^i$ are completely determined by the Chow line; every class can be written as the product of a class on the Chow line and a power of $\tau$. 
\begin{proof}
    Note that the equivariant motivic spectra $\MGL_A/p^i$ assemble to form an absolute motivic ring spectrum such that their values at any nice abelian group are $\MGL_A$-modules hence inherit Thom isomorphisms. Following \cref{rem: regularity for variants} we see that this global group law is regular by the same argument as in the integral case. Further, note that by \cref{prop: vanishing in MGL} and \cref{prop: identification of the Chow line}, $p$ acts regularly on the Chow line of $\pi_{\ast,\ast}^A\MGL$ and there are no nonzero classes in homotopical degrees below it, so that one may identify $\pi_{2\ast,\ast}^A\MGL/p^i$ with $\pi_{2\ast}^A\MU/p^i\cong (\pi_{2\ast}^A\MU)/p^i$ since the equivariant Lazard rings at tori are integral domains by \cite[Corollary 5.44]{hausmann2022global}. At any torus $\T^r$ one then defines a map of \emph{bigraded} rings
    \[f\colon\pi_{2\ast}^{\T^r}\MU/p^i[\tau_i]\to \pi_{\ast,\ast}^{\T^r}\MGLP/p^i\]
    extending the inclusion of the Chow degree zero line by sending the free variable $\tau_i$ in bidegree $(0,-1)$ to the correspondingly named class in the source. We want to show that $f$ is an equivalence. To prove this, we may use regularity to obtain a commutative diagram
    \[\begin{tikzcd}
\pi_{2\ast}^{\T^r}\MUP/p^i[\tau_i]\arrow[d]\arrow[r, "f"]&\pi_{2\ast,\ast}^{\T^r}\MGLP/p^i\arrow[d]\\
        \Phi^{\T^r}_{2\ast}\MUP/p^i[\tau_i]\arrow[r]&\Phi^{\T^r}_{2\ast,\ast}\MGLP/p^i,
    \end{tikzcd}\]
    where the vertical arrows are injective and \cref{thm: HKO computation} tells us that the bottom horizontal arrow is an equivalence whence $f$ is injective. In fact, in the bottom left (hence also bottom right) vertex of the diagram, every class arises as an appropriate $\tau_i$-multiple of a class on the Chow line. In particular, the right vertical arrow being injective tells us that there are no classes in odd topological degree in the top right vertex. We may therefore prove $f$ is surjective by providing a section
    \[g\colon \pi_{\ast,\ast}^{\T^r}\MGL/p^i\to \pi_{2\ast}^{\T^r}\MU/p^i[\tau_i]\]
    as follows: since every nonzero class $x$ in the source of is in a bidegree of the form $(2w,w)+(0,-d)$, we set $g(x)=\Be^{\T^r}(x)\cdot\tau_i^d$, and this is a well defined map of bigraded rings. To show that $fg$ is the identity, we note that the equality $f(g(x))=x$ may be tested after injecting into geometric fixed points, where it is true by \cref{thm: HKO computation}. The passage from tori to general nice abelian groups follows from regularity as in \cref{cor: regularity consequence 1}.
\end{proof}
\begin{cor}\label{cor: p-complete HKO}
    For any torus $\T^r$, taking the limit along $i$ of the result in \cref{prop: mod p^i equivariant HKO} gives rise to an identification
    \[\pi_{\ast,\ast}^{\T^r}\MGL^\wedge_p\cong (\pi_{2\ast}^{\T^r}\MU^\wedge_p)[\tau].\]
\end{cor}
\begin{proof}
     Since the equivariant Lazard rings at tori are an integral domains (which is a consequence of regularity, see \cite[Corollary 5.44]{hausmann2022global}), the system of bigraded rings $\{\pi_{\ast,\ast}^{\T^r}\MGL/p^i\}$ has surjective transition maps.
\end{proof}
\begin{remark}
    By our computation of the $p$-complete homotopy of equivariant $\MGL$, we see that $i=1$ suffices to equip $\MGL/p^i$ with a regular global group law since its coefficients already form commutative rings.
\end{remark}

\newpage
\section{Pure reconstruction}\label{sec: perfect pure}
From this section onward, $A$ will be a \emph{finite} abelian group. Some analogues of the results below are valid for more general nice abelian groups, and these will be indicated. Equipped with \cref{prop: vanishing in MGL}, we give now a sheafy model for an appropriately cellular version of $A$-equivariant motivic spectra over $\C$. The key notion is that of a \emph{perfect pure} $A$-equivariant motivic spectrum; our arguments (and result) should be compared to those of \cite[\S 3]{haine2023spectral}.
\subsection{Equivariant pure motives}
\begin{defn}\label{def: perfect pure motives}
    Let $\Pure_\C(A)$ denote the minimal full subcategory of $\SH^{A}(\C)$ which is closed under extensions and retracts and contains all objects of the form $\Sigma^\cE_+A/K$ for subgroups $K$ of $A$ and virtual $A$-representations $\cE$. The cellular subcategory $\SH^{A}(\C)^\cell$ is the minimal full subcategory of $\SH^{A}(\C)$ which contains $\Pure_\C(A)$ and is closed under colimits and desuspensions.
\end{defn}
    It is clear that $\Pure_\C(A)$ consists of compact objects, and is closed under taking tensor products and duals. We conclude that $\SH^A(\C)^\cell$ admits a presentably symmetric monoidal structure as well. Furthermore, it follows from the definition that the various categories of perfect pure motivic spectra at $A$ and its subgroups are preserved by restriction, inflation, and geometric fixed points. In the following, all fixed point and (co)induction functors will be replaced by their cellular variants, namely the respective adjoints to inflation and/or restriction between cellular subcategories. Furthermore, the cellular subcategory is not closed under arbitrary limits. In particular, we warn the reader that a completion of a cellular object need not be cellular anymore. Henceforth, any limit operations in the cellular subcategory must be understood internally.
    An essential property of the perfect pure objects in the cellular category is that their $\MGL_A$-homology is well behaved.
\begin{lemma}\label{lem : Chow vanishing for perfect pure 2}
    Every $X\in \Pure_\C(A)$ satisfies Chow vanishing: for all $k>0,w\in \Z$ the group $\MGL^{A}_{2w-k,w}X$ vanishes.
\end{lemma}
\begin{proof}
Using \cref{rem : Chow vanishing for perfect pure} in conjunction with the motivic Thom isomorphism of \cref{rem: motivic Thom iso}, we obtain the base case for the generators $\Sigma^\cE_+A/K$. By the vanishing result, any extension gives rise to a short exact sequence in $\MGL_A$-homology so we may deduce the same vanishing result holds for any perfect pure equivariant motivic spectrum.
\end{proof}
\begin{defn}\label{def: pure epimorphisms}
    A map $f\colon X\to Y$ in $\Pure_\C(A)$ is said to be a pure epimorphism if its fibre, computed in $\SH^A(\C)$, is again perfect pure. Similarly, we say that $f$ is a pure monomorphism if its cofibre is perfect pure or, equivalently, if its dual is a pure epimorphism.
\end{defn}
\begin{ex}\label{ex: transition maps}
    Using \cref{cor: cell structure on Thom over Grass at T} we see that all Gra{\ss}mannians are perfect pure and the Thom spectra of their (reduced) canonical bundles are perfect pure. In fact, by Thomifying the purity cofibre sequences produced by \cref{lem: inclusion of Grassmannians 1}, we see that \cref{cor: cellularity for MGL} can be strengthened to the observation that $\MGL_A$ admits a presentation as a filtered colimit of perfect pure objects along pure monomorphisms.
\end{ex}
\begin{remark}
    When $A$ is a finite abelian group, the motives of its orbits are self-dual and satisfy Chow vanishing, whence we include them as perfect pure objects. This will have some advantages later on, as it means the subcategory of cellular motivic spectra has some improved equivariant properties. We note here that in the case of a general nice abelian group, the more correct definition of the perfect pure objects the closure under extensions and retracts of the full subcategory on Thom spectra of virtual representations, without orbits. In fact, the attentive reader will note that every result stated above goes through in this setting: the cell structure on Gra{\ss}mannians, and hence $\MGL_A$, does not make use of orbits.
\end{remark}
\subsection{Chow weight structure}
Consider the module category $\Mod(\SH^A(\C)^\cell;\MGL_{A})$, which we will abbreviate to $\Mod(\MGL_{A})$. Essentially per definition this is generated under limits and colimits by the full subcategory denoted $\Pure(A;\MGL)$ on objects the form $X\otimes \MGL_{A}$ with $X\in \Pure_\C(A)$. We will see that this module category is in fact entirely determined by perfect pure motives in a controllable way.
\begin{propn}\label{prop: Chow weight structure}
    The inclusion of $\Pure(A;\MGL)$ induces an equivalence
    \[\cP_\Sigma(\Pure(A;\MGL);\Sp)\simeq \Mod(\MGL_A).\]
\end{propn}
We being with the following elementary observation.
\begin{lemma}\label{lem: Chow weight structure}
    The full subcategory $\Pure(A;\MGL)$ is such that it generates $\Mod(\MGL_A)$ under finite (co)limits and filtered colimits, is closed under extensions and retracts, and is such that the mapping spectrum between any two elements is connective.
\end{lemma}
\begin{proof}
    The first two statements are true per construction, so it suffices to verify that last one i.e. that for all perfect pure $X,Y$ the mapping spectrum
    \[\Map_{\MGL_A}(X\otimes\MGL_A,Y\otimes\MGL_A)\]
    is connective. Since perfect pure motivic spectra are closed under duals and tensor products, it suffices to prove that
    \[\MGL^{A}_{-k,0}X=0\]
    for any perfect pure $X$. This now follows from \cref{lem : Chow vanishing for perfect pure 2}.
\end{proof}
The main result now follows immediately after noting that this equips $\Mod(\MGL_A)$ with a weight structure which we dub the \emph{Chow weight structure}.
\begin{proof}[Proof of \cref{prop: Chow weight structure}]
    Following \cite[Remark 2.2.6]{elmanto2022nilpotent}, the results of \cref{lem: Chow weight structure} equip $\Mod(\MGL_A)^\omega$ with a bounded weight structure. By the classification in \cite[Theorem 2.2.9]{elmanto2022nilpotent} of idempotent complete bounded weighted categories we deduce that $\Mod(\MGL_A)$ must be of the desired form.
\end{proof}
This result on the existence of a bounded weight structure on $\MGL$-modules, and its reduction to a vanishing result, has a nonequivariant precedent due to Bondarko and others. We refer to \cite[\S 4]{bondarko2016constructing} for an overview.
\subsection{Chow heart structure}
Using the cell structure on $\MGL_{A}$, we construct an approximation of the weight heart structure which now lives over $\mathbb{1}_{A}$. Instead of a weight structure, we now more naturally obtain a \emph{Chow heart structure} in the sense of \cite[\S 2]{saunier2023theorem}, but we will avoid this point of view for sake of simplicity. Morally the Chow weight structure exists because of the vanishing of Ext groups between perfect pure objects in $\Mod(\MGL_A)$. When working in $\SH^A(\C)^\cell$ these Ext groups only vanish locally for the topology generated by perfect pure epimorphisms.
\begin{thm}\label{thm: Chow heart structure}
    The inclusion of $\Pure_\C(A)$ induces an equivalence
    \[\Shv_{\Sigma}(\Pure_\C(A);\Sp)\simeq \SH^A(\C)^\cell,\]
    where the left hand side is the category of product-preserving presheaves that send cofibre sequences in $\Pure_\C(A)$ to fibre sequences in $\Sp$.
\end{thm}
\begin{remark}\label{rem: stable hull vs additive sheaves}
    The sheaf category appearing in the statement above may be interepreted in several ways. One may view $\Pure_\C(A)$ as an excellent ($\infty$-)site in the sense of \cite[\S 2]{pstragowski_synthetic_2022} and consider additive sheaves of spectra on it. Alternatively, note that the classes of pure monomorphisms and pure epimorphisms make $\Pure_\C(A)$ into an exact category in the sense of \cite{barwick2015exact} so that the left hand side is its \emph{presentable stable envelope} in the sense of \cite[Definition 2.12]{winges2025presentable}. See the proof of \cite[Theorem 2.7]{winges2025presentable} and  \cite[Proposition 4.4 ]{winges2025presentable} for the equivalence between the category of additive sheaves and and the presentable stable envelope.
\end{remark}
\begin{remark}\label{rem: pure extensions are MGL split} 
    Any extension $X\to Y\to Z$ in $\Pure_\C(A)$ is such that after tensoring with $\MGL_A$ it is classified by a class in $\pi^A_{-1,0}\MGL_A\otimes X\otimes Z^\vee$ which vanishes when one notes that perfect pure objects are closed under tensor products and duals and applies \cref{lem : Chow vanishing for perfect pure 2}. Any such extension therefore becomes split and the exact structure on $\Pure_\C(A)$ therefore becomes split exact after tensoring with $\MGL_A$.
\end{remark}
We now prove the crucial \emph{local} vanishing result that powers the existence of the Chow heart structure.
\begin{lemma}\label{lem: local vanishing}
    Let $m>0$ be a positive integer and let $f\colon \Sigma^{-m}X\to Y$ be a map of perfect pure spectra. Then there exists a pure epimorphism $p\colon X'\to X$ such that the composite $f\circ\Sigma^{-m}p$ vanishes.
\end{lemma}
The proof of \cite[Lemma 3.3.4]{haine2023spectral} goes through verbatim, but we spell it out for convenience.
\begin{proof}
    Consider the composite of $f$ with the unit map
    \[\Sigma^{-m}X\to Y\to Y\otimes \MGL_A\]
    which is null by \cref{rem: pure extensions are MGL split}. Since perfect pure motivic spectra are compact, we may find some $i,d$ in the colimit presentation from \cref{cor: Grass model for MGL} such that the composite
    \[\Sigma^{-m}X\to Y\to Y\otimes \Th_{\Gr^{d}(V_i)}(\overline{Q}^d_i)\]
    is already null. By dualising, this tells us that the composite
    \[\Sigma^{-m}(\Th_{\Gr_d(V_i)}(\overline{Q}_i^d)^\vee\otimes X)\to \Sigma^{-m}X\to Y\]
    is null as well, and we note that the unit map $\mathbb{1}_A\to\Th_{\Gr_d(V_i)}(\overline{Q}_i^d)^\vee$ is a pure monomorphism by \cref{ex: transition maps} so that the dual map $\Th_{\Gr_d(V_i)}(\overline{Q}_i^d)^\vee\otimes X\to X$ is a pure epimorphism and we may conclude.
\end{proof}
This is now enough to prove main result of this section.
\begin{proof}[Proof of \cref{thm: Chow heart structure}]
    The arguments of \cite[Theorem 3.3.5]{haine2023spectral} go through once again in this generality. The gist is that the inclusion of $\Pure_\C(A)$ preserves cofibre sequences hence gives rise to a left adjoint
    \[\Shv_\Sigma(\Pure_\C(A);\Sp)\to \SH^A(\C)^\cell.\]
    The left hand side is generated under colimits by objects of the form $\tau_{c\geq 0}\map(-,X)$ for $X$ perfect pure, where $\tau_{c\geq 0}$ denotes the connective cover in the sheafy t-structure on the left hand side (see \cref{cons: Chow t-structure} below). Checking that this left adjoint is fully faithful then reduces to showing that for all perfect pure $X$, the sheaf of spectra $\map(-,X)$ is connective in this $t$-structure. This follows immediately from \cref{lem: local vanishing} since it tells us that any negative degree homotopy elements can be killed locally. Essential surjectivity is then immediate since the target is generated under colimits by $\Pure_\C(A)$ per construction.
\end{proof}
\begin{construction}\label{cons: Chow heart structure}
    A posteriori, using \cite[Theorem 1.6]{saunier2025exact} we see that $\SH^A(\C)^{\cell,\omega}$ therefore admits a bounded heart structure whose heart is precisely $\Pure_\C(A)$, since it arises as the stable envelope of the latter. We will call this heart structure the Chow heart structure.
\end{construction}
\begin{remark}
    Since the inclusion of $\Pure_\C(A)$ into $\SH^A(\C)^\cell$ is symmetric monoidal, we see that the equivalence of \cref{thm: Chow heart structure} lifts to a strong symmetric monoidal equivalence if we equip the left hand side with the Day convolution symmetric monoidal structure (see \cite[Proposition 2.30]{pstragowski_synthetic_2022}).
\end{remark}
\subsection{Chow t-structure}
We can now equip $\SH^A(\C)^\cell$ with a novel t-structure coming from the equivalence in \cref{thm: Chow heart structure}.
\begin{construction}\label{cons: Chow t-structure}
    Under the equivalence $\SH^A(\C)^\cell\simeq \Shv_\Sigma(\Pure_\C(A);\Sp)$, let the Chow t-structure be the sheafy t-structure inherited from the standard t-structure on $\Sp$ (see \cite[Proposition 2.16]{pstragowski_synthetic_2022}). It is right complete, compatible with filtered colimits, and compatible with the symmetric monoidal structure. 
\end{construction}
\begin{remark}\label{rem:chars of chow t}
    The heart of the Chow t-structure is therefore given by the abelian category
    \[\SH^A(\C)^{\cell}_{c=0}\simeq \Shv_\Sigma(\Pure_\C(A);\mathrm{Ab})\]
    with its associated homotopy group functors
    \[\tau_{c=n}\colon    \SH^A(\C)^{\cell}\to \Shv_\Sigma(\Pure_\C(A);\mathrm{Ab}). \]
    In particular, we see that some $X\in \SH^A(\C)^\cell$ is
    \begin{enumerate}
        \item Chow connective if $\tau_{c=-n}X=0$ for all $n>0$,
        \item Chow coconnective if the sheaf of animæ $\Omega^\infty X$ is levelwise discrete.
    \end{enumerate}
\end{remark}
In fact, we can be more explicit about the connective objects; the connective part is generated under colimits by (suspension spectra of) representables, so that this corresponds to the unique t-structure (see \cite[1.4.4.11(2)]{HA}) on $\SH^A(\C)^\cell$ whose nonnegative part is generated under colimits by perfect pure motivic spectra. In particular, we see that this is compatible with the cellular Chow t-structure defined in \cite{bachmann2022chow}. Let us outline a few essential properties of the Chow t-structure that are immediate from the definition.
\begin{lemma}\label{lem: char of fib seq's}
    Let $p\colon X\to Y$ be a map of perfect pure motivic spectra with fibre $F$. Then $p$ is a pure epimorphism, i.e. $F$ is again perfect pure, if and only if the induced sequence
    \[\tau_{c=0}F\to \tau_{c=0}X\to \tau_{c=0}Y\]
    is a short exact sequence in $\SH^A(\C)^\cell_{c=0}$
\end{lemma}
\begin{proof}
    This is \cite[Lemma 2.19]{pstrkagowski2023perfect}.
\end{proof}
\begin{lemma}\label{lem: restriction is t-exact}
    The equivariant restriction and inflation functors between cellular $A$-equivariant motivic categories are right $t$-exact, furthermore restriction is t-exact.
\end{lemma}
\begin{proof}
    Both inflation and restriction clearly preserve colimits and the subcategory of perfect pure motivic spectra, so we conclude that both are right t-exact. Restriction is furthermore left t-exact since it admits a right t-exact left adjoint given by induction.
\end{proof}
\begin{lemma}\label{lem: tensor with pure is t-exact}
    Let $X$ be a perfect pure motivic spectrum, then the endofunctor $X\otimes-\colon $ of $\SH^A(\C)^{\cell}$ is t-exact.
\end{lemma}
\begin{proof}
    By compatibility with the symmetric monoidal structure and the fact that perfect pure motivic spectra are connective per construction, we see that $X\otimes-$ is right t-exact. On the other hand, suppose that $E$ is a Chow coconnective object of $\SH^A(\C)^\cell$ and let $Y$ be another perfect pure motivic spectrum. Then we may compute for all $k\geq 0$
    \[\map(\Sigma^kY,X\otimes E)\simeq \map(\Sigma^k(Y\otimes X^\vee),E)=0\]
    since $Y\otimes X^\vee$ is again perfect pure. Therefore $X\otimes E$ must also be Chow coconnective.
\end{proof}
\begin{cor}\label{cor: MGL t-exact}
    The endofunctor $\MGL_A\otimes-$  of $\SH^A(\C)^\cell$ is also t-exact.
\end{cor}
\begin{proof}
    Since the Chow t-structure is compatible with filtered colimits and $\MGL_A$ can be written as a filtered colimit of perfect pure motivic spectra by \cref{cor: cellularity for MGL} this follows immediately from \cref{lem: tensor with pure is t-exact}.
\end{proof}
As with any other t-structure on a sheaf category, recall that coconnectivity may be detected on the level of presheaves, while connectivity is more subtle. In particular, this tells us that bigraded homotopy groups need not interact nicely with connectivity.
\begin{remark}\label{rem: Chow ccon has ccon htpy}
    Let $E$ be Chow coconnective, then $\pi^{A}_{\ast,\ast}E$ is concentrated in nonpositive Chow degrees. Indeed, if $t-2w> 0$ then a class in $\pi^{A}_{t,w}E$ is represented by a map
    \[\Sigma^{t,w}\mathbb{1}_{A}\to E\]
    but the assumption $t>2w$ allows us to write this as a positive shift $\Sigma^{t-2w}\Th(\epsilon^w)$ of a perfect pure object so this must be null. 
\end{remark}
After we tensor with $\MGL_A$, the situation simplifies. Indeed, the induced Chow t-structure on $\Mod(\MGL_A)$ is now a t-structure on a category of additive presheaves so that both connectivity and coconnectivity are detected levelwise.
\begin{lemma}\label{lem: connecivity of MGL modules}
    Let $E$ be Chow connective, then $\MGL_A\otimes E$ has $A$-equivariant homotopy groups concentrated in nonnegative Chow degrees. If $E$ is Chow coconnective, then $\MGL_A\otimes E$ has $A$-equivariant homotopy groups concentrated in nonpositive Chow degrees.
\end{lemma}
\begin{proof}
    Note that the second statement is immediate from \cref{cor: MGL t-exact} and \cref{rem: Chow ccon has ccon htpy}. For the second statement, note that any Chow connective object is a (filtered) colimit of perfect pure objects and the result of \cref{lem : Chow vanishing for perfect pure 2} is compatible with (filtered) colimits.
\end{proof}
\begin{cor}\label{cor: Chow heart conc on Chow line}
    If $E$ is contained in the Chow heart, then $\MGL_A\otimes E$ has homotopy groups concentrated on the Chow line.
\end{cor}
\begin{cor}\label{cor: chow heart weight structure}
    Let $\Mod(\tau_{c=0}\MGL_A)^\omega$ denote the category of compact $\tau_{c=0}\MGL_A$-modules in $\SH^A(\C)^\cell$. This admits a bounded weight structure whose heart is an additive $1$-category.
\end{cor}
\begin{proof}
    The arguments of \cref{prop: Chow weight structure} go through again; consider the subcategory $\Pure(A;\tau_{c=0}\MGL)$, i.e. the subcategory on free $\tau_{c=0}\MGL_A$-modules on perfect pure motivic spectra. This once again satisfies the assumptions of \cite[Remark 2.2.6]{elmanto2022nilpotent}. However, we see that for $X,Y$ perfect pure one can easily compute the mapping spectrum between them as
    \begin{align*}
        \pi_\ast\Map(\mathbb{1}_A,\tau_{c=0}\MGL_A\otimes Y\otimes X^\vee))
        &\cong \pi_{\ast,0}^{A}(\tau_{c=0}\MGL_A\otimes Y\otimes X^\vee).
    \end{align*}
    Now $Y\otimes X^\vee$ is perfect pure so that the right hand side must have homotopy concentrated on the Chow line; in particular this vanishes unless $\ast=0$ in which case we recover the group $\MGL^{A}_{0,0}Y\otimes X^\vee$.
\end{proof}
In the next section, we will use this observation to give an explicit description of the heart of the Chow t-structure on $\SH^A(\C)^\cell$ itself.

\newpage
\section{The special fibre}\label{sec: special fibre}
We now identify what is to be thought of as the special fibre of the deformation encoded by cellular $A$-equivariant motivic spectra over $\C$: the heart of the Chow $t$-structure. As in \cite{bachmann2022chow}, the results in this section go through integrally.
\subsection{The heart of the Chow t-structure}
Consider the commutative algebra object $\tau_{c=0}\mathbb{1}_{A}$ in $\SH^A(\C)^{\cell}$ given by the truncation of the unit. We see that $\tau_{c=0}\MGL_A$ forms a commutative algebra over it and therefore obtain a chain of adjunctions
\[\Mod(\tau_{c=0}\mathbb{1}_{A})\leftrightarrows \coMod(C)\rightleftarrows \Mod(\tau_{c=0}\MGL_A)\]
where $C$ is short for the comonad on $\Mod(\tau_{c=0}\MGL_A)$ given by tensoring with $\tau_{c=0}\MGL_A$ over ${\tau_{c=0}\mathbb{1}_{A}}$.
\begin{lemma}
    All three categories participating in the chain of adjunctions above inherit t-structures from the Chow t-structure on $\SH^A(\C)^\cell$ such that all left adjoints above are t-exact.
\end{lemma}
\begin{proof}
    The existence of the inherited t-tructure on the two module categories is immediate since both are module categories over Chow connective algebra objects. For the comodule category, note that this is defined as a limit in $\PrL_\mathrm{st}$ of the form
    \[\coMod(C)=\lim_\Delta\Mod(\tau_{c=0}\MGL_A^{\otimes\bullet+1}).\]
    Now note that the limit only depends on the semicosimplicial structure maps, i.e. the subdiagram indexed by $\Delta^{\mathrm{inj}}$. The structure maps in this semicosimplicial diagram are given by tensoring with $\tau_{c=0}\MGL_A^{\otimes\bullet+1}$ so that these are t-exact by \cref{cor: MGL t-exact} and we may equip the limit with the inherited t-structure. In fact, from this definition it becomes clear that the left adjoints into and out of $\coMod(C)$ are then t-exact.
\end{proof}
\begin{remark}\label{rem: restricted comonad}
    By the result above, we obtain a restricted adjunction on hearts
    \[\Mod(\tau_{c=0}\mathbb{1})_{c=0}\rightleftarrows \Mod(\tau_{c=0}\MGL_A)_{c=0}\]
    with comonad $C_{c=0}$. The category of comodules for this restricted comonad is simply the heart of $\coMod(C)$
    \[\coMod(C_{c=0})\simeq \coMod(C)_{c=0}.\]
\end{remark}
\begin{lemma}\label{lem: conservativity}
    The left adjoint
    \[-\otimes\tau_{c=0}\MGL_A\colon \Mod(\tau_{c=0}\mathbb{1}_{A})_{c=0}\to \Mod(\tau_{c=0}\MGL_A)_{c=0}\]
    is conservative.
\end{lemma}
The proof of this crucially result uses the existence of an equivariant motivic Adams--Novikov spectral sequence with good convergence properties. We refer the reader to \cref{sec: ANSS convergence}. Having established this, the proof of \cite[Theorem 3.14 (2)]{bachmann2022chow} goes through rather directly and we provide a sketch of the argument here.
\begin{proof}
    Let $X$ in the source be such that $\tau_{c=0}\MGL_A\otimes_{\tau_{c=0}\mathbb{1}_A} X$ vanishes. Note that this basechange is equivalent to $\MGL_A\otimes X$. Let $d$ be any integer and consider the subcategory $\SH^A(\C)^\cell_{\geq d}\cap\SH^A(\C)^\cell_{c\geq 0}$ given by the Chow connective objects which are furthermore $d$-connective in the homotopy t-structure of \cref{def: homotopy t-structure}. For each $d$ these subcategories define (the connective parts of) unique t-structures with corresponding covers denoted $\tau_{c\geq 0}^d$. We note that as $d$ decreases, the truncation functors come with natural transformations $\tau_{c\leq 0}^{d}\to \tau_{c\leq 0}^{d-1}$ with colimit $\tau_{c\leq 0}$ by virtue of right completeness of the homotopy t-structure. Furthermore, for every Chow connective and homotopy $d$-connective $E$ the cofibre sequence $\tau_{c\geq 1}^dE\to E\to \tau_{c\leq 0}^dE$ shows that $\tau_{c\leq 0}^d E$ is again homotopically $d$-connective. Now let $X$ be in the Chow heart as above, and write $X$ as the filtered colimit of the diagram $\tau_{c\leq 0}^d X$. For every $d$, note that this term is bounded below. Furthermore, since $\eta$ is in Chow degree $-1$ this tells us that the map $\tau_{c\leq 0}^dX\to \tau_{c\leq 0}^dX^\wedge_\eta$ is an equivalence on homotopy groups. Indeed for any $n\geq 1$ we may consider the long exact sequence associated to the cofibre sequence $$\Sigma^{n,n}\tau_{c\leq 0}^dX\xrightarrow{\eta^n}\tau_{c\leq 0}^dX\to \tau_{c\leq 0}^dX/\eta^n,$$ now we claim that for all $d$ the pro-system $\{\pi_{\ast,\ast}^{A}\tau_{c\leq 0}^dX/\eta^n\}_n$ is pro-equivalent to $\pi_{\ast,\ast}^{A}\tau_{c\leq 0}^dX$. For this, consider a bidegree $(2w+i,w)$ with $i>0$ and note that if $w+i>d$ then this a positive suspension of a homotopically $d$-connective object hence cannot map nontrivially to $\tau_{c\leq 0}^d X$. For large enough $n$, we therefore see that $\Sigma^{n,n}\tau_{c\leq 0}^dX$ has homotopy groups in Chow degrees $\leq -n$ and the desired pro-equivalence holds. Finally, note that all arguments above go through after we replace $X$ with $\Sigma^{\cE}_+A/K\otimes X$ so that this equivalence on homotopy becomes a cellular equivalence. We may therefore conclude that every $\tau_{c\leq 0}^dX$ has a convergent equivariant Adams--Novikov spectral sequence with $E_2$-page \[\Ext^{\ast,\ast}_{\pi_{\ast,\ast}^{A}\MGL^{\otimes 2}}(\MGL^{A}_{\ast,\ast},\MGL^{A}_{\ast,\ast}\tau_{c\leq 0}^dX).\]
    By the same argument as in \cref{rem: Chow ccon has ccon htpy} we see that there can not be any nontrivial differentials and this spectral sequence must collapse on the $E_2$-page. One may again replace $X$ by $\Sigma^{\cE}_+A/K\otimes X$ to see that the assumption that $\MGL_A\otimes X$ vanishes therefore implies that $X$ is zero as a cellular equivariant motivic spectrum and we are done.
\end{proof}
\begin{cor}\label{cor: heart as comodules}
    The adjunction
    \[\Mod(\tau_{c=0}\mathbb{1}_{A})_{c=0}\leftrightarrows \coMod(C_{c=0})\]
    is an adjoint equivalence.
\end{cor}
\begin{proof}
    It suffices to show that the adjunction in the statement of \cref{lem: conservativity} is comonadic since $C_{c=0}$ is precisely the comonad of this adjunction. As ascertained above, the left adjoint is conservative. In order to prove the necessary commutation with totalisations, see the argument in \cite[Proposition 4.3, (2,b)]{bachmann2022chow} which uses that we are working with truncated categories whence the condition simplifies. We may then apply \cite[Theorem 4.7.3.5]{HA} to conclude.
\end{proof}
\begin{remark}\label{rem: comonad is cooperations}
    The comonad $C_{c=0}$ on $\Mod(\tau_{c=0}\MGL_A)_{c=0}$ can be identified explicitly. Since $C$ it is cocontinuous and in particular additive, it is entirely determined by its action on the weight heart of $\Mod(\tau_{c=0}\MGL_A)$ which was identified with an additive $1$-category in \cref{cor: chow heart weight structure}. On this $1$-category it is given by tensoring with the co-operation algebra $\pi^A_{\ast,\ast}\tau_{c=0}\MGL_A^{\otimes 2}$ since its homotopy groups are flat over the coefficients (see \cref{rem: regularity of co-ops}). In particular, we see that this comonad is associated to a flat Hopf algebroid in $\RU(A)$-graded $A$-Mackey functors.
\end{remark} 
\begin{construction}
    Let $c\colon \Sp^{A}\to \SH^A(\C)$ denote the unit functor from \cref{cor: unit map (additive)}. Per construction, this sends orbits to the corresponding equivariant motivic spectra which are perfect pure. This can further be base changed to a composite left adjoint
    \[\Sp^{A}\to \Mod(\tau_{c=0}\MGL_A)\]
    such that this functor is t-exact for the standard t-structure on the source and the Chow t-structure on the right hand side. On hearts, we obtain a colimit-preserving functor of abelian categories
    \[c^\ast\colon \Mack(A)\simeq \cP_\Sigma(\Span(\Fin_{A});\mathrm{Ab})\to \Mod(\tau_{c=0}\MGL_A)_{c=0}.\]
    In fact, for every virtual complex $A$-representation $V^\an$ the Thom spectrum $\Th(V)$ gives rise to a tensor-invertible object $\tau_{c=0}(\Sigma^V\MGL_A)$ in the target. We may therefore uniquely extend $c^\ast$ to a symmetric monoidal functor
    \[c^\star\colon \Mack(A)^{\RU}\to \Mod(\tau_{c=0}\MGL_A)_{c=0}\]
    from $\RU(A)$-graded $A$-Mackey functors.
\end{construction}
\begin{lemma}\label{lem: heart is Mackey}
    The functor $c^\star$ above is the left adjoint in a monadic adjunction with monad given by the $\RU(A)$-graded Green functor homotopy groups of $\MU$, $\underline{\pi}_\star^{A}\MU.$ Thus, it induces an equivalence
    \[\Mod(\tau_{c=0}\MGL_{A})_{c=0}\simeq \Mod(\underline{\pi}^{A}_\star\MU)\]
\end{lemma}
\begin{proof}
    Essentially per construction, the image of $c^\star$ contains a family of compact dualisable generators for the target abelian category so the adjunction must be monadic. In order to identify the monad, let $c_\star$ denote the right adjoint to $c^\star$. We may compute
    \begin{align*}
        c_\star(\tau_{c=0}\MGL_A)(A/K,V)&\cong \MGL^{A}_{0,0}\Sigma^{-V}_+A/K,\\
        &\cong \pi^{K}_{V^\an}\MU
    \end{align*}
    using \cref{prop: identification of the Chow line}. In fact, this is an identification of $\mathrm{RU}(A)$-graded Mackey functors since $c_\star$ is compatible with restrictions and transfers per construction.
\end{proof}
In particular, note that this abelian category is determined entirely in terms of non-motivic equivariant $\MU$. Using the description of the co-operations in \cref{rem: comonad is cooperations} as well as its explicit description in \cref{rem: cooperations on the Chow line}, we obtain a complete description of the heart of the Chow t-structure\footnote{In particular, the right hand side really is an abelian category whose coforgetful functor to modules is exact. This follows from t-exactness of the comonad exhibited in \cref{rem: restricted comonad} and is manifest in the flatness of the cooperation algebra exhibited in \cref{rem: comonad is cooperations}, cf. \cref{rem: regularity of co-ops}.}.
\begin{cor}\label{cor: identification of the heart}
    The functor
    \[\SH^{A}(\C)^\cell_{c=0}\to \coMod(\underline{\pi}_\star^{A}\MU^{\otimes 2})\]
    sending $X$ to $\tau_{c=0}(\MGL_{A}\otimes X)$ is an equivalence. 
\end{cor}
\begin{proof}
    First note that the forgetful functor $\Mod(\tau_{c=0}\mathbb{1}_{A})\to \SH^A(\C)^\cell$ induces an equivalence on hearts, using e.g. \cite[Lemma 29, (2)]{bachmanntambara}. We can therefore use \cref{cor: heart as comodules} and \cref{lem: heart is Mackey} to conclude.
\end{proof}
We therefore come to the remarkable conclusion that the heart of the Chow t-structure on $\SH^A(\C)^\cell$ is defined entirely in terms of non-motivic data. In fact, we see that the truncation of a perfect pure object is already purely topological.
\begin{remark}\label{rem: sf is topological}
    Let $X$ be a perfect pure motivic spectrum, then $\tau_{c=0}X$ can be identified with the comodule corresponding to the Chow degree zero $\MGL_A$-homology of $X$. Now every extension of perfect pure motivic spectra becomes split in $\MGL_A$ by \cref{rem: pure extensions are MGL split}. Therefore, this comodule is just an extension of a direct sum of free comodules on Thom classes and orbits. Since we are in Chow degree zero, this data is entirely captured on the topological side by the equivalence induced by Betti realisation.
    In summary for $X$ perfect pure, there is an identification of comodules
    \[\tau_{c=0}X\simeq \underline{\pi}_\star^{A}(\MU_{A}\otimes \Be^{A}X).\]
\end{remark}
In fact, one can refine this identification of the heart to an identification of the module category $\Mod(\tau_{c=0}\mathbb{1}_{A})$. Indeed, per construction the latter is compactly generated and the compact objects identify with the thick subcategory generated by objects of the form $\tau_{c=0}\mathbb{1}_A\otimes X\simeq \tau_{c=0}X$ for $X$ perfect pure. 
\begin{defn}
    Let $\cD^b_\Pure(\coMod(\underline{\pi}_\star^{A}\MU^{\otimes 2}))$ denote the thick subcategory of the derived category of $\coMod(\underline{\pi}_\star^{A}\MU^{\otimes 2})$ generated by objects of the form $\underline{\pi}_\star^A(\MU_A\otimes\Sigma^V_+A/K)$ for virtual $A$-representations $A$ and orbits $V$. 
\end{defn}
\begin{cor}\label{cor: identification of SF}
    There is a symmetric monoidal equivalence
    \[\Mod(\tau_{c=0}\mathbb{1}_{A})\simeq \Ind\cD^b_\Pure(\coMod(\underline{\pi}_\star^{A}\MU^{\otimes 2}))\]
\end{cor}
The arguments in \cite[Proposition 4.18]{bachmann2022chow} go through, and we replicate them here for convenience.
\begin{proof}
    The left hand side is compactly generated by the thick subcategory generated by Chow truncations of perfect pure objects, so it suffices to exhibit an equivalence
    \[\Mod(\tau_{c=0}\mathbb{1}_A)^\omega\simeq \cD^b_\Pure(\coMod(\underline{\pi}_\star^{A}\MU^{\otimes 2})).\]
     By \cref{cor: identification of the heart} we are therefore showing that the left hand side is the bounded derived category of its heart. Following the dual of \cite[Proposition 1.3.3.7]{HA}, it suffices to show that the heart has enough injective objects and that these have no higher $\Ext$ groups. Since we are working in a category of comodules over a comonad (in fact, a flat Hopf algebroid), the cofree comodule functor provides us with enough injectives coming from the category of modules over the Green functor $\underline{\pi}_\star^A\MU$, which itself is the heart of a category with bounded weight structure (\cref{cor: chow heart weight structure}) hence has enough injectives since it is equivalent to a category of additive presheaves of abelian groups. Since the injectives arise as cofree comodules on injective modules the Ext vanishing condition is immediate.
\end{proof}
\begin{remark}\label{rem: not really stable comodules}
    While the special fibre is entirely determined in terms of comodule-theoretic information, the answer above is not entirely satisfactory. Indeed, in the non-equivariant synthetic story, the special fibre may be identified with Hovey's stable derived category of (even) Adams--Novikov comodules (\cite[Proposition 4.53]{pstragowski_synthetic_2022}). This uses that $\MU$ is Landweber exact, since it makes recourse to the Landweber filtration (see \cite{hovey2004homotopy} and \cite[Lemma 3.5]{barthel2018algebraic}). We expect that the special fibre above arises as the stable derived category of an $\RU(A)$-graded variant of the moduli stack of equivariant formal groups of \cite{hausmann2023invariant} which should further admit a Landweber filtration. In particular, this would allow us to identify $\cD^b_\Pure(\coMod(\underline{\pi}_\star^{A}\MU^{\otimes 2}))$ with the perfect derived category of this stack. This question will be revisited this in future work. Structural identifications aside, we note that the primary comodule of interest for us is the unit comodule $\MU^{A}_\star$ whose endomorphisms recover the $\mathrm{E}_2$-page of the equivariant Adams--Novikov spectral sequence, and it is tautologically clear how to produce this element in the special fibre.
\end{remark}

\newpage
\section{Synthetic reconstruction}\label{sec: synthetic reconstruction}
In this section, we bring together all of the machinery established above to obtain our synthetic comparison result. Let us begin by introducing a synthetic deformation of $\Sp^{A}$. A more careful analysis of its properties, as well as a justification of why it deserves the title of equivariant synthetic spectra, will be provided in \cref{sec: filtered reconstruction,sec:featuresofsynthetic}. In the following, $A$ denotes a finite abelian group.
\subsection{Synthetic equivariant spectra}
\begin{defn}\label{def: topological ppure}
    Define the category $\Pure(A)$ of \emph{perfect pure $A$-spectra} to be the minimal subcategory of $\Sp^{A}$ containing all objects of the form $\Sigma^V_+A/K$ for virtual complex $A$-representations $V$ and subgroups $K$ and which is closed under extensions and retracts. 
\end{defn}
\begin{remark}
    This can again be equipped with an exact structure inherited from the inclusion, in which an epimorphism is a map whose fibre is again perfect pure and monomorphisms are dual to epimorphisms.
\end{remark}
\begin{remark}
At the trivial group, we see that $\Pure(\{e\})$ recovers the category of perfect even modules $\mathrm{Perf}_\even(\mathbb{S})$ considered in \cite{pstrkagowski2023perfect}, with the same exact structure.
\end{remark}
\begin{defn}\label{def: equivariant synthetif}
    Define the category of $A$-equivariant synthetic spectra $\Syn^{A}$ as the presentable stable hull
    \[\Syn^{A}=\Shv_\Sigma(\Pure(A);\Sp),\]
    i.e. the category of those additive presheaves that take cofibre sequences in $\Pure(A)$ to fibre sequences in $\Sp$.
\end{defn}
\begin{remark}\label{rem: synthetic sheaf condition}
    The remarks of \cref{rem: stable hull vs additive sheaves} go through here as well, where we may view this as a category of additive sheaves. In particular, $\Syn^A$ comes equipped with a preferred presentably symmetric monoidal structure extending the one on $\Pure(A)$.
\end{remark}
\begin{definition}\label{rem: synthetic t-structure}
    As a sheaf category, $\Syn^{A}$ also admits a t-structure coming from the standard t-structure on $\Sp$, and we will also denote it by $\tau_{c\geq\ast}$.
\end{definition}
Betti realisation then provides us with an immediate comparison functor between motivic and synthetic spectra.
\begin{lemma}\label{lem: setup of comparisoon functor}
    Equivariant Betti realisation restricts to a symmetric monoidal functor of exact categories
    \[\Be^A\colon \Pure_\C(A)\to\Pure(A),\]
    giving rise to a symmetric monoidal left adjoint
    \[L\colon\SH^A(\C)^\cell\to \Syn^{A}.\]
\end{lemma}
\begin{proof}
    It is clear from the definition of either side that Betti realisation restricts to a functor between the perfect pure categories. To see that it is a functor of exact categories, note that it preserves cofibre sequences hence sends a pure epimorphism in the source to a pure epimorphism in the target. Since it is symmetric monoidal, it does the same for pure monomorphisms.
\end{proof}
\subsection{Covering lifting property}
The restricted Betti realisation functor furthermore has excellent site-theoretic properties; we saw that it preserves covers, we will show that it furthermore detects covers and that covers can be lifted. The key tool is the identification of the special fibre: just as perfect pure motivic spectra have well behaved $\MGL_{A}$-homology, we can say the same in the topological setting, using that $\MU_{A}$ is concentrated in even degrees.
\begin{lemma}\label{lem: topological epi gives ses}
    Let $p\colon X\to Y$ be a map in $\Pure(A)$ with fibre $F$. If $p$ is a pure epimorphism, i.e. $F$ is again perfect pure, then we obtain a short exact sequence
    \[0\to \underline{\pi}^A_\star\MU_{A}\otimes F\to\underline{\pi}^A_\star\MU_{A}\otimes X
    \to \underline{\pi}^A_\star\MU_{A}\otimes Y\to 0\]
    of $\RU(A)$-graded Mackey comodules.
\end{lemma}
\begin{proof}
    Since the homotopy groups of $\MU_{A}$ are concentrated in even degrees by \cite[Theorem 5.3]{comezana1996calculations}, we see that the same is true for the $\MU_{A}$-homology of a $A$-spectrum of the form $\Sigma^V_+A/K$ by applying the Thom isomorphism and passing to a subgroup. This property is closed under extensions and retracts, so if $F$ is perfect pure then then boundary homomorphism in the long exact sequence associated to $F\to X\to Y$ must vanish.
\end{proof}
\begin{cor}\label{cor: Betti reflects coverings}
    The restricted Betti realisation functor
    \[\Be^A\colon \Pure_\C(A)\to \Pure(A)\]
    reflects epimorphisms.
\end{cor}
\begin{proof}
    There is a commutative diagram coming from \cref{rem: sf is topological}
    \[\begin{tikzcd}
        \Pure_\C(A)\arrow[d, "\Be^A"]\arrow[r]&\SH^A(\C)^{\cell}_{c=0}\arrow[d, "\simeq"]\\
        \Pure(A)\arrow[r]&\coMod(\underline{\pi}^A_\star\MU^{\otimes 2})
    \end{tikzcd}\]
    where the top horizontal arrow is induced by taking the Chow coconnective cover -- or equivalently, taking $\tau_{c=0}\MGL_{A}$-homology -- and the bottom horizontal arrow is induced by taking $\MU_{A}$-homology. The right vertical arrow is the equivalence of \cref{cor: identification of the heart}.
    Let $p\colon X\to Y$ be a map of perfect pure motivic spectra with fibre $F$, then if $\Be^Ap$ is a pure epimorphism it must induce a short exact sequence in the bottom right vertex by \cref{lem: topological epi gives ses}. Equivalently, this means that $p$ gave rise to a short exact sequence on Chow connective covers so that by \cref{lem: char of fib seq's} it must have been a pure epimorphism. 
\end{proof}
We now establish a second technical tool, namely that of a common envelope for a map of excellent sites in the sense of \cite[Definition 2.39]{pstragowski_synthetic_2022}. The idea of this construction is to monomorphically embed objects in the source and target of the restricted Betti realisation functor into a sufficiently large and well-behaved Ind-object so that one can lift covers from $\Pure(A)$ to $\Pure_\C(A)$.
\begin{lemma}\label{lem: common envelope}
    The restricted Betti realisation functor
    \[\Be^A\colon \Pure_\C(A)\to \Pure(A)\]
    admits a common envelope.
\end{lemma}
\begin{construction}
    Let $\{M(\alpha)\}_\alpha$ be an Ind-object in $\Pure_\C(A)$ such that its colimit $M$ in $\SH^A(\C)^{\cell}$ can be identified as
    \[M\simeq \bigoplus_{V,K}\MGL_{A}\otimes \Sigma^{V}_+A/K\]
    where the sum is such that every virtual $A$-representation $V$ appears countably infinitely many times and every subgroup $K$ of $A$ appears countably infinitely many times as well. Note that such an Ind-object certainly exists: for every $K$ one can let the index $V$ run through a saturated flag of a virtual $A$-universe and tensor this with the Gra{\ss}mannian formula for $\MGL_{A}$. Then taking the filtered colimit iterated extensions over an indexing set in which every subgroup 
    $K$ of $A$ appears infinitely many times gives us a model for such an $M$.
\end{construction}
\begin{proof}[Proof of \cref{lem: common envelope}]
    We proceed to verify the conditions in \cite[Definition 2.39]{pstragowski_synthetic_2022}, cf. \cite[Remark 2.40]{pstragowski_synthetic_2022} for Ind-object $\{M(\alpha)\}_\alpha$ defined above, which in our case take the following form.
    \begin{enumerate}
        \item The presheaf (of sets) on $\Pure_\C(A)$ defined by $$X\mapsto\varinjlim_\alpha\pi_0\map(X,M(\alpha))$$ is a sheaf.
        \item The presheaf (of sets) in $\Pure(A)$ defined by  $$Y\mapsto\varinjlim_\alpha\pi_0\map(Y,\Be^AM(\alpha))$$ is a sheaf.
        \item For all $X$ in $\Pure_\C(A)$, the natural map
        \[\varinjlim_\alpha\pi_0\map(X,M(\alpha))\to\varinjlim_\alpha\pi_0\map(\Be^AX,\Be^AM(\alpha))\]
        is an equivalence.
        \item For any $Y\in \Pure(A)$ there exists a map $Y\to \Be^AM$ such that the sub diagram on indices $\alpha$ for which the map $Y\to \Be^AM(\alpha)$ is a pure monomorphism is cofinal.
    \end{enumerate}
    In order to verify the first condition, note that this presheaf is additive and sends $X$ to $\pi_{0,0}^AM\otimes X^\vee$. Since $M$ is a sum of of tensors of $\MGL_A$ by perfect pure objects, we see that any $M\otimes X^\vee$ has homotopy concentrated in nonnegative Chow degree, so that any fibre sequence in $\Pure_\C(A)$ gets sent to a short exact sequence which suffices to prove that it is a sheaf following \cite[Theorem 2.8]{pstragowski_synthetic_2022}. In fact, the exact same argument for the second condition by means of \cref{lem: topological epi gives ses}. The third point then follows from this description and \cref{rem: sf is topological}. Finally, note that any $Y\in \Pure(A)$ admits a pure monomorphism into $\Be^AM$. This can be constructed by means of the unit map
    \[Y\to Y\otimes \Be^AM\simeq \Be^AM.\]
    Indeed, from the construction and \cref{ex: transition maps} we see that the structure maps in the filtered diagram $\Be^AM(\alpha)$ are all pure monomorphisms since this is true for $M(\alpha)$ and Betti realisation preserves pure monomorphisms. The final identification $Y\otimes \Be^AM\simeq \Be^AM$ then follows from the definition of $M$: every cofibre sequence in $\Pure(A)$ becomes split after tensoring with $\MU_{A}$ so that $Y\otimes \Be^AM$ can be identified with a sum of twists of $\Be^AM$ by complex representation spheres and orbits, which is itself again equivalent to $\Be^AM$.
\end{proof}
\begin{cor}\label{cor: clp}
    The restricted Betti realisation functor
    \[\Be^A\colon \Pure_\C(A)\to \Pure(A)\]
    has the covering lifting property.
\end{cor}
\begin{proof}
    Using \cite[Proposition 2.43]{pstragowski_synthetic_2022} this follows from \cref{lem: common envelope} and \cref{cor: Betti reflects coverings}.
\end{proof}
\subsection{Perfect pure reconstruction}
In this section we come to the key computation that allows us to identify the two perfect pure subcategories mod $p$ for an arbitrary prime. 
\begin{thm}\label{thm: pure reconstruction}
    Let $p$ be an arbitrary prime, then the restricted Betti realisation functor
    \[\Be^A\colon \Pure_\C(A)\to \Pure(A)\]
    is fully faithful mod $p$.
\end{thm}
\begin{remark}
    Note that both sides are viewed as enriched in spectra in order to make sense of the mod $p$ mapping spectra.
\end{remark}
Note that both categories are closed under taking duals and tensor products, which $\Be^A$ preserves, so \cref{thm: pure reconstruction} can be reduced to the following by setting $w=0,t\geq 0$.
\begin{propn}\label{prop: equivariant GI}
    Let $X$ be a perfect pure motivic spectrum, then the Betti realisation map
    \[\pi^A_{t,w}X/p\to \pi^{A}_t\Be^AX/p\]
    is an isomorphism whenever $t-2w\geq 0$
\end{propn}
This passes through two intermediate results; first we show that a weaker statement is true for all perfect pure $X$, and then we refine this to the desired statement using the Adams--Novikov spectral sequence.
\begin{construction}
    Let $\cC_{A}\subset \SH^A(\C)^\cell$ be the full subcategory on those $E$ admitting an integer $k_E$ such that the map
    \[\pi_{t,w}^{A}E/p\to \pi_t^{A}\Be^AE/p\]
    is an equivalence for $t-2w\geq k_E$. 
\end{construction}
\begin{remark}\label{rem: closure properties}
    This category $\cC_{A}$ is clearly thick and closed under bigraded (de)suspensions, but it is not closed under all filtered colimits: given some sequential diagram $E_\alpha$ in $\cC_{A}$ it might be that the integers $k_{E_\alpha}$ tend to infinity. If this sequence can be uniformly bounded, then this colimit clearly remains in $\cC_{A}$. The reason for working with this slightly awkward category is that any version where the possible indices $k_E$ are uniformly bounded below would no longer be closed under fibres.
\end{remark}
\begin{propn}\label{prop: pure is in A}
    There is an inclusion $\Pure_\C(A)\subset \cC_{A}$.
\end{propn}
    Using the closure properties in \cref{rem: closure properties} we note the following: it suffices to prove that objects of the form $\Sigma^\cE_+A/K$ are contained in $\cC_{A}$. In fact, since this is a condition on fixed points we can make a further reduction; by induction on the cardinality of $A$. The base case of the trivial group is covered by the nonequivariant result \cite[Theorem 7.30]{pstragowski_synthetic_2022}. if $A/K$ is an orbit corresponding to a proper subgroup, then the perfect pure motivic spectrum $\Sigma^\cE_+A/K$ is in the image of the (co-)induction functor from $K$ so we may reduce to lower ranks. In order to handle the $A$-fixed points of a perfect pure motivic spectrum of the form $\Th(\cE)$ we appeal to induction again; using \cref{ex: horbits of Thom spectrum} we see that it suffices for the associated graded pieces $\Th(\cE)\otimes\E[K]$ of the filtration by adjacent isotropy to be contained in $\cC_{A}$. Taking geometric fixed points for a proper subgroup $K$ lands us in the case of a group $W=A/K$ of strictly lower rank so we again only need to work at the top group which is covered by the lemma below.
\begin{lemma}\label{lem: clspace is in A}
    For any virtual $A$-representation $\cE$, the homotopy orbits $\Th(\cE)_{\mathrm{h}A}$ are contained in $\cC_{\{e\}}$.
\end{lemma}
\begin{proof}
    Recall from \cref{ex: horbits of Thom spectrum} that the homotopy orbits are computed by
    \[\Th(\cE)_{\mathrm{h}A}\simeq \Th_{\B_{\mot}A}(\widetilde{\cE}),\]
    i.e. the Thom spectrum of the associated bundle. Since the formation of classifying spaces as well as the Thomification process preserves finite products and $\cC_{\{e\}}$ is closed under these, we may take $A$ to be a cyclic group $\mu_n$.
    Next, we see that the motive of $\B_\mot\mu_n$ participates in a purity cofibre sequence of the form
    \[\B_\mot\mu_{n+}\to \P^\infty_+\to \Th_{\P^\infty}((Q^1_\infty)^{\otimes n}).\]
    This purity cofibre sequence can be deduced from the unstable fibre sequence \[\B_\mot \mu_n\to \B_\mot T\xrightarrow{[n]} \B_\mot T\] which itself arises from the short exact sequence of fppf group schemes
    \[0\to \mu_n\to T\xrightarrow{[n]}T\to 0.\]
    Indeed, it suffices to identify $\mathbb{V}((Q^1_\infty)^{\otimes n})\setminus 0$ with $\B_\mot\mu_n$, but since these punctured total spaces of vector bundles are pulled back from the universal case, this can be deduced from the universal identification $\mathbb{V}(Q^1_\infty)\setminus 0\simeq \A^\infty_{\P^\infty}\setminus \P^\infty\simeq\P^\infty$ and the aforementioned unstable fibre sequence. Furthermore, the purity cofibre sequence can be appropriately Thomified for any virtual bundle on $\B_\mot \mu_n$ arising as the associated bundle of a $\mu_n$-representation $\cE$: any such $\cE$ can be lifted to a $T$-representation $\cG$ such that the restriction of the associated bundle $\widetilde{\cG}$ on $\P^\infty$ to $\B_\mot \mu_n$ recovers $\widetilde{\cE}$. In summary, we have reduced to proving that object of the form $\Th_{\P^\infty}(\widetilde{\cG})$ are contained in $\cC_{\{e\}}$.
    This now follows from the non-equivariant version of \cref{prop: equivariant GI}, where we note that such nonequivariant motivic spectra $\Th_{\P^\infty}(\widetilde{\cG})$ arise as filtered colimits of (Thom twists of) finite $\MGL$-projective spectra $\P^n$ in the sense of \cite[Definition 7.6]{pstragowski_synthetic_2022}.
\end{proof}
\subsection{Independence of weight}
We now refine the statement of \cref{prop: pure is in A} by proving that for any perfect pure $X$ one may take $k_X=0$ without loss of generality. This follows from the following observation
\begin{propn}\label{prop: indep of weight}
Let $\tau\in \pi_{0,-1}^{A}\mathbb{1}_A/p$ be the class formerly labeled $\tau_1$ in \cref{rem: tau classes}. Then for all perfect pure $X$, multiplication by $\tau$ induces an isomorphism
\[\pi_{t,w}^{A}X/p\cong \pi_{t,w-1}^{A}X/p\]
whenever $t-2w\geq 0$.
\end{propn}
 Using the long exact sequence in homotopy associated to the cofibre sequence
\[\Sigma^{0,-1}X/p\xrightarrow{\tau}X/p\to X/(p,\tau)\]
it suffices to show that the homotopy groups of $X/(p,\tau)$ are concentrated in nonpositive Chow degrees. In fact, by \cref{rem: Chow ccon has ccon htpy} it suffices to prove that this object is Chow coconnective. We can now reduce this to a question about the Adams--Novikov spectral sequence.
\begin{construction}\label{cons: comparison map}
    Let $X$ be a perfect pure motivic spectrum. The composite
    \[\Sigma^{0,-1}X/p\simeq \Sigma^{2}\Sigma^{-\epsilon}X/p\xrightarrow{\tau}X\to\tau_{c=0}X/p\]
    goes from a Chow $2$-connective object to a Chow coconnective object hence admits a unique factorisation through a map
    \[X/(p,\tau)\to \tau_{c=0}X/p.\]
\end{construction}
\begin{lemma}\label{lem: Chow ccon of X/p}
    The map $X/(p,\tau)\to \tau_{c=0}X/p$ constructed above is an equivalence in $\SH^A(\C)^\cell$.
\end{lemma}
\begin{proof}
    Using \cref{cor: heart as comodules} we see that the right hand side is entirely determined by its $\underline{\pi}^A_\star\MU^{\otimes 2}$-comodule and hence by the motivic Adams--Novikov spectral sequence for this comodule, which must collapse for degree reasons. In fact, the left hand side also has a converging Adams--Novikov spectral sequence up to $\eta$-completion by \cref{thm: convergence} whose $E_2$-page can be computed explicitly. Indeed, the computation of \cref{prop: mod p^i equivariant HKO} tells us that $\pi_{\ast,\ast}^{A}\MGL_{A}\otimes X/(p,\tau)$ is concentrated on the Chow line, where it is isomorphic to the Adams--Novikov comodule of the target, and idem for its twists by orbits and representation spheres. Since $X/(p,\tau)$ and its twists are is strongly dualisable, they are alredy $\eta$-complete by \cref{cor: unit is eta complete}. Further, since its comodule is concentrated on the Chow line, its motivic Adams--Novikov spectral sequence must collapse on the $E_2$-page so that map $X/(p,\tau)\to \tau_{c=0}X/p$ induces an equivalence on $A$-equivariant homotopy groups. By replacing $X$ with a perfect pure motivic spectrum of the form $X\otimes \Sigma^{\cE}_+A/K$ we see that this map is therefore a cellular equivalence.
\end{proof}
\begin{cor}
    For any perfect pure motivic spectrum $X$, the cofibre sequence
    \[\tau_{c\geq 1}X/p\to X/p\to \tau_{c=0}X/p\]
    coming from the Chow t-structure may be identified with the cofibre sequence
    \[\Sigma^{0,-1}X/p\xrightarrow{\tau}X/p\to X/(p,\tau)\]
    coming from $\tau$-multiplication.
\end{cor}
We may now bring together these two results to finish the proof the main computational result of this section.
\begin{proof}[Proof of \cref{prop: equivariant GI}]
    Combine \cref{prop: pure is in A} and \cref{prop: indep of weight}.
\end{proof}
This may now be combined with our previously established structural results to obtain the main result.
\begin{thm}\label{thm: synthetic reconstruction}
    The symmetric monoidal left adjoint
    \[L\colon \SH^A(\C)^\cell\to \Syn^{A}\]
    constructed in \cref{lem: setup of comparisoon functor} induces an equivalence on $p$-complete subcategories for any prime $p$.
\end{thm}
\begin{proof}
    Using \cref{thm: Chow heart structure}, write the source and target as sheaf categories
    \[L\colon \Shv_\Sigma(\Pure_\C(A);\Sp)\to \Shv_\Sigma(\Pure(A);\Sp)\]
    and let $R$ denote the right adjoint of $L$. Upon restriction to subcategories of $p$-complete objects, $R$ restricts to such a right adjoint while $L$ is to be replaced with the composite $L_p=(-)^\wedge_p\circ L$. We begin by showing that $L_p$ is fully faithful, or equivalently that the unit $\id\to RL_p$ is a $p$-complete equivalence on $\SH^A(\C)^\cell$. First note that $R$ preserves colimits by \cref{cor: clp} so that we may restrict our attention to the colimit generators of the source, i.e. for any perfect pure motivic spectrum $X$ we want to prove that the unit map
    \[X\to L_p(R(X/p))\]
    is an equivalence. In terms of sheaves on $\Pure_\C(A)$, this is asking whether the map 
    \[\tau_{c\geq 0}\map(-,X/p)\to \tau_{c\geq 0}\map(\Be^A-,\Be^AX/p)\]
    is an equivalence for all perfect pure $X$. Since perfect pure motivic spectra are closed under tensor products and duals, we see that this is equivalent to asking whether for all $t\geq 0$ the natural map
    \[\pi_{t,0}^{A}X/p\to \pi_t^{A}\Be^AX/p\]
    is an equivalence. This follows from \cref{prop: equivariant GI}. In order to see that $L_p$ is essentially surjective, note that it preserves colimits and is fully faithful so it suffices to note that it hits a family of compact generators for the target which follows from \cref{thm: pure reconstruction}.
\end{proof}

\newpage
\section{Features of equivariant synthetic spectra}\label{sec:featuresofsynthetic}
In this section and the next, we shift our attention to the category of synthetic $A$-spectra. We begin by defining a number of important functors, collect some results that describe $\Syn^A$ more explicitly, and identify one part of the deformation span it sits in.
\subsection{The synthetic analogue}
Recall that $\Syn^{A}$ comes with a sheafy t-structure denoted $\tau_{c\geq \ast}$ in analogy with the Chow t-structure. We will refer to this as the \emph{synthetic t-structure}. 
\begin{remark}\label{rem:basic facts about syn t structure}
    As usual, this synthetic t-structure is right complete and compatible with the symmetric monoidal structure and filtered colimits by \cite[Proposition 2.16]{pstragowski_synthetic_2022}.
\end{remark}

The synthetic t-structure allows us to define an equivariant version of the synthetic analogue.
\begin{defn}
    The \emph{equivariant synthetic analogue} is the functor
    \begin{equation*}
        \begin{split}
            \nu\colon \Sp^{A} &\to \Syn^{A} \\ 
        X &\mapsto \tau_{c\geq 0}\Map(-,X),
        \end{split}
    \end{equation*}
    while the \emph{equivariant spectral Yoneda embedding} is the functor
    \begin{equation*}
        \begin{split}
            Y\colon \Sp^{A} &\to \Syn^{A} \\
            X &\mapsto \Map(-,X).
        \end{split}
    \end{equation*}
\end{defn}
Let us collect some key properties of the functors $\nu$ and $Y$ that follow from general considerations.
\begin{lemma}\label{lem: properties of nu}
    The functors $\nu, Y\colon \Sp^{A}\to \Syn^{A}$ satisfy the following:
    \begin{enumerate}
        \item \label{item:nucont}The functor $\nu$ commutes with filtered colimits.
        \item \label{item:Ycont}The functor $Y$ commutes with all limits and colimits.
        \item \label{item:symmon}Both $\nu$ and $Y$ obtain a lax symmetric monoidal structure. 
        \item \label{item:strongsymmon}The restriction of $\nu$ to $\Pure(A)$ is strong symmetric monoidal. In fact, if $X\in \Pure(A)$ and $Z\in \Sp^{A}$ is arbitrary, then the map
        \[\nu X\otimes \nu Z\to \nu(X\otimes Z)\]
        is an equivalence.
    \end{enumerate}
\end{lemma}
\begin{proof}
    \cref{item:nucont,item:symmon} follow from \cite[Lemma 4.4]{pstragowski_synthetic_2022}, and \cref{item:strongsymmon} follows from \cite[Lemma 4.5]{pstragowski_synthetic_2022} and \cite[Lemma 4.24]{pstragowski_synthetic_2022}. For statement \ref{item:Ycont}, note that limits in $\Syn^A$ are computed levelwise, and since $\Map(-, X)$ also preserves limits one can check that $Y$ commutes with all limits. Thus $Y$ also commutes with finite colimits; to conclude that it commutes with all colimits it suffices to show that it preserves filtered colimits. But \cite[Corollary 2.9]{pstragowski_synthetic_2022} tells us that filtered colimits in $\Syn^A$ are also computed levelwise, and since all $A$-spectra in $\Pure(A)$ are compact, the same argument as for limits tells us that $Y$ commutes with filtered colimits. 
\end{proof}
We will consistently borrow motivic notation for some objects in the synthetic category.
\begin{notation}
    The tensor unit $\nu \mathbb{1}_A$ of $\Syn^{A}$ will be denoted $\mathbb{1}_A^\Syn$. More generally, given a virtual complex representation sphere $S^V$ in $\Sp^{A}$, we write $\Th(V)$ for the Picard element $\nu S^V$.
    \footnote{Note that $\Th(V)$ really is tensor-invertible by \cref{lem: properties of nu}.}
\end{notation}
One can check that under the motivic-synthetic equivalence of \cref{thm: synthetic reconstruction}, these objects correspond to each other (after $p$-completion).
\subsection{Equivariant lifts of tau}\label{ssec: eq lifts of tau}
In this section, we construct equivariant analogues of the $\tau$ map from synthetic homotopy theory.
\begin{definition}
    For some perfect pure $P$ we denote the sections over $P$ by 
    \begin{equation*}
        \begin{split}
            \Gamma(P;-)\colon \Syn^{A} &\to \Sp \\
            F &\mapsto F(P).
        \end{split}
    \end{equation*}
\end{definition}
This definition of the sections functor $\Gamma(P;-)$ is unsatisfactory: one would like to be able to recover the perfect even filtration on a $A$-spectrum in terms of the sections of its synthetic analogue, but the latter simply lands in $\Sp$ as constructed. We therefore define an enhanced analogue that remembers sections over all orbits along with restriction and transfer maps that allow us to lift to an $A$-spectrum.
\begin{construction}
    Note that $\Sp^{A}$ admits a Mackey weight structure coming from the tom Dieck splitting. Let $\cB(A)\subset \Sp^{A}$ denote the minimal full subcategory containing all orbits $(A/K)_+$ and closed under extensions and retracts. Then the inclusion $\cB(A)\subset (\Sp^{A})^\omega$ satisfies the hypotheses of \cite[Remark 2.2.6]{elmanto2022nilpotent}. In particular, \cite[Theorem 2.2.9]{elmanto2022nilpotent} gives an equivalence\footnote{Of course, closer analysis of $\cB(A)$ reveals that it is in fact merely a $(2,1)$-category and furthermore equivalent to $\Span(\Fin_{A})$, cf. \cite{guillou2024models, barwick2017spectral}.}
    \[\cP_\Sigma(\cB(A);\Sp)\simeq \Sp^{A}.\]    
\end{construction}
\begin{defn}\label{def: enhanced sections}
    Let $V$ denote a virtual complex $A$-representation. Then we obtain a functor of the form
    \[S^V\otimes-\colon \cB(A)\to \Pure(A).\]
    The \emph{enhanced sections functor} $\widetilde{\Gamma}(S^V;-)$ at $V$ is obtained as the restriction functor
    \[ \widetilde{\Gamma}(S^V;-):\Syn^{A}\simeq \Shv_\Sigma(\Pure(A);\Sp)\xrightarrow{(S^V\otimes-)^*}\Shv_\Sigma(\cB(A);\Sp)\simeq \Sp^{A}.\]
\end{defn}
\begin{remark}
    In order to verify that this is well defined, we need to make sure that $S^V\otimes-$ as defined above is a functor of exact categories. It is clearly additive, and in fact the exact structure on the source it inherits from the target degenerates to a split exact, i.e. additive structure so this is clear.
\end{remark}

The lemma below illustrates that this enhanced sections functor is a genuine lift of the sections functor defined previously.

\begin{lemma}
    The enhanced sections functor refines the sections functors in the sense that for all virtual complex $A$-representations $V$ there is a commutative diagram
    \[\begin{tikzcd}
        {\Syn^{A}} && {\Sp^{A}} \\
        & \Sp
        \arrow["{\widetilde{\Gamma}(S^V; -)}", from=1-1, to=1-3]
        \arrow["{\Gamma(\Sigma^{V}_+ A / K; -)}"', from=1-1, to=2-2]
        \arrow["{(-)^{K}}", from=1-3, to=2-2]
    \end{tikzcd}\]
\end{lemma}

\begin{proof}
    Note that the $K$-fixed points functor on $\Sp^{A}$ can be identified with evaluation at $(A/K)_+$ in $\cB(A)$, i.e. precomposition along the inclusion of the subcategory generated under extensions and retracts by just this orbit.
\end{proof}
\begin{construction}\label{constr: synthetic unit}
    From general facts about sheaf categories (for example \cite[Proposition 2.10]{pstragowski_synthetic_2022}), the enhanced sections functors $\widetilde{\Gamma}(S^V,-)$ admit left adjoints. When $V = 0$, the fact that the inclusion $\mcB(A) \to \Pure(A)$ is symmetric monoidal tells us that the left adjoint to $\widetilde{\Gamma}(S^0, -)$ is symmetric monoidal. We denote this left adjoint by 
    $$c: \Sp^{A} \to \Syn^A,$$
    cf. the constructions of \cref{cor: unit map (multiplicative)}. 
\end{construction}

\begin{notation}
    We will view $\Syn^A$ as tensored over $\Sp^A$ via the functor $c$, given an $X \in \Sp^A$ and $E \in \Syn^A$, we write $X \otimes E$ for $c(X) \otimes E$. When $X$ is a sphere $S^V$, we will write $\Sigma^V E$ for $c(S^V) \otimes E$. 
\end{notation}

An adjunction argument shows us that the functor $c$ and the synthetic analogue $\nu$ agree on orbits; we have $c(A / K_+) \simeq \nu(A / K_+)$. However, they do not agree in general; this is due to the failure of $\nu$ to commute with desuspensions or colimits. Instead, we have comparison maps arising as assembly maps.
\begin{construction}
    Let $\alpha$ be a character of $A$, then we denote the assembly map associated to $\Sigma^\alpha$ by
    \[\tau_\alpha\colon S^\alpha=\Sigma^\alpha\nu\mathbb{1}_A\to \nu \Sigma^\alpha\mathbb{1}_A=\Th(\alpha).\]
    When $\alpha=\epsilon$ is the trivial character, this is simply denoted
    \[\tau\colon S^{2,0}\to S^{2,1}\]
    and we will frequently identify these maps with their (de)suspensions.
\end{construction}

\subsection{The generic fibre}
As in the nonequivariant case and as suggested by our work in the motivic category, the map $\tau$ defined above is of central importance. To start, we will illustrate how this map provides a further relationship between the synthetic analogue and the spectral Yoneda embedding.
\begin{proposition}\label{prop:tauinversion}
    The connective cover $\nu X \to Y(X)$ is also a $\tau$-inversion; that is, it induces an isomorphism $\tau^{-1} \nu X \xrightarrow{\sim} Y(X)$. 
\end{proposition}
The proof proceeds identically to the nonequivariant case in \cite[Proposition 4.36]{pstragowski_synthetic_2022} since it may be reduced to a connectivity argument.
\begin{proof}
    Note that the target is obviously $\tau$-invertible since $Y$ is an exact functor. The cofibre of the connective cover map $\nu X\to Y(X)$ is $(-1)$-coconnective. We claim more generally that for any $k$-coconnective equivariant synthetic spectrum $E$ with $k<\infty$, the localisation $\tau^{-1}E$ vanishes. This now follows from the formula
    \[\tau^{-1}E\simeq\varinjlim_n\Sigma^{0,n}E\]
    and the observation that for all $n\geq 0$, the shift $\Sigma^{0,n}E\simeq \Sigma^{-2n}\Th(\epsilon)^{\otimes n}\otimes E$ is $(k-2n)$-coconnective as tensoring with a perfect pure object is t-exact. Now note that the t-structure on $\Syn^A$ is right complete and compatible with filtered colimits to conclude.
\end{proof}

With the above result, we are ready to prove half of Theorem \ref{introthm:main} \ref{introthm:mainpart2}, showing that inverting $\tau$ in $\Syn^A$ recovers the category of $A$-spectra. 
\begin{theorem}\label{thm: generic fiber}
    The equivariant spectral Yoneda embedding $Y: \Sp^A\to \Syn^A$ is fully faithful, with essential image given by the $\tau$-invertible synthetic $A$-spectra. 
\end{theorem}
\begin{proof}
    Again, our proof proceeds identically to that of \cite[Theorem 4.37]{pstragowski_synthetic_2022}. We first show that $Y$ is fully faithful. By cocontinuity of $Y$, we can reduce to generators; it suffices to show that the induced map $$\map(\Sigma^n_+ A/K, X) \xrightarrow{\sim} \map(Y(\Sigma^n_+ A/K), Y(X))$$ is an isomorphism for all subgroups $K$ and all $A$-spectra $X$. The exactness of $Y$ and $\map$ lets us further reduce to when $n = 0$. In this case, we have $$\map(Y(A/K_{+}), Y(X)) \simeq \map(\nu(A/K_+), Y(X)) \simeq \map(A/K_+, X),$$ where the first equivalence is via \ref{prop:tauinversion} and the second equivalence is via the Yoneda lemma as in \cite[Lemma 4.11]{pstragowski_synthetic_2022}.

    It remains to show that the essential image of $Y$ is all of $\tau$-invertible synthetic $A$-spectra $\tau^{-1}\Syn^A$. Towards this, note that by the cocontinuity of $Y$, the adjoint functor theorem gives us an adjunction 
    \[\begin{tikzcd}
        {} & {\Sp^A} & {\tau^{-1}\Syn^A}
        \arrow[""{name=0, anchor=center, inner sep=0}, "Y", shift left=2, hook, from=1-2, to=1-3]
        \arrow[""{name=1, anchor=center, inner sep=0}, "\pi", shift left=2, from=1-3, to=1-2]
        \arrow["\dashv"{anchor=center, rotate=-90}, draw=none, from=0, to=1]
    \end{tikzcd}\]
    Using general facts about localisations, it suffices to check that $\pi$ is conservative. We can repeat an argument similar to above; if $\pi(M) = 0,$ then $$\Omega^\infty M(P) \simeq \map(\nu P, M) \simeq \map(Y(P), M) \simeq \map(P, \pi(M)) \simeq 0$$ for any perfect pure $P$, and so by the characterisation of coconnectivity in the synthetic $t$-structure (see \cref{rem:chars of chow t}), $M$ is $(-1)$-coconnective. We can then repeat the argument employed in Proposition \ref{prop:tauinversion} to see that $\tau^{-1}M \simeq 0$; since $M$ was assumed to be $\tau$-invertible, we are done. 
\end{proof}

\newpage
\section{Filtered reconstruction}\label{sec: filtered reconstruction}
Nonequivariantly, synthetic spectra admit a description as a category of modules over a certain commutative algebra in $\Z$-filtered spectra. We will use the equivariant lifts of $\tau$ constructed in the previous section to show that a similar statement is true.
\subsection{Equivariant perfect even filtration}
We begin by showing how the category $\Syn^A$ equips $A$-spectra with a filtration which we call the equivariant perfect even filtration.
\begin{defn}
    Given $E\in \Syn^A$ and $V$ a virtual complex $A$-representation, we write
    \[\tau_{c\geq V}E=\Sigma^V\tau_{c\geq 0}\Sigma^{-V}E.\]
\end{defn}

\begin{ex}
    When $E=Y(X)$ for a $A$-spectrum $E$, the exactness of $Y: \Sp^A \to \Syn^A$ lets us compute
    \[\tau_{c\geq V}Y(X) \simeq \Sigma^V\tau_{c \geq 0} Y(\Sigma^{-V} X)\simeq \Sigma^V\nu(\Sigma^{-V} X) \simeq \Sigma^V \Th(-V) \otimes \nu X. \] In particular, when $X = \mathbb{1}_A$, we obtain 
    $$\tau_{c \geq V} Y(\mathbb{1}_A) \simeq \Sigma^V \Th(-V).$$
\end{ex}

\begin{construction}\label{cons: equivariant tau}
    As a consequence, if $V$ and $W$ are two virtual complex $A$-representations such that $W \ominus \alpha = V$ for a character $\alpha$, we can use $\tau_\alpha$ to construct a map
    \begin{align*}
    \tau_{c\geq W}Y(X)\xrightarrow{\sim}&\Sigma^{W} \Th(-W) \otimes \nu X\\
    \xrightarrow{\sim}&\Sigma^{W - \alpha} S^{\alpha} \otimes \Th(-W) \otimes \nu X\\
    \xrightarrow{}&\Sigma^{W - \alpha} \Th(\alpha) \otimes \Th(-W)\otimes \nu X\\
    \xrightarrow{\sim}&\Sigma^{V} \Th(-V) \otimes \nu X\\
    \xrightarrow{\sim}&\tau_{c\geq V} Y(X)      
    \end{align*}
    for any $X \in \Sp^{A}$ using \cref{lem: properties of nu}.
\end{construction}

\begin{defn}
    Let $\RU(A)$ denote the poset whose objects consist of isomorphism classes of virtual complex representations of $A$ (i.e. $\Z[A^\vee]$) with the poset structure where $V\leq W$ if and only if there exists a collection of characters $\alpha_1,\ldots,\alpha_n$ such that
    \[W\cong V\oplus\alpha_1\oplus\cdots\oplus\alpha_n.\] This admits a symmetric monoidal structure coming from the direct sum of representations. The category of $\RU(A)$-filtered $A$-spectra is the functor category
    \[\Fil_{\RU(A)}(\Sp^{A})=\Fun(\RU(A)^\op,\Sp^A)\]
    equipped with the Day convolution symmetric monoidal structure.
\end{defn}
We can use \cref{cons: equivariant tau} to obtain a functor
\begin{align*}
    \tau_{c\geq -\star}Y(\mathbb{1}_A)\colon\RU(A)&\to \Syn^{A}, \\
    V&\mapsto \tau_{c\geq -V}Y(\mathbb{1}_A)
\end{align*}
The observation that the synthetic t-structure is compatible with the symmetric monoidal structure tells us that this is a lax symmetric monoidal functor. In fact, \cref{lem: properties of nu} tells us that this is a strong symmetric monoidal functor.
\begin{defn}\label{def: filtered synthetic unit}
    The filtered unit map for $\Syn^{A}$ is the symmetric monoidal left adjoint
    \[L\colon\Fil_{\RU(A)}(\Sp^A)\to \Syn^{A}\]
    obtained by left Kan extending the functor $\tau_{c\geq -\star}\mathbb{1}_{A}$ and tensoring with the symmetric monoidal left adjoint from \cref{constr: synthetic unit}.
\end{defn}
\begin{remark}
    The upshot of phrasing everything in terms of the synthetic t-structure becomes apparent here: it is clear from the construction that the equivariant synthetic spectra $\tau_{c\geq -V}Y(\mathbb{1}_A)\simeq \Sigma^{-V}\Th(V)$ give rise to an $\RU(A)$-indexed diagram of Picard elements. In order to extend this to a \emph{strong} symmetric monoidal functor one would have to verify that these are \emph{strict} Picard elements. Since these equivalently arise as connective covers in a t-structure which is compatible with the symmetric monoidal structure and further satisfies \cref{lem: properties of nu}, this is now essentially formal.
\end{remark}
\begin{remark}\label{rem:RU generators}
    For formal reasons, the presheaf category $\Fil_{\RU(A)}(\Sp^A)$ has a family of elements $\mathrm{ins}_VX$ for every $V\in \RU(A), X\in \Sp^A$ defined by
    \[\mathrm{ins}_V(X)(W)=\begin{cases}
        X & W\geq V,\\
        0 & \text{else},
    \end{cases}\]
    i.e. the representable presheaves. The functor $L$ can then be uniquely characterised as a map of $\Sp^A$-algebras by the identity
    \[L(\mathrm{ins}_V \mathbb{1}_A)\simeq \tau_{c\geq -V}{Y(S_A^0)}.\] 
    Consistent with the observation that $\tau_{c\geq -V}{Y(\mathbb{1}_A)}$ is a Picard element, we note that the elements $\mathrm{ins}_V\mathbb{1}_A$ are tensor-invertible as well since every element of the symmetric monoidal poset $\RU(A)$ is invertible.
\end{remark}
The $\RU(A)^\op$-indexed diagram $L(\mathrm{ins}_\star \mathbb{1}_A)$ then formally recovers the synthetic Whitehead tower on a synthetic analogue: for any $A$-spectrum $X$ there is an equivalence
\[L(\mathrm{ins}_{V} \mathbb{1}_A)\otimes \nu X\simeq \tau_{c\geq -V}Y(X)\]
following immediately from an application of \cref{lem: properties of nu} and \cref{def: filtered synthetic unit}.
\begin{construction}\label{cons: equivariant perfect even filtration}
     \cref{def: enhanced sections} and \cref{cons: equivariant tau} give rise to a lax symmetric monoidal functor
    \[R=\widetilde{\Gamma}(S^0;\tau_{c\geq\star}\mathbb{1} \otimes - )\colon\Syn^A\to \Fil_{\RU(A)}(\Sp^A)\]
    where the filtration index $\star$ now ranges over $\RU(A)$.
\end{construction}
\begin{definition}
    The \emph{equivariant perfect even filtration} is the lax symmetric monoidal functor
    \[\fil^\even_\star=R\circ \nu \colon\Sp^A\to \Fil_{\RU(A)}(\Sp^A).\]
\end{definition}
\begin{remark}\label{rem: formula for even filtration}
    The functor $R$ is right adjoint to $L$. Indeed, $L$ allows us to view $\Syn^A$ as enriched over $\Fil_{\RU(A)}(\Sp^A)$ so that under this enrichment (denoted $\hom^\fil$) we have
\[R(E)\simeq \hom^\fil({\mathbb{1}_A^{\Syn}},E).\]
Further, our discussion in Remark \ref{rem:RU generators} let us rewrite the even filtration as $$\fil_\star^{\ev}(X) = \widetilde{\Gamma}(S^0; \tau_{c \geq \star} \mathbb{1} \otimes \nu X) \simeq \widetilde{\Gamma}(S^0; \tau_{c \geq \star} Y(X)).$$
\end{remark}

\subsection{Filtered monadicity}
While the category $\Syn^A$ provides a natural construction of the equivariant perfect even filtration, we now prove that it is in fact entirely determined by it.
\begin{construction}\label{cons: filtered synthetic adjunction}
    Since the adjunction $L\dashv R$ in \cref{cons: equivariant perfect even filtration} has a strong symmetric monoidal left adjoint, we obtain an induced symmetric monoidal adjunction
    \[\cL\colon \Mod(\Fil_{\RU(A)}(\Sp^A);\fil^\even_\star\mathbb{1}_{A})\rightleftarrows \Syn^A:\cR\]
    coming from the identification $R({\mathbb{1}_A^{\Syn}})\simeq \fil^\even_\star\mathbb{1}_{A}$.
\end{construction}
\begin{propn}\label{prop: filtered model}
    The adjunction in \cref{cons: filtered synthetic adjunction} is an equivalence.
\end{propn}
\begin{proof}
    The filtered Schwede--Shipley Theorem of \cite[Lemma 3.16]{pstrkagowski2023perfect} easily adapts to the $\RU(A)$-filtered setting, and it suffices to prove that the objects of the form $\tau_{c\geq V}\mathbb{1}^\Syn_{A}=\Sigma^V\Th(-V)$ generate $\Syn^{A}$ as an $\Sp^{A}$-module. Consider the subcategory $\mcU$ of $\Syn^{A}$ generated by the $\Sigma^V \Th(-V)$ as an $\Sp_{A}$-module. This contains the sheaves $\nu(A/K_+)$ and $\nu(S^V) = \Th(V)$. The former is immediate as we may identify $\nu(A/K_+) = c(A/K_+)$. Using the tensoring by $\Sp^A$, we see that $\cU$ further contains all $\Th(W)=\Sigma^{-W}\tau_{c\geq-W}\mathbb{1}^\Syn_A$. We conclude that $\cU=\Syn^A$.
\end{proof}
\subsection{Adams--Novikov filtration}
The result of \cref{prop: filtered model} gives us a conceptually clarifying model for $\Syn^A$, but we would like to have a better understanding of what exactly the algebra $\fil^{\ev}_\star \mathbb{1}$ is. Here, we relate this algebra to the equivariant Adams--Novikov spectral sequence.
\begin{defn}
    Let $\tau_{c\geq \star}, \star\in 2\Z$ denote the double speed Whitehead tower for the Mackey functor t-structure on $\Sp^A$. Extend this to $\star\in \RU(A)$ by setting
    \[\tau_{c\geq V}X=\Sigma^V\tau_{\geq 0}\Sigma^{-V}X\]
    for any $A$-spectrum $X$ and virtual complex $A$-representation $V$.
\end{defn}
The goal of this section is to prove the following result.
\begin{thm}\label{thm: identification of even filtration on unit}
    There is an equivalence of commutative algebras in $\RU(A)$-filtered $A$-spectra
    \[\fil^\even_\star\mathbb{1}_A\simeq \Tot(\tau_{\geq \star}\MU_A^{\otimes\bullet+1})\]
\end{thm}
Combining this with \cref{prop: filtered model}, we obtain an equivariant analogue of the familiar filtered model for synthetic spectra.
\begin{cor}
    There is a symmetric monoidal equivalence
    \[\Syn^A\simeq \Mod(\Fil_{\RU(A)}(\Sp^A);\Tot(\tau_{\geq 2\star}\MU_A^{\otimes\bullet+1})).\]
\end{cor}
The proof of \cref{thm: identification of even filtration on unit} follows the strategy of \cite[Proposition C.22]{burklund_galois_2022} after establishing some analogues of the results in \cite[Appendix A]{burklund2023boundaries}.
\begin{lemma}\label{lem: unit is indmono}
    The cofibre sequence
    \[\mathbb{1}_A\to \MU_A\to C\]
     associated to the unit map of $\MU_A$ is preserved by $\nu$.
\end{lemma}
\begin{proof}
    The unit map of $\MU_A$ can be expressed as a filtered colimit of maps $\mathbb{1}_A\to G(\alpha)$, where $\{G(\alpha)\}_\alpha$ is a presentation of $\MU_A$ as a filtered colimit of perfect pure spectra along pure monomorphisms, e.g. in terms of the standard Gra{\ss}mannian cell structure which arises as the Betti realisation of \cref{cor: Grass model for MGL} (see \cref{ex: transition maps}). In particular, we may choose this diagram such that the indexing category has an initial object $\alpha_0$ at which it assumes the value $G(\alpha_0)=\mathbb{1}_A$. As such, the maps $\mathbb{1}_A\to G(\alpha)$ are simply the composites of maps in this filtered diagram hence are all pure monomorphisms, in particular, $C$ is itself a filtered colimit of perfect pure objects along pure monomorphisms since these are closed under cobase change. Per construction, $\nu$ preserves cofibres of pure monomorphisms. Since it furthermore preserves filtered colimits we may conclude.
\end{proof}
\begin{lemma}\label{lem: synthetic nilpotent completion}
    Let $X\in \Sp^A$ be $\MU_A$-nilpotent complete, then $\nu X$ is $\nu \MU_A$-nilpotent complete in $\Syn^A$. 
\end{lemma}
\begin{proof}
    Using the Dold--Kan equivalence (see \cite[\S 2.1]{mathew2017nilpotence}), let $\{(\nu X)_i\}_{i\geq 0}$ denote the preferred $\nu \MU_A$-based Adams tower for $\nu X$ defined inductively in terms of the unit map for $\nu\MU_A$ by
    \begin{align*}
        (\nu X)_0&=\nu X,\\
        (\nu X)_{i+1}&=\fib((\nu X)_i\to (\nu X)_i\otimes \nu\MU_A).
    \end{align*}
    Then $\nu X$ is $\nu \MU_A$-nilpotent complete if and only if the limit of this tower vanishes. Similarly, let $X_i$ denote the $\MU_A$-based Adams tower for $X$ in $\Sp^A$, so that the assumption guarantees that its limit vanishes. Then we claim that there is an equivalence for all $i\geq 0$ of the form
    \[\tau_{c\geq -i}Y(X_i)\simeq (\nu X)_i.\]
    The case $i=0$ is true per construction. For the induction step from $i$ to $i+1$, note that the cofibre sequences 
    \[\Sigma^iX_i\to \Sigma^iX_i\otimes \MU_A\to \Sigma^{i+1} X_{i+1}\]
    are preserved by $\nu$ as a consequence of \cref{lem: unit is indmono} and \cref{lem: properties of nu}. We may then use \cref{lem: properties of nu} again to compute
    \begin{align*}
        (\nu X)_{i+1}&\simeq \fib((\nu X)_i\to (\nu X)_i\otimes \nu\MU_A),\\
        &\simeq \Sigma^{-(i+1)}\cof (\tau_{c\geq 0}Y(\Sigma^iX_i)\to \tau_{c\geq 0}(\Sigma^i X_i\otimes \MU_A)),\\
        &\simeq \Sigma^{-(i+1)}\tau_{c\geq 0}\cof(Y(\Sigma^iX_i\to \Sigma^iX_i\otimes \MU_A)),\\
        &\simeq \tau_{c\geq -(i+1)}Y(X_{i+1}).
    \end{align*}
    As a consequence, we obtain comparison maps
    \[(\nu X)_i\simeq \tau_{c\geq -i}Y(X_i)\to \nu X_i\]
    such that the coconnectivity of their cofibre is linear in $i$. Then right completeness of the synthetic t-structure and continuity of $Y$ combine to give an equivalence
    \[\varprojlim_i(\nu X)_i\simeq Y(\varprojlim_iX_i).\]
\end{proof}
\begin{lemma}\label{lem: synthetic analogue MU}
    There is an equivalence
    \[\widetilde{\Gamma}(S^0;\nu\MU_A)\simeq \tau_{\geq 0}\MU_A.\]
\end{lemma}
\begin{proof}
    Let $\tau^\mathrm{lvl}_{\geq 0}Y(\MU_A)$ denote the levelwise connective cover of $Y(\MU_A)$ in presheaves. Then, per construction, its (enhanced) sections over $S^0$ are the connective cover of $\MU_A$. Recall that $\nu\MU_A$ is the synthetic connective cover of $Y(\MU_A)$ so that there is a map of presheaves
    \[\tau^{\mathrm{lvl}}_{\geq 0}Y(\MU_A)\to \tau_{c\geq 0}Y(\MU_A)=\nu\MU_A\]
    which is precisely the map from the source to its sheafification. In order to prove the identification above it therefore suffices to prove that the presheaf $\tau^{\mathrm{lvl}}_{\geq 0}Y(\MU_A)$ is actually a sheaf. We will prove this more generally for the filtered presheaf $\{\tau^{\mathrm{lvl}}_{\geq 2k}Y(\MU_A)\}_{k\in \Z}$, i.e. the double speed Whitehead tower. Since the levelwise t-structure is left and right complete this is a complete filtration so it suffices to prove that it is a sheaf of graded spectra on associated graded. Now note that $\MU_A$ is even so every associated graded presheaf of this filtered presheaf is a presheaf of abelian groups $\pi^{\mathrm{lvl}}_{2k}Y(\MU_A)$. The sheaf condition of \cref{rem: synthetic sheaf condition} now reduces to a checkable criterion: namely that for any fibre sequence of perfect pure spectra $F\to P\to Q$, the sequence
    \[0\to \MU^{2k}_{A}Q\to \MU^{2k}_{A}P\to \MU^{2k}_{A}F\to 0\]
    is short exact. Since the $\MU_A$-cohomology of a perfect pure object is concentrated in even degrees, we see that the cohomological long exact sequence associated to the fibre sequence must split into the desired short exact sequence as above and we may conclude.  
\end{proof}
\begin{cor}\label{cor: perfect even filtration of MU}
    For all $d\geq 1$ there is an equivalence
    \[\fil^\even_\star\MU_A^{\otimes d}\simeq \tau_{\geq\star}\MU_A^{\otimes d}.\]
\end{cor}
\begin{proof}
    This follows from \cref{lem: synthetic analogue MU} by noting that the enhanced sections functor now commutes with $\Sigma^V$, and we only used evenness of $\MU_A$ whence we are free to replace it by $\Sigma^{-V}\MU_A^{\otimes d}$.
\end{proof}
We can now assemble these auxiliary computations into the desired identification.
\begin{proof}[Proof of \cref{thm: identification of even filtration on unit}]
    Since the functor $\cR$ from \cref{cons: filtered synthetic adjunction} preserves limits and $\mathbb{1}_A$ is itself $\MU_A$-nilpotent complete\footnote{This result and its proof have been communicated to us by Balderrama and Hausmann. We note that the same argument provided motivically in \cref{sec: ANSS convergence} goes through in the topological setting.}, we may use \cref{lem: synthetic nilpotent completion} and \cref{lem: properties of nu} to write
    \[\fil^\even_\star\mathbb{1}_A=\cR({\mathbb{1}_A^{\Syn}})\simeq \Tot(\cR\nu(\MU_A^{\otimes\bullet+1})).\]
    Plugging in \cref{cor: perfect even filtration of MU} then gives the desired equivalence 
    \[\fil^\even_\star\mathbb{1}_A\simeq\Tot(\tau_{\geq \star}\MU_A^{\otimes\bullet+1}).\]
\end{proof}
\subsection{The cofibre of tau}
To close, we will use our understanding of the filtered description of synthetic $A$-spectra to identify what happens when $\tau$ is ``killed'' in $\Syn^A$, completing the proof of \cref{introthm:main}. Many results in this section can be compared to the motivic results in \cref{sec: special fibre}. In fact, the following proposition follows from these results. 

\begin{propn}\label{prop: heart of synthetic t-str}
    The heart of the synthetic t-structure can be identified as
    \[\Syn^{A}_{c=0}\simeq \coMod(\underline{\pi}_\star^{A}\MU^{\otimes 2}).\]
\end{propn}
\begin{proof}
    Per construction, the heart of the synthetic t-structure consists of additive sheaves of $0$-truncated connective spectra, i.e.
    \[\Syn^{A}_{c=0}\simeq \Shv_\Sigma(\Pure(A);\mathrm{Ab}),\]
    where we may replace abelian groups with sets since $\Pure(A)$ is additive. We then apply \cite[Theorem 2.49]{pstragowski_synthetic_2022} to the functor $\Be^{A}: \Pure_\C(A) \to \Pure(A)$. The fact that it reflects coverings is \cref{cor: Betti reflects coverings} and the common envelope is established in \cref{lem: common envelope}. Now use \cref{thm: Chow heart structure} and \cref{cor: identification of the heart} to conclude. 
\end{proof}
In the same fashion as \cref{sec: special fibre}, we can use this to identify the category of modules over the truncation of the unit algebraically.
\begin{corollary}\label{cor: modules over trunc syn unit}
    We have a symmetric monoidal equivalence $$\Mod(\tau_{c = 0} \mathbb{1}_A^\Syn) \simeq \Ind \mcD^b_\Pure(\coMod(\underline{\pi}_{\star}^A \MU^{\otimes 2})).$$ 
\end{corollary}
\begin{proof}
    The module category on the left is the derived category of its heart by the same argument in \cref{cor: identification of SF}
\end{proof}
The following proposition shows us that these truncations admits an alternate description in terms of the map $\tau$.
\begin{proposition}
    The map $\tau\colon \Sigma^{0,-1}\mathbb{1}_A^\Syn\to \mathbb{1}_A^\Syn$
    is a $1$-connective cover.
\end{proposition}
\begin{proof}
    Just as in \cref{cons: comparison map} it suffices to prove that the cofibre is coconnective. Using \cref{lem: synthetic nilpotent completion}, we may shift our attention to its $\nu\MU_A$-based cobar complex and prove coconnectivity for every term since coconnective objects are closed under limits. Since coconnectivity in the sheaf category $\Syn^A$ is detected levelwise on sections, it suffices to prove that for every perfect pure $A$-spectrum of the form $\Sigma^V_+A/K$  and $d\geq 1$ the spectrum
    \[\Gamma(\Sigma^V_+A/K,\nu\MU_A^{\otimes d}/\tau)\simeq \cof(\Gamma(\Sigma^V_+A/K,\tau_{c\geq 2}Y(\MU_A))\to \Gamma(\Sigma^V_+A/K,\nu\MU_A^{\otimes d}))\]
    is coconnective. Following the argument in \cref{lem: synthetic analogue MU} this may be computed explicitly as
    \[\cof(\tau_{\geq 2}(\Sigma^{-V}\MU_A^{\otimes d})^K\to \tau_{\geq 0}(\MU_A^{\otimes d})^K)\]
    and this is coconnective (in fact, concentrated in degree zero) since $\MU_A$ is even (\cite[\S XXVIII]{comezana1996calculations}).
\end{proof}
\begin{corollary}\label{cor: modules over cofiber of tau}
    The cofibre $\mathbb{1}_A^\Syn/\tau$ refines to a commutative algebra and we have a symmetric monoidal equivalence
    $$\Mod(\mathbb{1}_A^\Syn/\tau) \simeq \Ind \cD^b_\Pure\coMod(\underline{\pi}_\star^A \MU^{\otimes 2}).$$
\end{corollary}

\newpage
\appendix
\section{Equivariant Betti realisation and parametrised stability}\label{appendix: Betti realisation}
In this section, we set up some comparison functors between the categories of equivariant motivic spectra over $\C$ and categories of equivariant spectra. In one direction, we utilise parametrised properties of the former to construct a universal comparison map. The other direction is provided by a stacky Betti realisation. 
\subsection{Equivariant Betti realisation}
If $k$ is a characteristic zero base field with a specified complex embedding $\sigma\colon \Smooth_k\to \C$, one obtains a symmetric monoidal left adjoint
\[\Be_\sigma\colon \SH(k)\to \Sp\]
called Betti realisation\footnote{If the complex embedding is implicit, e.g. $k=\C$ the decoration $\sigma$ will be dropped.} with respect to $\sigma$. This functor is defined by sending some $X\in \Smooth_k$ to the underlying anima of its complex points $X(\C)$ equipped with the analytic topology and extending to $\SH(k)$. This was further generalised to the relative setting in \cite{ayoub2010note}, where one considers $\SH(B)$ for $B$ a finite type $\C$-scheme, so that the functor sending some $X\in \Smooth_B$ to the complex analytic space $X^\an$ induces a symmetric monoidal left adjoint
\[\Be_B\colon\SH(B)\to \SH^\an(B^\an)\simeq\Shv(B^\an;\Sp),\]
where the right hand side denotes the category of $S^2$-spectra in disc-invariant sheaves of pointed animæ on $B^\an$, which by \cite[Théorème 1.8]{ayoub2010note} can be identified with the category of sheaves of spectra on the complex analytic space $B^\an$. In \cite[Théorème 3.19]{ayoub2010note} it is then further shown that this is compatible with the six functor formalisms on either side if one restricts to quasiprojective base schemes. The goal of this section is to extend a version of this to the equivariant setting, i.e. given a nice abelian group $A$ over $\C$ we construct a symmetric monoidal left adjoint 
\[\Be^A\colon \SH^A(\C)\to \Sp^{A^\an_c}\]
landing in genuine equivariant spectra for (the maximal compact subgroup of) the corresponding Lie group. This will further be shown to be compatible with equivariant operations on either side.
\begin{remark}
    Since we have restricted to nice abelian group schemes over $\C$, the Lie group associated to its complex points always admits a maximal compact subgroup. By \cite[Theorem 1.4.31]{degrijse2023proper} the corresponding categories of equivariant spectra are furthermore equivalent. We will therefore drop the subscript $c$ from the notation, and $A^\an$ will refer to the maximal compact subgroup which is now a compact Lie group.
\end{remark}
Our main tool will be the pullback formalism (in the sense of \cite{drew2022universal}) of genuine sheaves on a differentiable stack introduced in \cite[Part II, \S 4.3]{cnossen2024twisted}.
\begin{propn}[{\cite[Part II, Proposition 4.5.27, Corollary 4.3.10, Proposition 4.4.17]{cnossen2024twisted}}]
    There exists a pullback formalism
    \[\SH^\diff\colon \mathrm{DiffStk}^\op\to \CAlg(\Prlst)\]
    such that the following hold.
    \begin{enumerate}
        \item If $M$ is a smooth manifold, there is a symmetric monoidal equivalence
        \[\SH^\diff(M)\simeq \Shv(M;\Sp).\]
        \item if $G$ is a compact Lie group and $\B G$ its classifying stack, there is a symmetric monoidal equivalence
        \[\SH^\diff(\B G)\simeq \Sp^{G}.\]
    \end{enumerate}
\end{propn}
Inspired by this, we make the following definition.
\begin{defn}
    Let $A$ be a nice abelian group over $\C$ and consider the functor
    \[\Be^A\colon \Smooth_\C^A\to \mathrm{Diff}^{A^\an}\]
    that sends a smooth $\C$-scheme $X$ with $A$-action to the smooth\footnote{We must confess that the notation $X^\an$ for the smooth manifold underlying the complex points of a smooth scheme $X$ is rather abusive: in contrast with \cite{ayoub2010note} we are in fact completely forgetting any complex analytic structure on this manifold. We hope that this does not cause confusion} manifold $X^\an$ with its resulting action by the compact Lie group $A^\an$.
\end{defn}
\begin{remark} The following are immediate.
\begin{enumerate}
    \item Any basic Nisnevich cover of $X\in \Smooth_\C^A$ can be refined to a usual open cover of $X^\an$ in $\mathrm{Diff}^{A^\an}$ in the sense of \cite[Part II, Definition 4.1.2]{cnossen2024twisted}.
     \item Products are sent to products.
    \item The affine line $\A^1\in \Smooth_\C^A$ gets sent to a complex disc which contains the real line  $\R$ as a retract in $\mathrm{Diff}^{A^\an}$.
\end{enumerate}
\end{remark}
Using \cite[Proposition 3.8]{hoyois_six_2017} for the descent property and \cite[Remark 3.13]{hoyois_six_2017} for the homotopy invariance condition, the following result may then be deduced.
\begin{lemma}
    The map $\an$ above induces a symmetric monoidal left adjoint
    \[\Be^A\colon \spcs(\B A)\to \spcs^{\mathrm{diff}}(\B A^\an)\simeq \Ani^{A^\an}.\]
\end{lemma}
Let us illustrate some key examples of the images of certain important equivariant motivic animæ under this functor.
\begin{lemma}\label{lem: examples of Betti realisations}
    Consider $X\in \Smooth_\C^A$, with a $A$-equivariant vector bundle $V$ over it. Let $V^\an$ denote the corresponding complex $A^\an$-equivariant bundle over $X^\an$. Then we have an identification $$\Be^A\Th_X(V)\simeq \Th_{X^\an}(V^\an)\simeq \Sigma^{V^\an}_+X^\an.$$ In particular, representation spheres get sent to the corresponding complex representation spheres and one can further verify that Gra{\ss}mannians get sent to the corresponding equivariant Gra{\ss}mannians.
\end{lemma}
\begin{proof}
This follows immediately by comparing the purity cofibre sequences of \cite[Proposition 5.2]{hoyois_six_2017} and \cite[Part II, Lemma 4.3.2]{cnossen2024twisted}.
\end{proof}
Since we now know what happens to representation spheres, we may extend our equivariant Betti realisation functor to the stable setting.
\begin{cor}\label{cor: equivariant Betti realisation}
    The unstable Betti realisation functor above extends to a symmetric monoidal left adjoint
    \[\Be^A\colon \SH(\B A)\to \Sp^{A^\an}\]
\end{cor}
\begin{proof}
    By \cite[Definition 6.1]{hoyois_six_2017}, it suffices to check that objects of the form $\Th(V)$ for $V$ a vector bundle on $\B A$ get sent to tensor-invertible elements in the target. This is also true per construction using \cref{lem: examples of Betti realisations} and \cite[Definition 4.3.4]{cnossen2024twisted}.
\end{proof}
\subsection{Compatibility with equivariant operations}
We now show that the functor in \cref{cor: equivariant Betti realisation} is compatible with the equivariant operations of restriction, inflation and coinduction. 
\begin{propn}\label{prop: compat}
    Let $f\colon \B K\to \B A$ be a morphism of classifying stacks of nice abelian groups over $\C$ and let $f_\an$ denote the corresponding map of classifying stacks of compact Lie groups. Then there is a commutative diagram of symmetric monoidal left adjoints
    \[\begin{tikzcd}
        \SH(\B A)\arrow[d, "f^\ast"]\arrow[r, "\Be^A"]&\Sp^{A^\an}\arrow[d, "f_\an^\ast"]\\
        \SH(\B K)\arrow[r, "\Be^K"]&\Sp^{K^\an}.
    \end{tikzcd}\]
    If $f$ was representable, then this square is furthermore vertically left adjointable.
\end{propn}
\begin{proof}
    Note that all functors in the diagram above commute with colimits, so it suffices to check commutativity on the level of smooth $\C$-schemes with $A$- resp. $K$-action and their images under the Betti realisation. Commutativity of the diagram then follows immediately from the observation that Betti realisation is compatible with pullbacks. In the case where $f$ is representable and we want to verify left adjointability, the same cocontinuity argument holds along with the observation that $f_\sharp$ resp. $(f_\an)_\sharp$ is given by postcomposition along $f$ resp. $f_\an$ on the level of stacks.
\end{proof}
\subsection{Parametrised stability}
Fixing the base $\C$ and letting $A$ vary over finite abelian groups, we saw in \cref{thm: 6FF omnibus} that the categories $\SH^A(\C)$ satisfy excellent functoriality properties. This may be recast in the language of parametrised homotopy theory and provides us with a comparison functor from genuine $A$-spectra by the universality of the latter. We will freely use the language of \cite{cnossen2023partial} and \cite{cnossen2024normed} below.
\begin{construction}\label{cons: glofab}
      Let $\Glo^{\fab}$ denote the full subcategory of the $(2,1)$-category of connected $1$-groupoids on groupids whose fundamental group is finite and abelian. Let $\Orb^\fab$ denote the wide subcategory on injective group homomorphisms.
\end{construction}
\begin{remark}
    We refer to \cite{cnossen2023partial} for more details on the categories $\Orb^\fab\subset \Glo^\fab$, and in particular note that it defines a \emph{cleft} in the sense of \cite[\S 3]{cnossen2023partial}.
\end{remark}
\begin{construction}\label{cons: glo to stk comparison}
    Consider the functor
    \[\Glo^\fab\to \Stk_\C\]
    that sends a $1$-groupoid to the classifying stack of the corresponding $1$-groupoid in $\C$-schemes.
\end{construction}
\begin{remark}
    Per construction, every $1$-groupoid in $\Glo^\fab$ is (noncanonically) of the form $\B A$ with $A$ a finite abelian group, so that the functor above simply sends it to the classifying stack of the corresponding nice abelian group $A_\C=\underline{\Hom}(A^\vee,\G_{m,\C})$ arising as the constant group scheme on $A$.
\end{remark}
\begin{remark}\label{rem: orb gets sent to smooth proper}
    It is clear from the construction of the functor $\Glo^\fab\to \Stk_\C$ that all morphisms in $\Orb^\fab$ get sent to representable morphisms which are furthermore smooth and proper.
\end{remark}
\begin{defn}
    The $\Glo^\fab$-category of $\C$-motivic spectra is the functor
    \[\SH^\bullet(\C)\colon \Glo^{\fab,\op}\to \Cat\]
    arising as the composite
    \[\Glo^{\fab,\op}\to \Stk_\C^\op\xrightarrow{\SH(\bullet)}\Cat\]
    of the functors in \cref{cons: glo to stk comparison} and \cref{thm: 6FF omnibus}.
\end{defn}
\begin{remark}
    We note that the functor from \cref{cons: glo to stk comparison} can be extended to a functor of presheaf categories $\cP(\Glo^\fab)\to \cP(\Stk_\C)$ and -- as is standard -- we will equivalently view $\SH^\bullet(\C)$ as a limit-preserving functor $\cP(\Glo^\fab)^\op\to \Cat$.
\end{remark}
It is then quite straightforward to convert the six-functor-formalism on $\SH$ into a statement about the parametrised properties of this category.
\begin{propn}\label{prop: parametrised properties}
    The $\Glo^\fab$-category $\SH^\bullet(\C)$ satisfies the following.
    \begin{itemize}
        \item It is $\Orb^\fab$-cocomplete.
        \item It is $\Orb^\fab$-presentable.
        \item It is $\Orb^\fab$-stable.
    \end{itemize}
\end{propn}
\begin{proof}
    For the cocompleteness statement, we use \cite[Lemma 4.9]{cnossen2023partial} to see that this boils down to the following observations.
    \begin{enumerate}
        \item Every $\SH^A(\C)$ is cocomplete by construction.
        \item For every morphism $f$ in $\Orb^\fab$ the resulting map of stacks (also denoted $f$) is smooth so that $f^\ast$ admits a further left adjoint $f_\sharp$.
        \item These further left adjoints satisfy the Beck--Chevalley condition by the smooth base change formula in \cref{thm: 6FF omnibus}.
    \end{enumerate}
    For the presentability statement, it suffices to note that every $\SH^A(\C)$ is presentable per construction. For the stability statement, note that $\SH^\bullet(\C)$ is clearly fibrewise stable so it suffices to prove that it is parametrised semiadditive. By \cref{rem: orb gets sent to smooth proper} this follows from smooth base change.
\end{proof}
We will summarise this by saying that $\SH^\bullet(\C)$ is equivariantly presentable and equivariantly stable as a $\fab$-global category. Having established these properties of $\SH^{\bullet}(\C)$, it is now formal that it obtains a comparison functor from the universal $\fab$-parametrised category satisfying these properties. If we worked instead over all finite groups, then \cite[Theorem 9.4]{cnossen2023partial} tells us that this universal example is given by the parametrised category sending a finite group $G$ to $\Sp^G$.
\begin{cor}\label{cor: unit map (additive)}
    Evaluation at the equivariant sphere spectra induces an equivalence of $\Glo^\fab$-categories
    \[\underline{\Fun}^{\Orb^\fab\text{-cc}}_{\Glo^\fab}(\Sp^\bullet,\SH^\bullet(\C))\simeq \SH^\bullet(\C),\]
    where the left hand side is the full subcategory of the parametrised functor category on $\Orb^\fab$-cocontinuous functors.
\end{cor}
\begin{proof}
    Following \cite[Theorem 8.11]{cnossen2023partial} and the results of \cref{prop: parametrised properties} this is immediate once we identify the source with the free equivariantly presentable equivariantly stable $\fab$-parametrised category on a point. This now follows by observing by restricting the result \cite[Theorem 9.4]{cnossen2023partial} to the finite abelian setting. Indeed, equivalences of parametrised categories may be detected on the wide subcategory of injective maps, and passing to all slices is equally conservative. This therefore follows from the observation that any subgroup of an abelian group is again abelian.
\end{proof}
The result of \cref{cor: unit map (additive)} is enough to provide us with compatible comparison functors from genuine $A$-spectra. This can be further refined to a multiplicative setting.
\begin{construction}
    For a fixed finite abelian group $A$, we may consider the restriction of $\SH^\bullet(\C)$ to the slice $\Orb^\fab_{/A}$ and further extend this to its finite coproduct completion $\Fin_A$, the $1$-category of finite $A$-sets. We therefore obtain functors
    \[\SH^A(\C)\colon \Fin_A^\op\to \CAlg(\Prl)\]
    by noting that every $\SH^K(\C), K\subset A$ admits a presentably symmetric monoidal structure compatible with restriction as in \cref{thm: 6FF omnibus}.
\end{construction}
\begin{lemma}\label{lem: naïve sym mon}
    Let $A$ be a finite abelian group, then the functor
    \[\SH^A(\C)\colon \Fin_A^\op\to\CAlg(\Prl) \]
    exhibits $\SH^A(\C)$ as an $A$-presentably symmetric monoidal $A$-category
\end{lemma}
\begin{proof}
    For the definition of an $A$-presentably symmetric monoidal $A$-category, see \cite[Definition 5.4.5]{cnossen2024normed} and the discussion below it. As remarked in loc. cit. it suffices to exhibit a smooth projection formula but this again follows from \cref{thm: 6FF omnibus}. 
\end{proof}
Using the results of \cref{cor: unit map (additive)} and plugging in the symmetric monoidal units $\mathbb{1}_A$, we see that for every finite abelian group $A$ we obtain a cocontinuous functor $c^A\colon \Sp^A\to \SH^A(\C)$ which we call the unit functor at $A$. 
\begin{cor}\label{cor: unit map (multiplicative)}
    For every finite abelian group $A$, the unit functor
    \[c^A\colon\Sp^A\to \SH^A(\C)\]
     admits a unique symmetric monoidal refinement which is furthermore compatible with pullbacks.
\end{cor}
\begin{proof}
    Following \cite[Theorem 5.4.10 (3)]{cnossen2024normed}, the source admits a unique $A$-presentably symmetric monoidal structure with unit given by the equivariant sphere spectrum. Using \cref{lem: naïve sym mon} we see that the $\E_0$-algebra map given by the units $c^A$ therefore uniquely extends to an $\E_\infty$-algebra map.
\end{proof}
\begin{proof}
    Following \cref{prop: compat} we see that the composite
    \[\Sp^A\xrightarrow{c}\SH^A(\C)\xrightarrow{\Be^A}\Sp^A\]
    is compatible with restriction maps in $A$ hence provides a parametrised functor. One can check that it is a levelwise equivalence since the composite is a left adjoint that sends $A$-orbits to $A$-orbits hence is an equivalence of parametrised categories.
\end{proof}

\newpage
\section{Equivariant motivic convergence}\label{sec: ANSS convergence}
In $\Sp$, convergence of the Adams--Novikov spectral sequence is an essentially formal matter; the unit map $\mathbb{1}\to \MU$ has $1$-connective fibre so one may easily conclude that the augmentation map
\[\mathbb{1}\to \Tot(\MU^{\otimes\bullet+1})\]
in the Amitsur complex is an equivalence. In fact, this is clearly seen to extend to any bounded below spectrum. In the motivic world, a similar result is proven in \cite[\S 5.1]{mantovani2023localizations} for any bounded below\footnote{In the homotopy t-structure which we will recall below} motivic spectrum.
\begin{thm}[{\cite[\S 5.1]{mantovani2023localizations}}]\label{thm: nonequivariant convergence}
    Let $E\in \SH(\C)$ be bounded below in the homotopy t-structure, then the map
    \[E\to E^\wedge_\MGL\coloneq\Tot(E\otimes\MGL^{\otimes\bullet+1})\]
    to its $\MGL$-nilpotent completion can be identified with the $\eta$-completion map $E\to E^\wedge_\eta$.
\end{thm}
Equivariantly, things become trickier: even in the topological setting, $\MU_A$ does not satisfy any kind of nice connectivity properties prima facie. For simplicitly, we will work over the stack $[\Spec(\C)/A]$ throughout as this is the setting of our main application.
\subsection{The homotopy t-structure}
The correct notion of bounded below object is provided by the homotopy t-structure of \cite{bachmann2022motivic}.
\begin{defn}[{\cite[\S 7.4]{bachmann2022motivic}}]\label{def: homotopy t-structure}
    The homotopy t-structure on $\SH^A(\C)$ is such that its nonnegative part is generated under colimits by objects of the form $\Sigma^{n,n}\Sigma^{\infty}_+X$ for $n\in \Z$ and $X\in \spcs^A(\C)$. Its cover and truncation functors will be denoted by $\tau_{\geq n}$ and $\tau_{\leq n}.$
\end{defn}
It is clear from the definition that when $A$ is the trivial group this recovers the usual homotopy t-structure on $\SH(\C)$ as recalled in \cite[\S 2.1]{hoyois2015algebraic}. Furthermore, the homotopy t-structure is compatible with the symmetric monoidal structure and filtered colimits.
In our setting, the stack $[\Spec(\C)/A]$ is a tame DM-stack and satisfies the Adams hypothesis in the sense of \cite{bachmann2022motivic} so that one may conclude the following.
\begin{lemma}[{\cite[Proposition 7.10]{bachmann2022motivic}}]\label{lem: homotopy t-structure bicomplete}
    The homotopy t-structure on $\SH^{A}(\C)$ is both left and right complete.
\end{lemma}
\begin{lemma}\label{lem: homotopy t-structure and stratification}
    The fixed points functor
    \[(-)^{A}\colon\SH^{A}(\C)\to \SH(\C)\]
     is t-exact. Furthermore, each of the functors appearing in the filtration of
     \cref{rem: adjacent isotropy filtration} as well as inflation and coinduction are right t-exact.
\end{lemma}
\begin{proof}
    This is \cite[Proposition 7.9 (3)]{bachmann2022motivic}. In fact, note that the filtration by adjacent isotropy is used to prove that fixed points are right t-exact.
\end{proof}
\begin{remark}\label{rem: perfect pures and homotopy t-structure}
    It is clear that all perfect pure motivic spectra are bounded below in the homotopy t-structure.
\end{remark}
\subsection{Adjacent isotropy and convergence}
We can now state and prove the desired convergence result.
\begin{thm}\label{thm: convergence}
    Let $E\in \SH^{A}(\C)^\cell$ be bounded below in the homotopy t-structure and $\eta$-complete, then the map
    \[E\to E^\wedge_{\MGL_{A}}\coloneq\Tot(E\otimes\MGL_{A}^{\otimes\bullet+1})\]
    to its $\MGL_{A}$-nilpotent completion is an equivalence.
\end{thm}
    We pass through an intermediate step using the inflation of nonequivariant $\MGL$ over $\C$. Letting $p\colon [\Spec(\C)/A]\to \Spec(\C)$ denote the structure map, then the absolute functoriality of $\MGL$ established in \cref{prop: MGL is absolute} provides us with a ring map
    \[\alpha_p\colon p^\ast\MGL\to \MGL_{A}.\]
    As remarked in \cref{lem: global action is trivial} this becomes an equivalence after tensoring with $\E A_+$. When nontrivial geometric fixed points are involved, we invoke another geometric construction which shows that they have the same class of nilpotent objects on the associated graded of the stratification by adjacent isotropy.
\begin{lemma}\label{lem: multiplicative deflation}
    For $K\subset A$ a nontrivial subgroup with quotient $W=A/K$, let $q\colon \B A\to \B W$ denote the projection, then there is a $q^\ast\MGLP_{W}$-algebra map
    \[\MGLP_{A}\otimes \E[K]\to q^\ast\MGLP_{W}\otimes \E[K].\]
\end{lemma}
\begin{proof}
    Following \cref{cor: gfp of Thom} and \cref{def: MGL} it suffices to work on the unstable level and then Thomify. Note that $\MGLP_{A}\otimes\widetilde{\E}\cF[K]$ then corresponds to the Thom spectrum of the restriction of the universal $A$-bundle to the fixed points $\BGLP_{A}^{K}$ viewed as an $A$-equivariant motivic space by inflating from the residual $W$-action (i.e. $q^\ast\BGLP_{A}^{K}$). Per construction, this then classifies $A$-equivariant vector bundles over smooth $W$-schemes. There is a well defined $A$-equivariant map
    \[q^\ast\BGLP_{A}^{K}\to q^\ast \BGLP_{W}\]
    which takes fibrewise $K$-fixed points of an $A$-equivariant bundle. Indeed, since $K$ acts trivially on the base, it must preserve the fibres so that this is indeed well defined. There is also a $A$-equivariant inflation map \[q^\ast \BGLP_{W}\to q^\ast\BGLP_{A}^{K}\]
    equipping a $W$-equivariant vector bundle over a smooth $W$-scheme with its inflated $A$-action. It is then clear from the construction that the composite map
    \[q^\ast\BGLP_{W}\to q^\ast\BGLP_{A}^{K}\to q^\ast\BGLP_{W}\]
    is the identity. Furthermore, all operations on equivariant vector bundles described above are clearly additive so that this Thomifies to a ring map.
\end{proof}
\begin{propn}\label{prop: convergence on fixed points}
    For $E\in \SH^{A}(\C)^\cell$ to be $\MGL_{A}$-nilpotent complete, it suffices that every nonequivariant motivic spectrum $(\Phi^{K}E)_{\mathrm{h}W}$ appearing in the stratification by adjacent isotropy is $\MGL$-nilpotent complete.
\end{propn}
\begin{proof}
    We proceed by induction on the rank of $A$, where the base case is the nonequivariant result \cref{thm: nonequivariant convergence}. Applying $A$-fixed points to the augmentation map, we obtain a map
    \[E^A\to \Tot(E\otimes\MGL_A^{\otimes\bullet+1})^A.\]
    Using the filtration by adjacent isotropy of \cref{rem: adjacent isotropy filtration} we may reduce to showing that this induces an equivalence on associated graded of the resulting filtations on either side, i.e. replacing $E$ by $E\otimes \E[K]$. If the subgroup $K$ of $A$ is a nontrivial subgroup, then the comparison map on this associated graded term is of the form
    \[(E\otimes \E[K])^A\to \Tot(\E[K]\otimes E\otimes \MGL_A^{\otimes\bullet+1})^A.\]
    We may then apply the induction step using \cref{lem: multiplicative deflation} as $W$ is of strictly smaller rank than $A$. When we are working at the trivial subgroup, the comparison map on associated graded becomes
    \[(E\otimes \E A_+)^A\to \Tot(\E A_+\otimes E\otimes \MGL_A^{\otimes\bullet+1})^A.\]
    However, since $\MGL_A^{\otimes\bullet+1}$ is the value at $A$ of an absolute motivic spectrum, we may apply \cref{lem: global action is trivial} to replace it by $p^\ast\MGL^{\otimes\bullet+1}$ where $p^\ast$ is the inflation functor from the trivial group. In order to apply the induction step, note that there is an equivalence of cosimplicial objects
    \[(\E A_+\otimes E\otimes p^\ast\MGL^{\otimes\bullet+1})\simeq E_{\mathrm{h}A}\otimes\MGL^{\otimes\bullet+1}.\]
    Indeed, viewing both sides as the evaluations at $\MGL^{\otimes\bullet+1}$ of  endofunctors of $\SH(\C)$, we see that the equivalence holds when one instead plugs in $\mathbb{1}_\C$ by the motivic Adams isomorphism of \cite[\S 6]{gepner2023tom}, now $\Smooth_\C$-parametrised cocontinuity of either side forces this natural equivalence to hold in general following \cite[Lemma C.6]{bachmann2021motivic}.
\end{proof}
\begin{lemma}\label{lem: bb+eta complete implies same for stratification}
    Let $E\in \SH^A(\C)^\cell$ be bounded below and $\eta$-complete, then the conditions of \cref{prop: convergence on fixed points} are met.
\end{lemma}
Using \cref{thm: nonequivariant convergence} it suffices to show that every $(\Phi^KE)_{\mathrm{h}W}$ is again bounded below and $\eta$-complete. The first part follows immediately from the observations in \cref{lem: homotopy t-structure and stratification}. The second part is more delicate: it is clear that the fixed points of $E$ and its twists by virtual $A$-representations will be $\eta$-complete as nonequivariant motivic spectra but it is not a priori clear that the same can be said for the associated graded terms in the filtration by adjacent isotropy. The bounded-below assumption as well as an explicit cell structure on $\E[K]$ allows us to conclude by mimicking the argument in \cite[Lemma 3.2]{greenlees1991equivariant}.
\begin{proof}[Proof of \cref{lem: bb+eta complete implies same for stratification}]
    Consider the inverse system $\{E/\eta^n\}_n$ which has limit $E$ and is uniformly bounded below. Since the $A$-fixed points functor commutes with limits and sends the inflated class $\eta$ to the corresponding nonequivariant class $\eta$, it suffices to verify that the map
    \[\E[K]\otimes E\to \varprojlim_n \E[K]\otimes E/\eta^n\]
    induces an equivalence on fixed points. Let $\{X_i\}_{i\in I}$ be a generic filtered diagram of dualisable $A$-equivariant motivic spaces with colimit $X$, then in order to commmute the directed limit over $n$ past tensoring with $X$ we may use the Bousfield--Kan formula to see that it is enough to commute it past an infinite product. Indeed, since every $X_i$ is dualisable, we obtain assembly maps of the form
    \[X_i\otimes\prod_nE/\eta^n\simeq \prod_n X_i\otimes E/\eta^n\to \prod_n X\otimes E/\eta^n.\]
    We claim that if the maps $X_i\to X$ increase in connectivity, then the desired assembly map becomes an equivalence after taking the colimit over $i$ as desired. Indeed, the cofibre of this assembly map is a product over $n$ of terms of the form $\mathrm{cof}(X_i\to X)\otimes E/\eta^n$ and the assumption that the system $E/\eta^n$ was uniformly bounded below and the left completeness of the homotopy t-structure will then guarantee that this object vanishes in the colimit. It therefore suffices to verify the premise of this claim for some suitable geometric model of $\E[K]$. In fact, one can separately consider $\widetilde{\E}\cF[K]$ as well as $\E W_+$ as objects of $A$-equivariant resp. $W$-equivariant motivic spectra. For the former we use the model in \cref{lem: cellular model for gfp} and for the latter we may use \cite[Example 3.5]{gepner2023tom} (and \cite[Lemma 6.29]{gepner2023tom} for the dualisability condition) to obtain our filtered diagrams. The connectivity conditions in either case are clearly met as well since both presentations consist of gluing on cells of increasing dimension. 
\end{proof}
In well behaved cases, the $\eta$-completion that arises in the process of $\MGL$-nilpotent completion of a bounded below object is not relevant since we are working over an algebraically closed field of characteristic zero. Let us recall the following nonequivariant observation, which we thank Klaus Mattis for pointing out. This result makes use of the notion of a very effective motivic spectrum over $\C$, i.e. an object of the minimal subcategory of $\SH(\C)$ containing suspension spectra of smooth $\C$-schemes and closed under colimits and extensions.
\begin{theorem}[{\cite[Theorem 5.1]{bachmann2020eta}}]\label{thm: BH eta comp}
    Let $E\in \SH(\C)$ be very effective, then its $2$-completion $E^\wedge_2$ is $\eta$-complete.
\end{theorem}
\begin{corollary}\label{cor: unit is eta complete}
    The unit $\mathbb{1}_A\in \SH^A(\C)^\cell$ is $\MGL$-nilpotent complete.
\end{corollary}
\begin{proof}
    By \cref{lem: bb+eta complete implies same for stratification} it suffices to prove that this object is cellularly $\eta$-complete. As in the proof of this lemma, we are reduced to verifying that the nonequivariant motivic spectra of the form $\Th_{\B_\mot A}(\widetilde{\cE})$ arising in the stratification by adjacent isotropy are $\eta$-complete. Using the standard $2$-adic fracture square of the form
    \[\begin{tikzcd}
     \Th_{\B_\mot A}(\widetilde{\cE})\arrow[r]\arrow[d]&\Th_{\B_\mot A}(\widetilde{\cE})^\wedge_2\arrow[d]\\\Th_{\B_\mot A}(\widetilde{\cE})[1/2]\arrow[r]&\Th_{\B_\mot A}(\widetilde{\cE})^\wedge_2[1/2],
    \end{tikzcd}\]
    it suffices to prove that the three bottom vertices are $\eta$-complete.
    When $2$ is inverted, this follows immediately as $\eta$ is then nilpotent (see \cite[Lemma 3.3.1]{mantovani2023localizations} for a general argument). In the case $p=2$, we use \cref{thm: BH eta comp} with Morel--Voevodsky's very effective model for motivic classifying spaces recalled in \cref{ex: model for classifying space}.
\end{proof}

\newpage
\section{Geometric fixed points and universes}\label{appendix: gfp2}
In \cite[\S 4.3]{gepner2023tom} and \cite[\S 4]{bachmann2022motivic}, the authors set up a theory of motivic geometric fixed point which, in our setting, accounts for the case of finite groups. The goal of this section is to set up a theory of geometric fixed points for general nice abelian groups over $\C$, generically denoted $A$. The techniques and proofs are not original and we will frequently be able to mimic the arguments in \cite{gepner2023tom}. As in op. cit. we state our results in the generality of partially stabilised categories of equivariant motivic spectra which we will encode in the notion of an $A$-universe.
\begin{defn}\label{def: A-universe}
    An $A$-universe is a countably infinite-dimensional $A$-representation $\cU$ which contains infinitely many summands of the form $\epsilon$, the trivial representation of $A$. An $A$-universe is complete if every character of $A$ appears as a summand countably infinitely many times, in which case it will generically be denoted by $\cU_A$.
\end{defn}
One can define a category of $A$-equivariant motivic spectra indexed on an arbitrary universe as follows.
\begin{defn}\label{def: partial A-spectra}
    Let $\cU$ be an $A$-universe, and let $\cT(\cU)$ denote the set of objects in $\spcs^A(\C)_\ast$ of the form $\Th(V)$ for $V$ a finite-dimensional subrepresentation of $\cU$. Then the category of $A$-equivariant motivic spectra indexed on $\cU$ is the inversion
\[\SH^A(\C)_{\cU}=\spcs^A(\C)_\ast[\cT(\cU)^{\otimes -1}].\]
\end{defn}
\begin{remark}
    The inversion operation in the definition above is described in \cite[\S 6.1]{hoyois_six_2017}. Further, note that all objects of $\cT(\cU)$ are symmetric by \cite[Lemma 6.3]{hoyois_six_2017} so that this inversion is presented by the mapping telescope in $\PrL$ and thus admits a unique $\spcs^A(\C)_\ast$-algebra structure. In fact, the assumption that $\epsilon$ appears infinitely many times in $\cU$ tells us that $\cT(\cU)$ is \emph{stabilising} in the sense of \cite[Definition 2.34]{gepner2023equivariant}.
\end{remark}
\begin{remark}\label{rem: Pic in incomplete over general base}
    For any $A$-scheme $X$ and $A$-universe $\cU$ we may further define
    \[\SH^A(X)_\cU=\spcs^A(X)_\ast\otimes_{\spcs^A(\C)_\ast}\SH^A(\C)_\cU.\]
    When $\cU=\cU_A$ is complete, this recovers $\SH^A(X)$ by \cite[Corollary 6.7]{hoyois_six_2017}, while in general only bundles on $X$ whose fibres appear as summands in $\cU$ will give rise to invertible elements; see the proof of \cite[Proposition 6.5]{hoyois_six_2017} for how to reduce the analogous claim for $X=\Spec(\C)$ where it is true per construction.
\end{remark}
It is clear that $\SH^A(\C)_{\cU_A}$ recovers $\SH^A(\C)$, and construction is additionally functorial in inclusions of universes: given an inclusion of universes $\cU\subset \cU'$ we denote the corresponding suspension spectrum functor by
\[\Sigma^{\cU'-\cU}\colon \SH^A(\C)_\cU\to \SH^A(\C)_{\cU'}.\]
\begin{construction}\label{cons: universe fixed points}
    Let $K$ a subgroup of $A$ and $\cU$ an $A/K$-universe, then inflation induces a symmetric monoidal left adjoint
    \[\SH^{A/K}(\C)_{\cU}\to \SH^A(\C)_{\mathrm{Inf}^A_K\cU}\]
    where $\mathrm{Inf}^A_{A/K}\cU$ is the $A$-universe with inflated action. Let $\cU'$ be an $A$-universe containing $\mathrm{Inf}^A_{A/K}\cU$ giving rise to a suspension spectrum functor $\Sigma^{\cU'-\mathrm{Inf}^A_{A/K}\cU}$, then one defines the $K$-fixed points functor indexed on $\cU'$
    \[(-)^K_{\cU'}\colon \SH^A(\C)_{\cU'}\to \SH^{A/K}(\C)_{\cU}\]
    as the composite right adjoint two the two aforementioned left adjoints.
\end{construction}
The condition that the $A$-universe $\cU'$ contains a copy of the inflated $A/K$-universe is always satisfied if the $A/K$-universe was trivial, or if $\cU'$ is complete. We now turn towards the construction of geometric fixed points functors. As is standard, the functor of (relative) geometric fixed points is defined by extending the unstable fixed point functor.
\begin{construction}\label{cons: gfp as Lan extension}
    Let $A$ be a nice abelian group with subgroup $K$, let $\cU$ be an $A$-universe, and let $\cU'$ be any $A/K$-universe such that for every finite-dimensional representation $V$ contains in $\cU$, the $A/K$-representation $V^K$ is contained in $\cU'$. Then there exists a unique filler $\Phi^K_{\cU}$ in the following diagram of symmetric monoidal left adjoints
    \[\begin{tikzcd}
    \spcs^{A}(\C)_\ast\arrow[r, "(-)^K"]\arrow[d, "\Sigma^\cU"]&\spcs^{A/K}(\C)_\ast\arrow[d, "\Sigma^{\cU'}"]\\
    \SH^A(\C)_{\cU}\arrow[r, "\Phi^K_\cU"]&\SH^{A/K}(\C)_{\cU'}.
    \end{tikzcd}\]
    Indeed, we just note that the composite $\Sigma^{\cU'}\circ(-)^K$ preserves colimits and sends a Thom object of the form $\Th(V)$ to $\Sigma^{\cU'}\Th(V^K)$ per construction. If $V$ arises as a summand of $\cU$, then $\Sigma^{\cU'}\Th(V^K)$ is tensor-invertible by our assumption on $\cU'$. We conclude that the composite $\Sigma^{\cU'}\circ (-)^K$ must factor uniquely over $\Sigma^\cU$. This functor $\Phi^K_{\cU}$ is the relative $K$-geometric fixed points functor indexed on the $A$-universe $\cU$.
\end{construction}
\begin{remark}
    Note that $\cU'$ can be arbitrarily large, as long as it contains the $K$-fixed points of $A$-representations in $\cU$. We therefore suppress $\cU'$ from notation since the relative $K$-geometric fixed points functor based on an $A/K$ universe $\cU''\supset \cU'$ would simply be given by $\Sigma^{\cU''-\cU'}\Phi^K_\cU$.
    In practice, it is common to assume that $\cU'$ is a complete $A/K$-universe.
\end{remark}
The remainder of this section is dedicated to describing the effect of this fixed-point functor in more detail. In particular, we show that it arises as a composite of a smashing localisation and the genuine fixed point functor.
\begin{construction}\label{cons: univ space of family}
    Let $A$ be a nice abelian group and $\cF$ a family of subgroups of $A$. Define a presheaf $\E\cF$ on $\Smooth^A_\C$ by sending an $A$-scheme $X$ to the point if $X$ has isotropy in $\cF$ and the empty set otherwise. Following \cite[Proposition 3.3]{gepner2023tom} this defines an object of $\spcs^A(\C)$ and we will further denote its image in $\SH^A(\C)_{\epsilon^\infty}$ by $\E\cF_+$. Note that this comes with a collapse map $\E\cF_+\to \mathbb{1}_A$ whose cofibre will be denoted $\widetilde{\E}\cF$.
\end{construction}
A key example in the remainder of this section is the family $\cF[K]$ of subgroups of $A$ that do not contain a chosen subgroup $K$. In the case $K=A$ this is the family of proper subgroups.
\begin{remark}
    Let $\Smooth^{A,\cF}_\C$ denote the full subcategory of $\Smooth^A_\C$ on schemes with isotropy contained in $\cF$. Then the inclusion of this subcategory induces a left Kan extension on product-preserving presheaves which we denote
    \[j_\sharp\colon \cP_\Sigma(\Smooth^{A,\cF}_\C)\to \cP_\Sigma(\Smooth^A_{\C}).\]
    It is clear from the construction that we may identify $\E\cF\simeq j_\sharp(\ast)$ where $\ast$ denotes the constant presheaf on a point. In particular, any $X\in \Smooth^{A,\cF}_\C$ is such that the collapse map $X_+\otimes \E\cF_+\to X_+$ is an equivalence in $\SH^A(\C)_{\epsilon^\infty}$.
\end{remark}
\begin{proposition}\label{prop: geometric subcat trivial spheres}
    The $\E_0$-algebra $\widetilde{\E}\cF[K]$ in $\SH^{A}(\C)_{\epsilon^\infty}$ is idempotent, and the $\widetilde{\E}\cF[K]$-local subcategory may be identified with $\SH^{A/K}(\C)_{\epsilon^\infty}$ via the inflation functor.
\end{proposition}
The proof of this proposition can be adapted from \cite[Proposition 3.21]{gepner2023tom}, but we present the argument here.
\begin{proof}
    For the first part, the argument of \cite[Lemma 2.38]{gepner2023tom} tells us that the endofunctor $\E\cF[K]_+\otimes -$ is a colocalisation of $\SH^A(\C)_{\epsilon^\infty}$. Per construction, $\widetilde{\E}\cF[K]$ then classifies the complimentary smashing localisation. In order to identify the local category, let $j^\ast$ denote the pullback functor in $\SH(-)_{\epsilon^\infty}$ along the inclusion of sites
    \[j\colon \Smooth_\C^{A,\cF[K]}\to \Smooth^A_\C,\]
    we claim that the $\widetilde{\E}\cF[K]$-equivalences are precisely $j^\ast$-equivalences. One implication is obvious once one notes that $j^\ast$ is strong symmetric monoidal and the construction of $\widetilde{\E}\cF[K]$ is such that $j^\ast\widetilde{\E}\cF[K]$ is the tensor unit. For the reverse implication, let $f$ be a morphisms such that $j^\ast f$ is an equivalence. It then suffices to prove that for all $X$ in $\Smooth^A_\C$ the map $\map(X_+,\widetilde{\E}\cF[K]\otimes f)$ is an equivalence. Let $i\colon X^K\to X$ denote the inclusion of the $K$-fixed points of $X$, which is a closed immersion with open complement denoted $X(\cF[K])$. Then the purity cofibre sequence of \cite[Theorem 3.23]{hoyois_cdh_2020} provides us with a cofibre sequence in $\SH^A(\C)_{\epsilon^\infty}$ of the form
    \[X(\cF[K])_+\to X_+\to \Th_{X^K}(\cN_i)\]
    where $\cN_i$ is the normal bundle of $i$. Per construction, $X(\cF[K])$ is in $\Smooth^{A,\cF[K]}_\C$ so we find that the localisation $\widetilde{\E}\cF[K]\otimes X(\cF[K])_+$ vanishes and it therefore suffices to prove that $\widetilde{\E}\cF[K]\otimes f$ becomes an equivalence after mapping out of $\Th_{X^K}(\cN_i)$. Since $A$ is assumed nice, we may split $\cN_i$ into its $K$-fixed point bundle $\cN_i^K$ and a complement $\cN_i'$. Now, once again, $\mathbb{V}(\cN_i')\setminus 0$ is contained in $\Smooth^{A,\cF[K]}_\C$ so that we obtain an identification
    \[
        \widetilde{\E}\cF[K]\otimes \Th_{X^K}(\cN_i)\simeq \widetilde{\E}\cF[K]\otimes X^K_+\otimes \Th_{X^K}(\cN_i^K).
    \]
    Now it is clear that $X^K_+\otimes \Th_{X^K}(\cN_i^K)$ is in the image of the left Kan extension functor $j_\sharp$ which is left adjoint to $j^\ast$. Unraveling the chain of equivalences, we see that $\widetilde{\E}\cF[K]\otimes f$ is an equivalence if $j^\ast(\widetilde{\E}\cF[K]\otimes f)$ is an equivalence. As in the proof of the obvious implication, this is satisfied if and only if $j^\ast f$ is an equivalence. Finally, the $j^\ast$-local subcategory may be identified with $\SH^{A/K}(\C)_{\epsilon^\infty}$ by noting that $K$ acts freely on every object of $\Smooth^{A,\cF[K]}_\C$ so that taking the quotient provides us with an equivalence of sites between the former and $\Smooth^{A/K}_\C$.
\end{proof}
We note that the result above can easily be extended to larger universes since the process of tensor-inverting Thom objects simplifies drastically in the local subcategory.
\begin{lemma}\label{lem: geometric local subcat generally}
Let $\cU$ be an $A$-universe and and $\cU^K$ the $A/K$-universe obtained by taking $K$-fixed points, then the equivalence of \cref{prop: geometric subcat trivial spheres} extends to an equivalence of the form
\[\Mod(\SH^A(\C)_\cU;\widetilde{\E}\cF[K])\simeq \SH^{A/K}(\C)_{\cU^K}\]
\end{lemma}
\begin{proof}
    Let $V$ be a finite-dimensional subrepresentation of $\cU$, then one can again split off the $K$-fixed points of $V$ so that $\Th(V)$ is $\widetilde{\E}\cF[K]$-locally equivalent to $\Th(V^K)$ since its complement $V'=V-V^K$ is such that $\mathbb{V}(V')\setminus 0$ is contained in $\Smooth^{A,\cF[K]}_\C$. Now $\Th(V^K)$ is already tensor-invertible since $V^K$ appears as a summand of $\cU^K$ per construction.
\end{proof}
We are now on equal footing with the construction of geometric fixed points provided in \cref{cons: gfp as Lan extension} and it is not hard to see that the two resulting functors agree.
\begin{corollary}
    Let $\cU$ be an $A$-universe and $\cU'$ an $A/K$-universe as in \cref{cons: gfp as Lan extension}, then the functor
    \[(\widetilde{\E}\cF[K]\otimes-)^{K}_\cU\colon \SH^A(\C)_\cU\to \SH^{A/K}(\C)_{\cU'}\]
    arising from the equivalence in \cref{lem: geometric local subcat generally} agrees with the functor $\Phi^K_\cU$.
\end{corollary}
\begin{proof}
    By the result above, the functor $(\widetilde{\E}\cF[K]\otimes-)^{K}$ is left Kan extended from its effect on $\SH^A(\C)_{\epsilon^\infty}$ along suspension spectrum functors. It therefore suffices to note that it agrees with $\Phi^K_\cU$ on this subcategory, which is true per construction since both are induced from the unstable fixed points functor, i.e. the right adjoint to inflation.
\end{proof}
In particular, we may interpret the $K$-geometric fixed points functor as the localisation which precisely kills all schemes with isotropy in $\cF[K]$. In the case $K=A$, so that $\cF[A]$ is the family of proper subgroups, this localisation therefore precisely kills all $A$-schemes with a nontrivial $A$-action. We note an immediate corollary.
\begin{corollary}\label{cor: gfp of Thom}
    Let $X$ be an $A$-scheme and $V-W$ an virtual vector bundle over it so that all fibres of $W$ appear as summands of an $A$-universe $\cU$, then we have an equivalence in $\SH^{A/K}(\C)_{\cU^K}$ of the form
    \[\Phi^K_\cU\Th_X(V-W)\simeq \Th_{X^K}(V^K-W^K).\]
\end{corollary}
When $K=A$ we additionally have a geometric model for $\widetilde{\E}\cF[A]$ which describes the corresponding smashing localisation explicitly. Given any nice abelian group $A$ and a countably infinite-dimensional $A$-representation $\cV$ such that each isotpyical summand appears countably infinitely many times (but not necessarily a universe), we say that a \emph{saturated flag} of $\cV$ is a filtered diagram $\{V_i\}_{i\in I}$ of finite-dimensional subrepresentations of $\cV$ and equivariant embeddings indexed by a poset $I$ such that its colimit is $\cV$ and such that for all $i\in I$ there exists some element $2i\in I$ such that $i\leq 2i$ and the resulting embedding $V_i\to V_{2i}$ may be identified with the embedding $V_i\to V_i\oplus V_i$ of the first component.
\begin{remark}
    Following \cite[Example 2.3]{hoyois_cdh_2020}, let $\chi(\cV)\subset A^\vee$ denote the set of characters appearing in the isotypical decomposition of $\cV$, and let $I$ be the poset of functions from $\chi(\cV)$ to $\mathbb{N}$ with only finitely many nonzero values we may set $V_i=\bigoplus_{\alpha\in \chi(\cV)}\alpha^{i(\alpha)}$ with the obvious transition maps. This defines a saturated flag of $\cV$.
\end{remark}
\begin{lemma}\label{lem: cellular model for gfp}
    Let $\cU_{A/K}, \cU_A$ denote complete $A/K$- and $A$-universes respectively. 
    Let $\{V_i\}_{i\in I}$ denote a saturated flag of the complement $\cU_A-\mathrm{Inf}^A_{A/K}\cU_{A/K}$. Then there is an equivalence
    \[\widetilde{\E}\cF[K]\simeq \varinjlim_{i\in I}\Th(V_i)\]
    with transition maps given by pre-Euler classes of vector bundle inclusions.
\end{lemma}
\begin{proof}
    Per construction, the complement of the zero section of every $V_i$ has stabilisers contained in $\cF[K]$, so we obtain a map
    \[\varinjlim_{i\in I}\mathbb{V}(V_i)\setminus 0\to \E\cF[A]\]
    in $\spcs^A(\C)$. By the purity formula for Thom objects and the definition of $\widetilde{\E}\cF[A]$, it suffices to prove that this is an equivalence. By \cite[Proposition 3.7]{gepner2023tom} (see \cite[Lemma 2.6]{hoyois_cdh_2020}), it suffices to prove that any \emph{affine} smooth $A$-scheme $X$ with isotropy in $\cF[K]$ admits an embedding into $\mathbb{V}(V_i)\setminus 0$ for $i$ sufficiently large. We can always find an embedding into some sufficiently large $A$-representation by \cite[Lemma 2.4]{hoyois_cdh_2020}, but the assumption that $X$ has isotropy in $\cF[K]$ guarantees that it avoids both the zero section and any summands on which $K$ acts trivially, i.e. inflated summands. Since $\cU_A-\mathrm{Inf}^A_{A/K}\cU_{A/K}$ contains all such representations, we may therefore find an embedding into some $\mathbb{V}(V_i)\setminus 0$.
\end{proof}
\begin{remark}
    Heuristically, the proof above says that every smooth $A$-scheme with isotropy in $\cF[K]$ (i.e. point of $\mathbb{E}\cF[K]$) can be embedded into some $\mathbb{V}(V_i)\setminus 0$. The saturated condition on the flag then asserts that the spaces of such embeddings are manifestly contractible after motivic localisation so that this collection of embeddings assembles into an inverse to the comparison map above.
\end{remark}
\begin{corollary}\label{cor: euler model for gfp}
    The localisation of $\SH^A(\C)$ associated to $\widetilde{\E}\cF[K]$ is precisely the localisation that inverts all pre-Euler classes of characters $\alpha$ of $A$ that are not inflated from $A/K$.
\end{corollary}
\begin{remark}
    Let $V$ be an $A$-representation appearing in $\cU_A-\mathrm{Inf}^A_{A/K}\cU_{A/K}$, i.e. not inflated from $K$. Then this cellular model for $\widetilde{\E}\cF[K]$ allows us to reinterpret the equivalence $\Th(V)\otimes\widetilde{\E}\cF[K]\simeq \widetilde{\E}\cF[K]$ as a simple consequence of the colimit formula and the observation that Thom objects of vector bundles are symmetric objects.
\end{remark}
\newpage
\printbibliography
\end{document}